\documentclass[10pt]{amsart}

\usepackage{amsmath,amssymb,amsthm,amsfonts,mathrsfs}
\usepackage{fullpage}
\usepackage[shortlabels]{enumitem}
\usepackage[colorlinks=true,
linkcolor=blue,
anchorcolor=blue,
citecolor=red
]{hyperref}
\allowdisplaybreaks 
\newtheorem{theorem}{Theorem}[section]
\newtheorem{lemma}[theorem]{Lemma}

\newtheorem{proposition}[theorem]{Proposition}
\newtheorem{conjecture}[theorem]{Conjecture}
\theoremstyle{definition}

\theoremstyle{remark}
\newtheorem{remark}[theorem]{Remark}

\author{Pan Yan}
\address{Department of Mathematics, The University of Arizona, Tucson, AZ 85721, USA}
\email{panyan@arizona.edu}

\makeatletter
\@namedef{subjclassname@2020}{\textup{}2020 Mathematics Subject Classification}
\makeatother

\title{$L$-function for $\mathrm{Sp}(4)\times\mathrm{GL}(2)$ via a non-unique model}
\subjclass[2020]{Primary 11F70; Secondary 11F55, 22E50, 22E55}
\keywords{New Way method, non-unique model, generalized doubling integrals, periods}

\begin{document}

\begin{abstract}
In this paper we prove a conjecture of Ginzburg and Soudry on an integral representation for the $L$-function $L^S(s, \pi\times \tau)$ attached to a pair $(\pi, \tau)$ of irreducible automorphic cuspidal representations of $\mathrm{Sp}_4({\mathbb A})$ and $\mathrm{GL}_2({\mathbb A})$, which is derived from  the generalized doubling method of Cai, Friedberg, Ginzburg and Kaplan. We show that the integral unfolds to a non-unique model and analyze it using the New Way method of Piatetski-Shapiro and Rallis. Two applications are given. First, we relate the existence of the poles of $L^S(s,\pi\times\tau)$ to the non-vanishing of certain period integrals. Second, for certain family of cuspidal representations, we prove that $L^S(s, \pi\times \tau)$ is holomorphic.
\end{abstract}

\maketitle
\tableofcontents 

\section{Introduction}\label{section-introduction}

The theory of $L$-functions of automorphic forms or automorphic representations is a central topic in modern number theory  \cite{Langlands1970}. One fruitful way to study $L$-functions is through an integral formula, 
which is commonly referred to as an integral representation.
The usual Rankin-Selberg method \cite{Bump1989} involves a global integral which unfolds to a unique model such as the Whittaker model or the Bessel model, and it is the local uniqueness of the model which forces the global integral to be Eulerian;
see, for example, \cite{JacquetPiatetski-ShapiroShalika197901, JacquetPiatetski-ShapiroShalika197902, Furusawa1993, Piatetski-Shapiro1997, GinzburgRallisSoudry1998, File2013, Yan2024JRMS}.  
In 1987, Piatetski-Shapiro and Rallis \cite{Piatetski-ShapiroRallis1987Doubling} 
introduced a family of integrals which represent the standard $L$-functions for classical groups.
Their construction, now known as the doubling method, unfolded to an integral involving a global matrix coefficient on the classical group. More recently, the doubling construction has been vastly generalized by Cai, Friedberg, Ginzburg and Kaplan \cite{CaiFriedbergGinzburgKaplan2019} to represent the tensor product partial $L$-function $L^S(s, \pi\times \tau)$, where $\pi$ and $\tau$ are irreducible automorphic cuspidal representations on a classical group and a general linear group respectively. The new ingredients of the generalized doubling constructions in \cite{CaiFriedbergGinzburgKaplan2019}, are the use of the generalized Speh representations as inducing data for the Eisenstein series, and the use of a new model which generalizes the Whittaker model. The generalized doubling constructions, together with the Converse Theorem of Cogdell and Piatetski-Shapiro \cite{CogdellPiatetski-Shapiro1994, CogdellPiatetski-Shapiro1999}, can be applied to obtain global functoriality from $G({\mathbb A})$ to $\mathrm{GL}_N({\mathbb A})$, where $G$ is either a symplectic group or a split special orthogonal group or a split general spin group, and ${\mathbb A}$ is the ring of adeles of a number field $F$ (see \cite[Theorem 1]{CaiFriedbergKaplan2024}).

In \cite{GinzburgSoudry2020}, Ginzburg and Soudry proposed a relatively simple procedure to derive integrals from the (generalized) doubling method. 
Starting from the global (generalized) doubling integral, 
they used the process of global root exchange, and identities between Eisenstein series proved in \cite{GinzburgSoudry2021}, to derive and obtain a ``simpler'' integral. 
They gave several examples of the procedure, such as the Jacquet-Langlands \cite{JacquetLanglands1970} integral for the standard $L$-function attached to an irreducible automorphic cuspidal representation of $\mathrm{GL}_2(\mathbb{A})$, as well as the integrals of Piatetski-Shapiro and Rallis analyzed by the New Way method in \cite{Piatetski-ShapiroRallis1988}.

The New Way method was originally developed by Piatetski-Shapiro and Rallis in \cite{Piatetski-ShapiroRallis1988}, where they used non-unique models to obtain integral representations for the standard $L$-functions for $\mathrm{Sp}_{2n}$. We mention that several other constructions via a non-unique model for general linear groups and split orthogonal groups
were given in \cite{BumpFurusawaGinzburg1995}. The standard $L$-function for a quasi-split unitary group was treated in \cite{Qin2007}. 
Recently, Pollack and Shah adapted the New Way method to obtain the Spin $L$-functions for $\mathrm{GSp}_4$ \cite{PollackShah2017} and $\mathrm{GSp}_6$ \cite{PollackShah2018}. For the standard $L$-function for the exceptional group $G_2$ as well as its $\mathrm{GL}_1$ twist, we mention the works in \cite{GurevichSegal2015, Segal2017}. The New Way method also applies to the standard $L$-function for the $n$-fold metaplectic covering group $\mathrm{GL}^{(n)}_r(\mathbb{A})$ \cite{BumpFriedberg1999, Ginzburg2018}.

Ginzburg and Soudry showed how one can use the procedure they described in \cite{GinzburgSoudry2020} to construct a new global integral which unfolds to a non-unique model, and they conjectured that this new integral is Eulerian by the New Way method. 
To give more details about their construction, we introduce some notations. 
Let $\pi$ be an irreducible automorphic cuspidal representation of $\mathrm{Sp}_4(\mathbb{A})$, with $\varphi\in V_\pi$ a cusp form in the space of $\pi$.
Let $T_0$ denote a symmetric matrix in $\mathrm{GL}_2(F)$ and let
$T=J_2 T_0$ with $J_2=\left(\begin{smallmatrix} & 1\\
1& \end{smallmatrix}\right)$. Let $\chi_T$ be the quadratic character on $F^\times\backslash \mathbb{A}^\times$, given by $\chi_T(x)=(x, \det(T))$ where $(\cdot, \cdot)$ is the global Hilbert symbol. 
Consider the dual pair $(\mathrm{SO}_{T_0}, \mathrm{Sp}_4)$ inside $\mathrm{Sp}_8$.
The adelic group $\mathrm{SO}_{T_0}(\mathbb{A})\times \mathrm{Sp}_4(\mathbb{A})$ splits in $\widetilde{\mathrm{Sp}}_8(\mathbb{A})$, and we consider the restriction of the Weil representation $\omega_\psi=\omega_{\psi, 8}$ of 
$\widetilde{\mathrm{Sp}}_8(\mathbb{A})$,
corresponding to the character $\psi$, to the group 
$\mathrm{SO}_{T_0}(\mathbb{A})\times \mathrm{Sp}_4(\mathbb{A})$ 
under the splitting. We may realize the Weil representation in the Schwartz space $\mathcal{S}(\mathrm{Mat}_2(\mathbb{A}))$ (see \cite{Rallis1984, Weil1965}). For a Schwartz function $\Phi\in \mathcal{S}(\mathrm{Mat}_2(\mathbb{A}))$, let $\theta_{\psi}^\Phi $ be the corresponding theta series, defined in \eqref{Preliminaries-eq-theta-series-Sp8}. 
Let $N_{2,8}$ be the unipotent radical of the standard parabolic subgroup of $\mathrm{Sp}_8$ whose Levi subgroup is isomorphic to $\mathrm{GL}_2\times \mathrm{Sp}_4$. Note that $N_{2,8}$ has a structure of the Heisenberg group $\mathcal{H}_9$ in $9$ variables, under the map $\alpha_T:N_{2,8}\to \mathcal{H}_9$  which will be given in \eqref{eq-alpha_T}. 
Let $\tau$ be an irreducible unitary automorphic cuspidal representation of $\mathrm{GL}_2(\mathbb{A})$, and we denote by $\Delta(\tau\otimes \chi_T, 2)$ the generalized Speh representation of $\mathrm{GL}_4(\mathbb{A})$ attached to $\tau\otimes\chi_T$, defined by Jacquet \cite{Jacquet1984}. 
Given a section 
$$
f_{\Delta(\tau\otimes \chi_T, 2), s}\in \mathrm{Ind}_{P_8(\mathbb{A})}^{\mathrm{Sp}_8(\mathbb{A})}(\Delta(\tau\otimes \chi_T, 2)\vert \det \cdot \vert^s),
$$
where $P_8\subset \mathrm{Sp}_8$ is the Siegel parabolic subgroup, we form an Eisenstein series $E(\cdot, f_{\Delta(\tau\otimes \chi_T, 2), s})$. Here and thereafter, induction is normalized. The representation $\Delta(\tau\otimes \chi_T, 2)$ affords a unique model, denoted by $\mathcal{W}(\tau\otimes \chi_T, 2, \psi_{2T})$. This is the space of functions of the form $g\mapsto \Lambda(g\cdot \phi)$ as $\phi$ varies in the space of $\Delta(\tau\otimes \chi_T, 2)$, where $\Lambda(\phi)$ is given by the integral
\begin{equation*}
\Lambda(\phi)=\int_{U_{(2^2)}(F)\backslash U_{(2^2)}(\mathbb{A})} \phi(u)\psi_{2T}^{-1}(u)du,	
\end{equation*}
with $U_{(2^2)}=\left\{ \left(\begin{smallmatrix} I_2 & x \\ & I_2\end{smallmatrix}\right)\in \mathrm{GL}_4  \right\}$, and $\psi_{2T}\left(\left(\begin{smallmatrix} I_2 & x \\ & I_2\end{smallmatrix}\right)\right)=\psi\left( \mathrm{tr}(2Tx)\right)$.

Starting from the global generalized doubling integral 
in \cite{CaiFriedbergGinzburgKaplan2019}, Ginzburg and Soudry used global computations to derive the following integral
\begin{equation}
\begin{split}
\mathcal{Z}( \varphi,   \theta_{\psi}^\Phi, & E(\cdot, f_{\Delta(\tau\otimes \chi_T, 2), s})) =
\\ 
&\int\limits_{\mathrm{Sp}_4(F)\backslash \mathrm{Sp}_4(\mathbb{A})} \int\limits_{N_{2, 8}(F)\backslash N_{2, 8}(\mathbb{A})}  
  \varphi(h)     \theta_{\psi}^\Phi (\alpha_T(v)(1, h))E(v(1, h), f_{\Delta(\tau\otimes \chi_T, 2), s})dvdh.
\label{Introduction-eq-global-integral-0}
\end{split}
\end{equation}
Here, inside the Eisenstein series, the term $(1,h)$ stands for $\mathrm{diag}(I_2, h, I_2)$.
Moreover, they formulated a conjecture on the relation between the integral above and the tensor product partial $L$-function $L^S( s+\frac{1}{2}, \pi\times \tau)$, where $S$ is a finite set of places defined in Subsection~\ref{subsection-proof-of-main-result}. 

\begin{conjecture}(Ginzburg-Soudry, \cite[Section 4.2]{GinzburgSoudry2020})
The integral \eqref{Introduction-eq-global-integral-0} is Eulerian in the sense of the New Way method, and represents the partial $L$-function $L^S( s+\frac{1}{2}, \pi\times \tau)$.
\label{conjecture-Ginzburg-Soudry}
\end{conjecture}

\subsection{Statement of main result}

The main purpose of this paper is to prove Conjecture~\ref{conjecture-Ginzburg-Soudry}. Our result can be formulated in terms of normalized Eisenstein series. Let 
$
    d_{\tau\otimes\chi_T}^{\mathrm{Sp}_8, S}(s)=\prod_{\nu\not\in S} d_{\tau_\nu\otimes\chi_T}^{\mathrm{Sp}_8 }(s),
$
where
\begin{equation*}
    d_{\tau_\nu\otimes\chi_T}^{\mathrm{Sp}_8}(s)=L (s+\frac{3}{2}, \tau_\nu\otimes\chi_T)L (2s+2, \chi_{\tau_\nu}) L (2s+1, \tau_\nu\otimes\chi_T, \mathrm{Sym}^2).
\end{equation*}
Here, $\mathrm{Sym}^2$ is the symmetric square representation of $\mathrm{GL}_2({\mathbb C})$ and $\chi_{\tau_\nu}$ is the central character of $\tau_\nu$. This is the normalizing factor outside $S$ for the Eisenstein series $E(\cdot,  f_{\Delta(\tau\otimes \chi_T, 2), s})$ (see \cite[(1.48)]{GinzburgSoudry2021}). Denote
\begin{equation*}
    E^{*, S}(g, f_{\Delta(\tau\otimes \chi_T, 2), s})=    d_{\tau\otimes\chi_T}^{\mathrm{Sp}_8, S}(s) E(g, f_{\Delta(\tau\otimes \chi_T, 2), s}).
\end{equation*}
This is the partially normalized Eisenstein series. 

\begin{theorem}[Theorem~\ref{thm-conjecture-holds-restate}]
Conjecture~\ref{conjecture-Ginzburg-Soudry} holds. That is,
given an irreducible automorphic cuspidal representation $\pi$ of $\mathrm{Sp}_4({\mathbb A})$, 
an irreducible unitary automorphic cuspidal representation $\tau$ of $\mathrm{GL}_2({\mathbb A})$, 
and a non-zero cusp form $\varphi\in V_\pi$ which corresponds to $\otimes_\nu \xi_\nu$ under the factorization $\pi\cong \otimes_\nu^\prime \pi_\nu$, there is a choice of a matrix $T=J_2 T_0\in \mathrm{GL}_2(F)$ where $T_0$ is symmetric, a section $f_{\Delta(\tau\otimes \chi_T, 2), s}\in \mathrm{Ind}_{P_8(\mathbb{A})}^{\mathrm{Sp}_8(\mathbb{A})}(\Delta(\tau\otimes \chi_T, 2)\vert \det \cdot \vert^s)$ and some $\Phi=\prod_{\nu}\Phi_\nu \in \mathcal{S}(\mathrm{Mat}_2({\mathbb A}))$ so that
\begin{equation}
\mathcal{Z}( \varphi,   \theta_{\psi}^\Phi   , E^{*, S}(\cdot, f_{\Delta(\tau\otimes \chi_T, 2), s})) = 
L^S(s+\frac{1}{2}, \pi\times \tau)  \cdot \mathcal{Z}_S(\varphi, \Phi, f^*_{\mathcal{W}(\tau\otimes \chi_{T}, 2, \psi_{2T}), s}) ,
\label{Introduction-eq-in-major-theorem}
\end{equation}
where 
$\mathcal{Z}_S(\varphi, \Phi, f^*_{\mathcal{W}(\tau\otimes \chi_{T}, 2, \psi_{2T}), s})$ 
is a meromorphic function in $s$ defined by \eqref{eq-zeta_omega}.
Moreover, for any $s_0\in \mathbb{C}$, the data can be chosen so that $\mathcal{Z}_S(\varphi, \Phi, f^*_{\mathcal{W}(\tau\otimes \chi_{T}, 2, \psi_{2T}), s})$ is holomorphic and
non-zero in a neighborhood of $s=s_0$. 
\label{thm-conjecture-holds}
\end{theorem}

\begin{remark}
In contrast to all previously known constructions, the integral \eqref{Introduction-eq-global-integral-0} proposed by Ginzburg and Soudry is the first example of a New Way type integral representation for the $L$-function for $G\times {\mathrm{GL}}_k$ where $k>1$. Recently, the above construction has been generalized to ${\mathrm{Sp}}_{2n}\times {\mathrm{GL}}_k$ in \cite{JinYan2023}.
\end{remark}

\subsection{Method of proof}

Our proof of Theorem~\ref{thm-conjecture-holds} uses the New Way method developed by Piatetski-Shapiro and Rallis in \cite{Piatetski-ShapiroRallis1988}. First, we perform the global unfolding computation.

\begin{theorem}[\text{Theorem~\ref{theorem-global-unfolding}}]
When $\mathrm{Re}(s)\gg 0$, the integral  $\mathcal{Z}( \varphi,   \theta_{\psi}^\Phi   , E^{*, S}(\cdot, f_{\Delta(\tau\otimes \chi_T, 2), s}))$ unfolds to
\begin{equation*}
\begin{split}
 \int\limits_{ N_4(\mathbb{A})  \backslash  \mathrm{Sp}_4(\mathbb{A})}     \int\limits_{N_{2,8}^0(\mathbb{A})}      \varphi_{\psi, T}(h) \omega_\psi\left( \alpha_T(v)(1,h)\right) )\Phi(I_2 )     f^*_{\mathcal{W}(\tau\otimes \chi_T, 2, \psi_{2T}), s}(\gamma v(1, h) )      dv dh,
\end{split}
\end{equation*}
where $N_4$ is the unipotent radical of the standard Siegel parabolic subgroup of $\mathrm{Sp}_4$, $N_{2,8}^0$ is a certain subgroup of $N_{2,8}$, $\varphi_{\psi, T}$ is the Fourier coefficient given by
\begin{equation}
\label{Introduction-eq-Fourier-coefficient}
   \varphi_{\psi, T}(h):= \int\limits_{N_4(F)\backslash N_4({\mathbb A})}\varphi(nh)\psi_T(n)dn,
\end{equation}
with $\psi_T$ defined by  $\psi_T\left(\begin{smallmatrix} I_2 & z \\ & I_2\end{smallmatrix}\right)=\psi(\mathrm{tr}(Tz))$, $f^*_{\mathcal{W}(\tau\otimes \chi_T, 2, \psi_{2T}), s}$ is a section of the parabolic induction from $\mathcal{W}(\tau\otimes \chi_T, 2, \psi_{2T})$ obtained from $f_{\Delta(\tau\otimes \chi_T, 2), s}$, and $\gamma$ is a certain element of $\mathrm{Sp}_8(F)$.\label{Introduction-thm-global-unfolding}
\end{theorem}

We refer the reader to Section~\ref{section-main-theorems} for the precise definition of notations. We remark that, by a theorem of Li \cite{LiJian-Shu1992}, every automorphic cuspidal representation of $\mathrm{Sp}_4(\mathbb{A})$ possesses a non-zero Fourier coefficient $\varphi_{\psi, T}$ given in \eqref{Introduction-eq-Fourier-coefficient} for certain $T$ of maximal rank.

This naturally leads us to consider the local analogue of \eqref{Introduction-eq-Fourier-coefficient}. Let $\nu$ be a finite place of $F$ so that $(\pi_\nu, V_{\pi_{\nu}})$ is an irreducible admissible unramified representation of ${\mathrm{Sp}}_4(F_\nu)$. Let $l_T:V_{\pi_\nu}\to {\mathbb C}$ be a linear functional on $V_{\pi_\nu}$ satisfying 
\begin{equation}
\label{Introduction-eq-l_T}
l_{T}\left( {\pi_\nu} \left(\begin{smallmatrix}
I_2 &z\\
& I_2
\end{smallmatrix}\right) v \right)={\psi}_\nu^{-1}(\mathrm{tr}(Tz))l_{T}(v), \  \  \text{ for all } v\in V_{{\pi_\nu}}, z=J_2{}^tz J_2\in \mathrm{Mat}_2({F_\nu}).
\end{equation}
The space of these functionals is usually infinite-dimensional (see \cite{Piatetski-ShapiroRallis1988}). Hence, the model corresponding to $l_T$ is usually non-unique. This is the non-unique model alluded in the title of this paper.
Therefore, the usual methods of unramified computation, such as the one based on the Casselman-Shalika type formulas, will not work in our situation. Instead, we follow the method pioneered by Piatetski-Shapiro and Rallis, and carry out the local unramified computation that works for any linear functional $l_T$ satisfying \eqref{Introduction-eq-l_T}.

\begin{theorem}[\text{Theorem~\ref{theorem-unramified-computation}}]
Let $\nu\nmid 2, 3$ be a finite place of $F$, such that $\pi_\nu$ and $\tau_\nu$ are unramified. Let $v_0$ be a non-zero unramified vector in $V_{\pi_\nu}$ and let $l_{T}$ be a linear functional on $V_{\pi_\nu}$  satisfying \eqref{Introduction-eq-l_T}. Denote
\begin{equation*}
\begin{split}
    \mathcal{Z}^*_\nu(l_{T}, s) =  \int\limits_{N_4({F_\nu})\backslash \mathrm{Sp}_4({F_\nu})}   \int\limits_{N_{2,8}^0({F_\nu})}  l_{T}({\pi_\nu}(h)v_0)    \omega_{{\psi}}\left( \alpha_T(v)(1,h)\right) )\Phi^0_\nu(I_2 )
       f^*_{\mathcal{W}({\tau_\nu} \otimes\chi_{T}, 2, \psi_{2T}), s}( \gamma v (1, h) )    dv     dh.
\end{split}
\end{equation*}
The integral converges absolutely for $\mathrm{Re}(s)>x_0$, where $x_0>0$ and depends on $\pi_v$, $\tau_v$. 
For $\mathrm{Re}(s)$ sufficiently large,  we have
\begin{equation}
\begin{split}
 \mathcal{Z}^*_\nu(l_{T}, s)        = L(s+\frac{1}{2},  {\pi_\nu}\times {\tau_\nu}) \cdot l_{T}(v_0)  .
\label{Introduction-eq-local-unramified-identity}
\end{split}
\end{equation}
Here, $\Phi^0_\nu$ is the characteristic function of $\mathrm{Mat}_2(\mathcal{O}_{F_\nu})$ where $\mathcal{O}_{F_\nu}$ is the ring of integers of $F_\nu$, and $f^*_{\mathcal{W}({\tau} \otimes\chi_{T}, 2, \psi_{2T}), s}$ is the normalized unramified  section in $\mathrm{Ind}_{P_8({F})}^{\mathrm{Sp}_8({F})}(\mathcal{W}( {\tau_\nu}\otimes \chi_{T}, 2, \psi_{2T})\vert \det \vert^s)$.
\label{Introduction-thm-unramified}
\end{theorem}

\begin{remark}
We emphasize that the local unramified computation does not depend on the choice of the local model, and  \eqref{Introduction-eq-local-unramified-identity} holds for any linear functional $l_{T}$ on $V_{\pi_\nu}$ satisfying \eqref{Introduction-eq-l_T}.
\end{remark}

Finally,  we obtain \eqref{Introduction-eq-in-major-theorem} by an induction argument similar to \cite{Piatetski-ShapiroRallis1988}; see Theorem~\ref{theorem-global-1}.

\subsection{Applications}
Now we give two applications of Theorem~\ref{thm-conjecture-holds}.

The first application is on the relation between the existence of the poles of $L^S(s, \pi\times \tau)$ and the non-vanishing of certain period integrals. The location of the poles of $L^S(s, \pi\times \tau)$ is known (for example, from the integral representations of \cite{GinzburgJiangRallisSoudry2011} or \cite{CaiFriedbergGinzburgKaplan2019}). More precisely, in the right-half plane $\mathrm{Re}(s)>\frac{1}{2}$, $L^S(s,\pi\times \tau)$ is holomorphic for $\mathrm{Re}(s)>\frac{3}{2}$, and admits at most a simple pole at $s=\frac{3}{2}$ and at $s=1$. In \cite{JiangLiuZhang2013}, the location of possible poles of the fully normalized Eisenstein series $E^*(\cdot, f_{\Delta(\tau\otimes \chi_T, 2), s})$, at $\mathrm{Re}(s)\ge 0$, is determined. Here, 
\begin{equation}
\label{eq-normalized-Eisenstein}
    E^{*}(g, f_{\Delta(\tau\otimes \chi_T, 2), s})=    d_{\tau\otimes\chi_T}^{{\mathrm{Sp}}_8}(s) E(g, f_{\Delta(\tau\otimes \chi_T, 2), s}),
\end{equation}
where
\begin{equation*}
d_{\tau\otimes\chi_T}^{{\mathrm{Sp}}_8}(s)=\prod_{\nu} d_{\tau_\nu\otimes\chi_T}^{{\mathrm{Sp}}_8 }(s).
\end{equation*}
If $\tau$ is self-dual with trivial central character and $L(\frac{1}{2}, \tau\otimes\chi_T)\not=0$, 
then $E^*(\cdot , f_{\Delta(\tau\otimes \chi_T, 2), s})$ has a simple pole at $s=1$. Moreover, this pole is the right-most pole of $E^*(\cdot, f_{\Delta(\tau\otimes \chi_T, 2), s})$.
In this case, we denote by $\mathcal{R}(1, \Delta(\tau\otimes\chi_T, 2))$ the space generated by the residues of the fully normalized Eisenstein series $E^*(\cdot , f_{\Delta(\tau\otimes \chi_T, 2), s})$ at $s=1$, as the section $f_{\Delta(\tau\otimes \chi_T, 2), s}$ varies. An element $R\in \mathcal{R}(1, \Delta(\tau\otimes\chi_T, 2))$ is an automorphic form on $\mathrm{Sp}_8(\mathbb{A})$. Similarly, if $\tau$ is self-dual, and $L(s, \tau\otimes \chi_T, \mathrm{Sym}^2)$ has a pole at $s=1$, then $E^*(\cdot, f_{\Delta(\tau\otimes \chi_T, 2), s})$ has a simple pole at $\frac{1}{2}$. We denote by $\mathcal{R}(\frac{1}{2}, \Delta(\tau\otimes\chi_T, 2))$ the space generated by the residues of the fully normalized Eisenstein series $E^*(\cdot , f_{\Delta(\tau\otimes \chi_T, 2), s})$ at $s=\frac{1}{2}$, as the section $f_{\Delta(\tau\otimes \chi_T, 2), s}$ varies. We have the following theorem.

\begin{theorem} 
\label{thm-nonvanishing-of-period}
Let $\pi$ be an irreducible automorphic cuspidal representation of ${\mathrm{Sp}}_4({\mathbb{A}})$, and let $\tau$ be
an irreducible unitary automorphic cuspidal self-dual representation of ${\mathrm{GL}}_2({\mathbb{A}})$, 
Let $T=J_2 T_0$ with $T_0\in \mathrm{Sym}_2(F)\cap \mathrm{GL}_2(F)$, and assume that $\varphi\in V_\pi$ such that 
$\varphi_{\psi, T}\not=0$.
Then we have the following.
\begin{enumerate}

    \item[(i)] If 
$
    \mathrm{Res}_{s=\frac{3}{2}} L^S(s, \pi\times \tau)\not=0,
$
then there exist 
a Schwartz function $\Phi\in \mathcal{S}(\mathrm{Mat}_2({\mathbb{A}}))$, and a residue $R\in \mathcal{R}(1,\Delta(\tau\otimes \chi_T, 2))$, such that the period integral
\begin{equation*}
\int\limits_{\mathrm{Sp}_4(F)\backslash \mathrm{Sp}_4(\mathbb{A})} \int\limits_{N_{2, 8}(F)\backslash N_{2, 8}(\mathbb{A})}   \varphi(h)     \theta_{\psi}^\Phi (\alpha_T(v)(1, h)) R(v(1,h)) dvdh 
\end{equation*}
is non-zero.

    \item[(ii)] If 
$
    \mathrm{Res}_{s=1} L^S(s, \pi\times \tau)\not=0,
$
then there exist 
a Schwartz function $\Phi\in \mathcal{S}(\mathrm{Mat}_2({\mathbb{A}}))$, and a residue $R\in \mathcal{R}(\frac{1}{2},\Delta(\tau\otimes \chi_T, 2))$, such that the period integral
\begin{equation*}
\int\limits_{\mathrm{Sp}_4(F)\backslash \mathrm{Sp}_4(\mathbb{A})} \int\limits_{N_{2, 8}(F)\backslash N_{2, 8}(\mathbb{A})}   \varphi(h)     \theta_{\psi}^\Phi (\alpha_T(v)(1, h)) R(v(1,h)) dvdh 
\end{equation*}
is non-zero.
\end{enumerate}
\end{theorem}
 
As a second application, we have the following theorem on the holomorphy of $L^S(s,\pi\times\tau)$ for certain family of representations $\pi$ and $\tau$.

\begin{theorem}
\label{thm-holomorphy-of-L(s,pixtau)}
Let $\pi$ be an irreducible automorphic cuspidal representation of ${\mathrm{Sp}}_4({\mathbb{A}})$, and let $\tau$ be
an irreducible unitary automorphic cuspidal self-dual representation of ${\mathrm{GL}}_2({\mathbb{A}})$.
Let $T=J_2 T_0$ with $T_0\in \mathrm{Sym}_2(F)\cap \mathrm{GL}_2(F)$, and assume that $\varphi\in V_\pi$ such that
$\varphi_{\psi, T}\not=0$. 
Assume that  $L(\frac{1}{2}, \tau\otimes\chi_T)=0$ and the central character $\chi_\tau$ of $\tau$ is trivial. Then $L^S(s, \pi\times \tau)$ is holomorphic.
\end{theorem}

\subsection{Structure of this paper}
The rest of this paper is organized as follows. In Section~\ref{section-Preliminaries}, we define the groups, and recall relevant theta seires and Eisenstein series, as well as the generalized Speh representations. In Section~\ref{section-main-theorems}, we discuss the global integral of Ginzburg and Soudry and prove Theorem~\ref{Introduction-thm-global-unfolding}.  In Section~\ref{section-Unramified computation}, we compute the local integral at an unramified place and prove Theorem~\ref{Introduction-thm-unramified}. In Section~\ref{section-Ramified computation}, we discuss the non-vanishing of the local integrals at finite ramified places and the archimedean places.  Finally, in Section~\ref{section-proofs-main-results}, we discuss the general technique of Piatetski-Shapiro and Rallis to analyze a non-unique model, and prove the main theorems.

\section{Preliminaries} \label{section-Preliminaries}

\subsection{General notations}
We start with some general notations. Let $F$ be a number field and $\mathbb{A}=\mathbb{A}_F$ the ring of adeles of $F$. Throughout the paper, all algebraic groups are defined over $F$.
For a positive integer $n$, we let $J_n$ denote the $n\times n$ matrix which has ones on the antidiagonal and zeros everywhere else. Let $\mathrm{Sp}_{2n}$ be the symplectic group of rank $n$, regarded as an algebraic group over $F$. Over an $F$-algebra $R$, it is realized in matrices as
\begin{equation*}
\mathrm{Sp}_{2n}(R)=\left\{g \in \mathrm{GL}_{2 n}(R) : {}^t g \left(\begin{smallmatrix}
 & J_n\\
 -J_n &
\end{smallmatrix}\right) g=\left(\begin{smallmatrix}
 & J_n\\
 -J_n &
\end{smallmatrix}\right)\right\}.
\end{equation*} 
For each $0\le r \le n$, we let $P_{r,2n}$ be the standard parabolic subgroup of $\mathrm{Sp}_{2n}$ with the standard Levi decomposition $P_{r,2n}=M_{r,2n}\ltimes N_{r,2n}$, where the Levi part is $M_{r,2n}\cong \mathrm{GL}_r\times \mathrm{Sp}_{2n-2r}$. When $r=0$, it is understood that $M_{0,2n}=\mathrm{Sp}_{2n}$ and $N_{0,2n}=\{I_{2n}\}$.
Let $P_{2n}=P_{n,2n}$,  $M_{2n}=M_{n,2n}$, and $N_{2n}=N_{n,2n}$. 
Then $P_{2n}$ is the standard Siegel parabolic subgroup of $\mathrm{Sp}_{2n}$.
Let $\mathrm{Mat}_{m\times n}$ denote the additive group of all matrices of size $m\times n$ and let $\mathrm{Mat}_n=\mathrm{Mat}_{n\times n}$. Over an $F$-algebra $R$, denote
\begin{equation*}
\mathrm{Mat}_{n}^{0}(R)=\left\{A \in \mathrm{Mat}_{n}(R) : {}^t A J_{n}=J_{n} A\right\}.
\end{equation*}
Then  
\begin{equation*}
M_{2n}(R)=\left\{ \hat{g}=\left(\begin{smallmatrix}
 g & \\
  & J_n {}^t g^{-1} J_n
\end{smallmatrix}\right) :g\in \mathrm{GL}_{n}(R)\right\}, \,  N_{2n}(R)=\left\{ \left(\begin{smallmatrix}
 I_n & z \\
  & I_n
\end{smallmatrix}\right) :z\in \mathrm{Mat}_n^0 (R) \right\}.
\end{equation*}

\subsection{The Weil representation}
We fix a nontrivial additive character $\psi:F\backslash \mathbb{A}\to \mathbb{C}^\times$. Let $T_0\in \mathrm{GL}_2(F)$ be a symmetric matrix. To simplify our calculation we may assume that $T_0=\mathrm{diag}(t_1, t_2)$ is diagonal, but this assumption is not essential (see \cite[Section 4.1]{GinzburgSoudry2020}). Let $T=J_2 T_0=\left(\begin{smallmatrix}
0& t_2\\
t_1 &0
\end{smallmatrix}\right)$. We denote by $\chi_T$ the quadratic character of $F^\times\backslash \mathbb{A}^\times$ given by 
$
    \chi_T(x)=(x, \det(T)),
$
where $(\cdot, \cdot)$ is the global Hilbert symbol. This is the quadratic character corresponding to $\mathrm{disc}(T_0)$.
Let $\mathrm{SO}_{T_0}$ be the special orthogonal group defined over $F$, corresponding to $T_0$. The matrix realization of $\mathrm{SO}_{T_0}$ is given by
$
    \mathrm{SO}_{T_0}=\left\{g\in {\mathrm{SL}}_2: {}^t g T_0 g =T_0\right\}.
$

We consider the following Heisenberg group $\mathcal{H}_9$ in $9$ variables, realized as $\mathrm{Mat}_2\times \mathrm{Mat}_2\times \mathrm{Mat}_1$. This is the group of
elements $(X, Y, z)$ where $X, Y\in \mathrm{Mat}_2$, $z\in \mathrm{Mat}_1$, with multiplication given by
\begin{equation*}
    (X_1, Y_1, z_1)(X_2, Y_2, z_2)=(X_1+X_2, Y_1+Y_2, z_1+z_2+\mathrm{tr}(T(X_1J_2{}^t Y_2 J_2- Y_1J_2{}^t X_2 J_2))).
\end{equation*}
We also consider the unipotent group $N_{2,8}$,  which consists of matrices of the form
\begin{equation}
    v(x, y, z)=\left(\begin{smallmatrix}
I_2 & x & y & z\\
& I_2 & & y^* \\
& &I_2 & x^*\\
& & &I_2
\end{smallmatrix}\right)
\label{Preliminaries-eq-N_2,8}
\end{equation}
where $y^*=J_2 {}^t y J_2, x^*=-J_2 {}^t x J_2$, and $z-xy^*\in \mathrm{Mat}_2^0$. Then $N_{2,8}$ is isomorphic to the Heisenberg group $\mathcal{H}_9$ by the map 
\begin{equation}
    \alpha_T \left(\begin{smallmatrix}
I_2 & x & y & z\\
& I_2 & & y^* \\
& &I_2 & x^*\\
& & &I_2
\end{smallmatrix}\right)=(x, y, \mathrm{tr}(Tz)).
\label{eq-alpha_T}
\end{equation}
We embed $\mathrm{SO}_{T_0}\times \mathrm{Sp}_4$ inside $\mathrm{Sp}_8$ via $(m,h)\mapsto \mathrm{diag}(m, h, m^*)$ where $m^*=J_2{}^t m^{-1} J_2$. 
To simplify our notation, we re-denote $(m, h)=\mathrm{diag}(m, h, m^*)$.

Let $\widetilde{\mathrm{Sp}}_8(\mathbb{A})$ be the metaplectic double cover of $\mathrm{Sp}_8(\mathbb{A})$.
We consider the dual pair $(\mathrm{SO}_{T_0}, \mathrm{Sp}_4)$ inside $\mathrm{Sp}_8$.
The adelic group $\mathrm{SO}_{T_0}(\mathbb{A})\times \mathrm{Sp}_4(\mathbb{A})$ splits in $\widetilde{\mathrm{Sp}}_8(\mathbb{A})$, and we consider the restriction of the Weil representation $\omega_\psi=\omega_{\psi, 8}$ of 
$\widetilde{\mathrm{Sp}}_8(\mathbb{A})$,
corresponding to the character $\psi$, to the group 
$\mathrm{SO}_{T_0}(\mathbb{A})\times \mathrm{Sp}_4(\mathbb{A})$ 
under the splitting.  We realize the Weil representation in the Schwartz space $\mathcal{S}(\mathrm{Mat}_2(\mathbb{A}))$.  We have the following formulas (see for example \cite[pp. 8]{GinzburgRallisSoudry2011}) for the Weil representations
\begin{equation}
\omega_\psi \left( \alpha_T(v(x, y, z)) \right) \Phi(\xi)= \psi(\mathrm{tr}(Tz)) \psi(\mathrm{tr}(TxJ_2 {}^t yJ_2)  ) \psi(\mathrm{tr}(2T\xi J_2 {}^t y J_2 ))\Phi(\xi+x),
\label{Preliminaries-eq-Weil-formula1}
\end{equation}
\begin{equation}
\omega_\psi   \left(1,  \left(\begin{smallmatrix}
a & \\
& a^*
\end{smallmatrix} \right) \right)  \Phi(\xi) =   \chi_T(\det(a)) \vert \det a \vert \Phi(\xi a), 
\label{Preliminaries-eq-Weil-formula2}
\end{equation}
\begin{equation}
\omega_\psi \left(       1,  \left(\begin{smallmatrix}
I_2 & w\\
& I_2
\end{smallmatrix}\right)  \right)   \Phi(\xi)=   \psi (   \mathrm{tr}( T   {}^t \xi w  \xi)     )\Phi(\xi)  .
\end{equation}
Here, $\xi\in \mathrm{Mat}_2(\mathbb{A})$, $\Phi\in \mathcal{S}(\mathrm{Mat}_2(\mathbb{A}))$, $v(x, y, z)\in N_{2,8}(\mathbb{A})$, $\left(\begin{smallmatrix}
a & \\
& a^*
\end{smallmatrix} \right)\in M_4(\mathbb{A})$, $\left(\begin{smallmatrix}
I_2 & w\\
& I_2
\end{smallmatrix}\right)\in N_4(\mathbb{A})$.

For $\Phi\in \mathcal{S}(\mathrm{Mat}_{2}(\mathbb{A}))$, we form the theta series 
\begin{equation}
\label{Preliminaries-eq-theta-series-Sp8} 
\theta_{\psi}^{\Phi}(\alpha_T(v) (m,h))=\theta_{\psi, 8}^{\Phi}(\alpha_T(v) (m,h))=\sum_{\xi\in \mathrm{Mat}_{2}(F)}\omega_\psi(\alpha_T(v) (m,h))\Phi(\xi), 
\end{equation}
where $v\in N_{2,8}(\mathbb{A}), m\in \mathrm{SO}_{T_0}(\mathbb{A}), h \in \mathrm{Sp}_4({\mathbb A})$.
It is known that the above series converges absolutely (see \cite[Proposition 5.3.1]{KudlaRallis1994}), and it is of moderate growth in each variable of $m$ and $h$ (see \cite[Lemma 2.1]{HowePiatetski-Shapiro1983}).

\subsection{Generalized Speh representations}

Let $\tau$ be an irreducible unitary automorphic cuspidal represntation of $\mathrm{GL}_2(\mathbb{A})$.  Let $c$ be a positive integer. We consider the unipotent subgroup $U_{(c^2)}$ of $\mathrm{GL}_{2c}$ of the type $(c, c)$, consisting of matrices of the form
\begin{equation*}
    u=\left(\begin{smallmatrix}
    I_c & u_{1,2}\\ & I_c\end{smallmatrix}\right), \, u_{1,2}\in \mathrm{Mat}_c.
\end{equation*}
We extend the character $\psi$ of $F\backslash \mathbb{A}$ to a character of $U_{(c^2)}(F)\backslash U_{(c^2)}(\mathbb{A})$ by
\begin{equation*}
    \psi(u)=\psi(\mathrm{tr}(u_{1,2})).
\end{equation*}
For an automorphic function $\varphi$ on $\mathrm{GL}_{2c}(F)\backslash \mathrm{GL}_{2c}(\mathbb{A})$, we consider the following Fourier coefficient
\begin{equation*}
    \Lambda(\varphi)=\int\limits_{U_{(c^2)}(F)\backslash U_{(c^2)}(\mathbb{A})}\varphi(u)\psi^{-1}(u)du.
\end{equation*}
This is a Fourier coefficient of Shalika type. More generally, it is called a
 Whittaker-Speh-Shalika coefficient in \cite{CaiFriedbergGinzburgKaplan2019}.

Let $\Delta(\tau, c)$ be the generalized Speh representation of $\mathrm{GL}_{2c}(\mathbb{A})$ defined by Jacquet \cite{Jacquet1984}. 
The Fourier coefficient $\Lambda(\varphi)$ does not vanish identically on the space of $\Delta(\tau, c)$ (see \cite[Theorem 13]{CaiFriedbergGinzburgKaplan2019}).
Let $\mathcal{W}(\tau, c, \psi)$ be the space of functions on $\mathrm{GL}_{2c}(\mathbb{A})$
\begin{equation*}
    g\mapsto \Lambda( \Delta(\tau, c)(g) \varphi),
\end{equation*}
as $\varphi$ varies in the space of $\Delta(\tau, c)$. This is a unique model that is afforded by $\Delta(\tau, c)$ (see \cite[Theorem 13]{CaiFriedbergGinzburgKaplan2019} for the global result and \cite{CaiFriedbergGourevitchKaplan2023} for the local result in a more general setting).

\subsection{A Siegel Eisenstein series}
Recall that $\chi_T$ is the  quadratic character on $F^\times \backslash \mathbb{A}^\times$ given by $\chi_T(x)=(x, \det(T))$. 
Given an irreducible unitary cuspidal automorphic representation $\tau$ of $\text{GL}_2(\mathbb{A})$, we form the generalized Speh representation $\Delta(\tau\otimes \chi_T, 2)$ of $\text{GL}_4(\mathbb{A})$ associated to $\tau\otimes\chi_T$,
and 
consider the space
\begin{equation*}
    \text{Ind}_{P_8(\mathbb{A})}^{\text{Sp}_8(\mathbb{A})}(\Delta(\tau\otimes\chi_T, 2)|\det|^s).
\end{equation*}
We realize elements of the above space as complex-valued functions, i.e., this is the space of all smooth functions $f_{\Delta(\tau\otimes\chi_T, 2), s}:\text{Sp}_8(\mathbb{A})\to \mathbb{C}$ such that for all $p=\left(\begin{smallmatrix} a & \\ & a^*\end{smallmatrix}\right)\left(\begin{smallmatrix} I_4 & z \\ & I_4\end{smallmatrix}\right)\in P_8(\mathbb{A})$ and $g\in \text{Sp}_8(\mathbb{A})$, we have
\begin{equation*}
    f_{\Delta(\tau\otimes\chi_T, 2), s}(pg)=|\det(a)|^{s+\frac{5}{2}} \Delta(\tau\otimes \chi_T, 2)(a) f_{\Delta(\tau\otimes\chi_T, 2), s}(g),
\end{equation*}
and for each $g\in \mathrm{Sp}_8(\mathbb{A})$, the function $a\mapsto |\det(a)|^{-s-\frac{5}{2}} f_{\Delta(\tau\otimes\chi_T, 2), s}\left( \left( \begin{smallmatrix} a&\\ &a^*\end{smallmatrix}\right) g \right)$ lies in the space of $\Delta(\tau\otimes\chi_T, 2)$.
For a smooth holomorphic section $f_{\Delta(\tau\otimes\chi_T, 2), s}\in \text{Ind}_{P_8(\mathbb{A})}^{\text{Sp}_8(\mathbb{A})}(\Delta(\tau\otimes\chi_T, 2)|\det|^s)$, we define an Eisenstein series by
\begin{equation*}
    E(g, f_{\Delta(\tau\otimes\chi_T, 2), s})=\sum_{\delta\in P_8(F)\backslash \text{Sp}_8(F)}f_{\Delta(\tau\otimes\chi_T, 2), s}(\delta g).
\end{equation*}
It's known (see \cite{Langlands1976}) that the above series converges absolutely for $\text{Re}(s)\gg 0$ and admits a meromorphic continuation to the entire complex plane.

We multiply the above Eisenstein series by its normalizing factor outside $S$, to get the (partially) normalized Eisenstein series. The normalizing factor outside $S$ is (see \cite[(1.48)]{GinzburgSoudry2021})
\begin{equation*}
\begin{split}
    d_{\tau\otimes\chi_T}^{\mathrm{Sp}_8, S}(s)=\prod_{\nu\not\in S} d_{\tau_\nu\otimes\chi_T}^{\mathrm{Sp}_8 }(s),
\end{split}
\end{equation*}
where
\begin{equation*}
    d_{\tau_\nu\otimes\chi_T}^{\mathrm{Sp}_8}(s)=L (s+\frac{3}{2}, \tau_\nu\otimes\chi_T)L (2s+2, \chi_{\tau_\nu}) L (2s+1, \tau_\nu\otimes\chi_T, \mathrm{Sym}^2).
\end{equation*}
We set
\begin{equation*}
    E^{*, S}(g, f_{\Delta(\tau\otimes \chi_T, 2), s})=    d_{\tau\otimes\chi_T}^{\mathrm{Sp}_8, S}(s) E(g, f_{\Delta(\tau\otimes \chi_T, 2), s}).
\end{equation*}

\subsection{The $L$-function for $\mathrm{Sp}_4\times \mathrm{GL}_2$}

In this subsection we describe the $L$-functions we study in this paper. 
Let $\nu$ be a finite place of $F$, and let $F_\nu$ be its completion with respect to $\nu$ with $q_\nu$ the cardinality of the residue field. Let $\pi_\nu$ be an irreducible unramified representation of $\mathrm{Sp}_4(F_\nu)$ and let $\tau_\nu$ be an irreducible unramified representation of $\mathrm{GL}_2(F_\nu)$ with central character $\chi_{\tau_\nu}$. Note that the $L$-group of $\mathrm{Sp}_4$ is ${}^L \mathrm{Sp}_4=\mathrm{SO}_5(\mathbb{C})$ and the $L$-group of $\mathrm{GL}_2$ is ${}^L \mathrm{GL}_2=\mathrm{GL}_2(\mathbb{C})$. Let $t_{\pi_\nu}$ be the semisimple conjugacy class of $\mathrm{SO}_5({\mathbb C})$ attached to $\pi_\nu$, given by
$
    t_{\pi_\nu}=\mathrm{diag}(\alpha_{1,\nu}, \alpha_{2, \nu}, 1, \alpha_{2, \nu}^{-1}, \alpha_{1,\nu}^{-1}).
$
Similarly, let $t_{\tau_\nu}$ be the semisimple conjugacy class of $\mathrm{GL}_2({\mathbb C})$ attached to $\tau_\nu$, given by
$
    t_{\tau_\nu}=\mathrm{diag}(\beta_{1, \nu}, \beta_{2, \nu}).
$
Then the local tensor product $L$-function for $\pi_\nu$ and $\tau_\nu$ is defined by
$
L(s, \pi_\nu \times \tau_\nu)=\det(1 - t_{\pi_\nu} \otimes t_{\tau_\nu} q_\nu^{-s})^{-1}.
$
The local standard $L$-function for $\tau_\nu$ is
$
L(s, \tau_\nu)=(1-\beta_{1, \nu} q_\nu^{-s} )^{-1}   (1-\beta_{2, \nu} q_\nu^{-s} )^{-1}.
$
The local symmetric square and exterior square $L$-functions for $\tau_\nu$ are 
\begin{equation*}
\begin{split}
L(s, \tau_\nu, \mathrm{Sym}^2) &=(1-\beta_{1,\nu}^2 q_\nu^{-s})^{-1} (1-\beta_{1,\nu} \beta_{2, \nu} q_\nu^{-s})^{-1} (1-\beta_{2, \nu}^2 q_\nu^{-s})^{-1}, \\
L(s, \tau_\nu, \mathrm{Ext}^2) &=L(s, \chi_{\tau_\nu})= (1-\beta_{1,\nu} \beta_{2, \nu} q_\nu^{-s})^{-1} .
\end{split}
\end{equation*}
It follows that for a quadratic character $\chi$ on $F_\nu^\times$, we have the identities
\begin{equation*}
\begin{split}
    L(s, \tau_\nu\otimes\chi, \mathrm{Sym}^2)=L(s, \tau_\nu, \mathrm{Sym}^2), \,L(s, \tau_\nu\otimes\chi, \mathrm{Ext}^2)=L(s, \tau_\nu, \mathrm{Ext}^2).
\end{split}
\end{equation*}

For a pair of irreducible automorphic cuspidal representations $\pi\cong \otimes^\prime_\nu \pi_\nu$ and $\tau\cong \otimes_\nu^\prime \tau_\nu$, of $\mathrm{Sp}_4(\mathbb{A})$ and $\mathrm{GL}_2(\mathbb{A})$ respectively, where the local represetations $\pi_\nu$ and $\tau_\nu$ are unramified outside a finite set $S$ of places, the tensor product partial $L$-function   is defined by
$$
L^S(s, \pi\times \tau)=\prod_{\nu\not\in S} L(s, \pi_\nu\times \tau_\nu).
$$
Similarly, the partial standard $L$-function, the partial symmetric square $L$-function, and the partial exterior square $L$-function for $\tau$ are given by
\begin{equation*}
\begin{split}
L^S(s, \tau) &= \prod_{\nu\not\in S} L(s, \tau_\nu),\\
L^S(s, \tau, \mathrm{Sym}^2) &=\prod_{\nu\not\in S} L(s, \tau_\nu, \mathrm{Sym}^2), \\
L^S(s, \tau, \mathrm{Ext}^2)  &=\prod_{\nu\not\in S} L(s, \tau_\nu, \mathrm{Ext}^2)=L^S(s,\chi_{\tau})=\prod_{\nu\not\in S} L(s, \chi_{\tau_\nu}) .
\end{split}
\end{equation*}

\section{The global integral and the unfolding computation}
\label{section-main-theorems}
In this section we discuss the global integral and perform the unfolding computation.

\subsection{Definition of the global integral}
\label{subsection-the-global-integral}
Let $F$ be a number field and $\mathbb{A}=\mathbb{A}_F$ the ring of adeles of $F$. We fix a nontrivial additive character $\psi$ on $F\backslash \mathbb{A}$. Let $(\pi, V_\pi)$ be an irreducible automorphic cuspidal representation of $\mathrm{Sp}_4(\mathbb{A})$ and let $(\tau, V_\tau)$  be an irreducible unitary automorphic cuspidal representation of $\mathrm{GL}_2(\mathbb{A})$.

For a cusp form $\varphi\in V_\pi$, we denote the Fourier coefficient of $\varphi$ along  the unipotent radical of the Siegel parabolic subgroup by
\begin{equation*}
\varphi_{\psi, T}(h):= \int\limits_{\mathrm{Mat}_2^0(F)\backslash \mathrm{Mat}_2^0(\mathbb{A})} \varphi \left(  \left(\begin{smallmatrix}
I_2 & z\\
& I_2
\end{smallmatrix}\right) h \right) \psi(\mathrm{tr}(Tz))dz,
\end{equation*}
where $T=J_2 T_0$, and $T_0\in \mathrm{GL}_2(F)$ is a symmetric matrix.
We remark that it follows from a theorem of Li \cite{LiJian-Shu1992} that every non-zero cusp form affords a non-zero Fourier coefficient $\varphi_{\psi, T}$ where $T$ has rank two.
Let $T$ be such that $\varphi_{\psi, T}\not= 0$. 
The space of such Fourier coefficients is denoted by $\mathcal{W}(\pi, \psi, T)$. 
To simplify our calculation we may assume that $T_0=\mathrm{diag}(t_1, t_2)$ is diagonal.

We consider the integral
\begin{equation}
\begin{split}
\mathcal{Z}( \varphi,   \theta_{\psi}^\Phi   , E^{*, S}(\cdot,  & f_{\Delta(\tau\otimes \chi_T, 2), s})) = \int\limits_{\mathrm{Sp}_4(F)\backslash \mathrm{Sp}_4(\mathbb{A})} \int\limits_{N_{2, 8}(F)\backslash N_{2, 8}(\mathbb{A})} \\
&
 \varphi(h)     \theta_{\psi}^\Phi (\alpha_T(v)(1, h))E^{*, S}(v(1, h), f_{\Delta(\tau\otimes \chi_T, 2), s})dvdh.
\label{eq-global-integral-0}
\end{split}
\end{equation}
Here, $\varphi\in V_\pi$ is a cusp form in the space of $\pi$. The function $\theta_{\psi}^\Phi $ is the theta series on $\mathcal{H}_9(\mathbb{A}) \rtimes \mathrm{Sp}_8(\mathbb{A})$ associated with a Schwartz function $\Phi\in \mathcal{S}(\mathrm{Mat}_2(\mathbb{A}))$. The function $f_{\Delta(\tau\otimes \chi_T, 2)}$ is a smooth, holomorphic section in the induced representation $\mathrm{Ind}_{P_{8}(\mathbb{A})}^{\mathrm{Sp}_{8}(\mathbb{A})}(\Delta(\tau\otimes\chi_T, 2)\vert \det \vert^s)$ and $E^{*, S}(\cdot, f_{\Delta(\tau\otimes \chi_T, 2), s})$ is the corresponding normalized Eisenstein series.

\subsection{The unfolding computation}
\label{section-Unfolding}

Let
\begin{equation*}
\begin{split}
\gamma = \left(\begin{smallmatrix}
&I_2 & &\\&&&-I_2\\I_2&&&\\&&I_2&
\end{smallmatrix}\right),
\end{split}
\end{equation*}
and let
\begin{equation*}
    N_{2,8}^0=\left\{ v(x, 0, z)=\left(\begin{smallmatrix} I_2 & x & 0& z\\ &I_2 &0 &0 \\ &&I_2 & x^*\\ &&&I_2\end{smallmatrix}\right) \in N_{2,8} \right\}.
\end{equation*}

In the following theorem we state the basic properties of the global integral. 

\begin{theorem}
The integral $\mathcal{Z}( \varphi,   \theta_{\psi}^\Phi   , E^{*, S}(\cdot, f_{\Delta(\tau\otimes \chi_T, 2), s}))$ converges absolutely for all $s\in \mathbb{C}$, away from the poles of the Eisenstein series.  There exists a $x_0>0$ which depends only on $\pi$ and $\tau$ such that the integral converges absolutely when $\mathrm{Re}(s)>x_0$, for all $\varphi, \Phi,$ and $f_{\Delta(\tau\otimes \chi_T, 2), s}$. When $\mathrm{Re}(s)$ sufficiently large, the integral  $\mathcal{Z}( \varphi,   \theta_{\psi}^\Phi   , E^{*, S}(\cdot, f_{\Delta(\tau\otimes \chi_T, 2), s}))$ unfolds to
\begin{equation}
\begin{split}
 \int\limits_{ N_4(\mathbb{A})  \backslash  \mathrm{Sp}_4(\mathbb{A})}     \int\limits_{N_{2,8}^0(\mathbb{A})}      \varphi_{\psi, T}(h) \omega_\psi\left( \alpha_T(v)(1,h)\right) )\Phi(I_2 )     f^*_{\mathcal{W}(\tau\otimes \chi_T, 2, \psi_{2T}), s}(\gamma v(1, h) )      dv dh.
\end{split}
\label{eq-thm-global-integral-unfolding-final}
\end{equation}
Here, 
$
    f^*_{\mathcal{W}(\tau\otimes \chi_T, 2, \psi_{2T}), s}(g)  = d_{\tau\otimes\chi_T}^{\mathrm{Sp}_8, S}(s) f_{\mathcal{W}(\tau\otimes \chi_T, 2, \psi_{2T}), s}(g )  
$
where $f_{\mathcal{W}(\tau\otimes \chi_T, 2, \psi_{2T}), s}$ is the composition of the section and the unique functional attached to $\Delta(\tau\otimes\chi_T, 2)$; explicitly, for any $g\in \mathrm{Sp}_8(\mathbb{A})$, 
\begin{equation*}
\begin{split}
f_{\mathcal{W}(\tau\otimes \chi_T, 2, \psi_{2T}), s}(g)=
\int\limits_{U_{(2^2)}(F)\backslash U_{(2^2)}(\mathbb{A}) } f_{\Delta(\tau\otimes \chi_T, 2), s} \left(     \left(\begin{smallmatrix} u & \\ & u^*\end{smallmatrix}\right) g \right) \psi_{2T}^{-1}(u) du,
\end{split}
\end{equation*}
where $U_{(2^2)}=\left\{ \left(\begin{smallmatrix} I_2 & x \\ & I_2\end{smallmatrix}\right)\in \mathrm{GL}_4  \right\}$, and $\psi_{2T}\left(\begin{smallmatrix} I_2 & x \\ & I_2\end{smallmatrix}\right)=\psi\left( \mathrm{tr}(2Tx)\right)$.
\label{theorem-global-unfolding}
\end{theorem}

The absolute convergence of the integral $\mathcal{Z}( \varphi,   \theta_{\psi}^\Phi   , E^{*, S}(\cdot, f_{\Delta(\tau\otimes \chi_T, 2), s}))$ follows due to the rapid decrease of $\varphi$ on $\mathrm{Sp}_4(F)\backslash \mathrm{Sp}_4(\mathbb{A})$, the moderate growth of the Eisenstein series and the theta series, and the compactness of $N_{2, 8}(F)\backslash N_{2, 8}(\mathbb{A})$.
The proof of the unfolding computation occupies the rest of this section. It's convenient to work with the integral $\mathcal{Z}( \varphi,   \theta_{\psi}^\Phi   ,  E(\cdot, f_{\Delta(\tau\otimes \chi_T, 2), s}))$ where the Eisenstein series is not normalized.

We begin with a parametrization of the representatives of  the $P_{2,8}(F)$-orbits in $P_8(F)\backslash \mathrm{Sp}_8(F)$. For each orbit $P_8(F)\gamma_i$, we denote $H_{\gamma_i}(F)$ its stabilizer in $P_{2,8}(F)$. There are three $P_{2,8}(F)$-orbits by \cite[Section 4]{GinzburgRallisSoudry1998}. The representatives of the orbits and their stabilizers are given below.
\begin{lemma}
The following is a list of representatives of the $P_{2,8}(F)$-orbits in $P_8(F)\backslash \mathrm{Sp}_8(F)$ and their stabilizers: \\
\begin{itemize}
\item[(i)]  $\gamma_0=I_8$, and the stabilizer $H_{\gamma_0}(F)=M_{\gamma_0}(F)\ltimes N_{\gamma_0}(F)$, where
\begin{equation*}
\begin{split}
M_{\gamma_0}(F)          =\left\{ \left(\begin{smallmatrix}
g_0 &&\\
&h_0 &\\
&&J_2 {}^{t}g_0^{-1}J_2
\end{smallmatrix}\right) : g_0\in \mathrm{GL}_2(F), h_0\in P_4(F)\right\}, \, 
 N_{\gamma_0}(F)          =N_{2,8}(F).
 \end{split}
\end{equation*}

\item[(ii)]    $ \gamma_1=  \left(\begin{smallmatrix}
& & &-1\\
& &I_{6} &\\
&1  &&
\end{smallmatrix}\right)$, and the stabilizer  $H_{\gamma_1}(F)=M_{\gamma_1}(F)\ltimes N_{\gamma_1}(F)$, where
\begin{equation*}
\begin{split}
M_{\gamma_1}(F)          &=\left\{ \left(\begin{smallmatrix}
g_0 &&\\
&h_0 &\\
&&J_2 {}^{t}g_0^{-1}J_2
\end{smallmatrix}\right) : g_0 \text{ in lower triangular Borel } B_{\mathrm{GL}_2}^{-}(F), h_0\in P_4(F)\right\}, \\ 
 N_{\gamma_1}(F)            &=\left\{ v(x, y, z)\in N_{2,8}(F): x=\left(\begin{smallmatrix}
0 & 0\\
* & *
\end{smallmatrix}\right) , z=\left(\begin{smallmatrix}
* & 0\\
* &*
\end{smallmatrix}\right) \right\}.
\end{split}
\end{equation*}

\item[(iii)]   $ \gamma_2=\left(\begin{smallmatrix}
& & &-I_2\\
& &I_{4} &\\
&I_2  &&
 \end{smallmatrix}\right)$, and the stabilizer $H_{\gamma_2}(F)=M_{\gamma_2}(F)\ltimes N_{\gamma_2}(F)$, where
 \begin{equation*}
 \begin{split}
 M_{\gamma_2}(F)         &=\left\{ \left(\begin{smallmatrix}
g_0 &&\\
&h_0 &\\
&&J_2 {}^{t}g_0^{-1}J_2
\end{smallmatrix}\right) : g_0\in \mathrm{GL}_2(F),  h_0\in P_4(F)\right\} ,\\
 N_{\gamma_2}(F)    &=\left\{ v(0, y, 0)\in N_{2,8}(F)\right\}.
\end{split}
 \end{equation*}
\end{itemize}
\label{lemma-double-coset-decomposition}
\end{lemma}

\begin{proof}
See \cite[Section 4]{GinzburgRallisSoudry1998} for the explicit representatives $\gamma_i$. Their stabilizers $H_{\gamma_i}(F)$ are not included in \cite[Section 4]{GinzburgRallisSoudry1998}, but they can be determined by matrix computations. We omit the details.
\end{proof}

We assume that $\mathrm{Re}(s)$ is sufficiently large. Now we unfold the Eisenstein series. We have
\begin{equation*}
\begin{split}
\mathcal{Z}(  \varphi,   &\theta_{\psi}^\Phi   ,  E(\cdot, f_{\Delta(\tau\otimes \chi_T, 2), s})) =  \int\limits_{\mathrm{Sp}_4(F)\backslash \mathrm{Sp}_4(\mathbb{A})} \int\limits_{N_{2, 8}(F)\backslash N_{2, 8}(\mathbb{A})} \varphi(h)     \theta_{\psi}^\Phi (\alpha_T(v)(1, h))   \\
& \sum_{\gamma_i \in P_8(F)\backslash \mathrm{Sp}_{8}(F)  / P_{2, 8}(F)}  \sum_{g \in H_{\gamma_i}(F) \backslash P_{2, 8}(F)} f_{\Delta(\tau\otimes \chi_T, 2), s}(\gamma_i g v(1, h))
dvdh. 
\end{split}
\end{equation*}
We  interchange the summation over $\gamma_i$ with integration. This shows that 
\begin{equation*}
\begin{split}
\mathcal{Z}( \varphi,   \theta_{\psi}^\Phi   , E(\cdot, f_{\Delta(\tau\otimes \chi_T, 2), s})) = \sum_{\gamma_i \in P_8(F)\backslash \mathrm{Sp}_8(F) / P_{2,8}(F)}I(\gamma_i;\varphi, \Phi, f_{\Delta(\tau\otimes \chi_T, 2), s}),
\end{split}
\end{equation*}
where $I(\gamma_i;\varphi, \Phi, f_{\Delta(\tau\otimes \chi_T, 2), s})$ is equal to
\begin{equation*}
\begin{split}
\int\limits_{\mathrm{Sp}_4(F)\backslash \mathrm{Sp}_4(\mathbb{A})} \int\limits_{N_{2, 8}(F)\backslash N_{2, 8}(\mathbb{A})}    \varphi(h)     \theta_{\psi}^\Phi (\alpha_T(v)(1, h)) \sum_{g \in H_{\gamma_i}(F) \backslash P_{2, 8}(F)     } f_{\Delta(\tau\otimes \chi_T, 2), s}(\gamma_i g v(1, h))  dvdh.
\end{split}
\end{equation*}

In the rest of this section we will show that there is a unique double coset $P_8(F)\gamma_2 P_{2,8}(F)$ coming from $\gamma_2$ that supports non-zero contribution towards the global integral $\mathcal{Z}( \varphi,   \theta_{\psi}^\Phi   , E(\cdot, f_{\Delta(\tau\otimes \chi_T, 2), s}))$. We will use the following two methods to show that $I(\gamma_i;\varphi, \Phi, f_{\Delta(\tau\otimes \chi_T, 2), s})=0$. The first one is to use the additive character. Specifically, if there is a unipotent group $U$ such that $\psi_U$ is a non-trivial additive character on $U({\mathbb A})$ which is trivial on $U(F)$, then the integral of $\psi_U$ on $U(F)\backslash U({\mathbb A})$ vanishes. The second one is to use the cuspidality of $\pi$: if $N$ is the unipotent radical of a parabolic subgroup of $\mathrm{Sp}_4$, then the integral $\int_{N(F)\backslash N({\mathbb A})}\varphi(nh)dn$ is zero because $\varphi$ is a cusp form. 
Similar methods appear in \cite[pp. 1000-1001]{CaiFriedbergGinzburgKaplan2019}. 

Henceforth until the end of Section~\ref{section-Unfolding}, to ease notation we may abbreviate$f_{\Delta(\tau\otimes \chi_T, 2), s}$ as simply $f$. We also write $[G]$ for the adelic quotient $G(F)\backslash G(\mathbb{A})$. All the computations involving integrals in the rest of this section are performed for $\mathrm{Re}(s)$ large enough. 

\subsubsection{Contribution from $I(\gamma_0;\varphi, \Phi, f_{\Delta(\tau\otimes \chi_T, 2), s})$}

In this subsection, we show that $I(\gamma_0;\varphi, \Phi, f_{\Delta(\tau\otimes \chi_T, 2), s})=0$.

By Lemma~\ref{lemma-double-coset-decomposition}, $ H_{\gamma_0}(F) \backslash P_{2, 8}(F)\cong P_4(F)\backslash \mathrm{Sp}_4(F)$.
Since $I(\gamma_0;\varphi, \Phi, f_{\Delta(\tau\otimes \chi_T, 2), s})$ is absolutely convergent (for $\mathrm{Re}(s)$ sufficiently large), we can interchange summation with integration over $N_{2,8}(F)\backslash N_{2,8}(\mathbb{A})$ to obtain
\begin{equation*}
\begin{split}
I(\gamma_0;\varphi, \Phi, f_{\Delta(\tau\otimes \chi_T, 2), s})= \int\limits_{[\mathrm{Sp}_4]}   \sum_{h_0 \in  P_4(F)\backslash \mathrm{Sp}_4(F)}    \int\limits_{[N_{2, 8}]}  
\varphi(h)     \theta_{\psi}^\Phi (\alpha_T(v)(1, h))     f (  (1, h_0)v(1, h))  dvdh.
\end{split}
\end{equation*}
We write $v=v(x, y, z)$. After conjugating $(1, h_0)$ to the right of $v(x, y, z)$, and doing a change of variables $v(x, y, z)\mapsto v((x, y)h_0, z)$, we get
\begin{equation}
\label{eq-unfolding-convengence-eq1}
\begin{split}
\int\limits_{[\mathrm{Sp}_4]}   \sum_{h_0 \in  P_4(F)\backslash \mathrm{Sp}_4(F)}    \int\limits_{[N_{2, 8}]}    
  \varphi(h) \theta_{\psi}^\Phi (\alpha_T(v((x, y)h_0, z))(1, h))     f (  v(x, y, z) (1, h_0)(1,  h))  dvdh.
\end{split}    
\end{equation}
Using the automorphy of $\varphi$ and $\theta_{\psi}^\Phi $, we see that for any $h_0\in \mathrm{Sp}_4(F)$, we have
\begin{equation*}
 \varphi(h) \theta_{\psi}^\Phi (\alpha_T(v((x, y)h_0, z))(1, h))  =  \varphi( h_0 h) \theta_{\psi}^\Phi (\alpha_T(v(x, y, z))(1, h_0 h)).
\end{equation*}
We may collapse the summation over $P_4(F)\backslash \mathrm{Sp}_4(F)$
with the $dh$-integration, to obtain
\begin{equation*}
\begin{split}
I(\gamma_0;\varphi, \Phi, f_{\Delta(\tau\otimes \chi_T, 2), s})  =  \int\limits_{P_4(F)\backslash \mathrm{Sp}_4(\mathbb{A})} \int\limits_{[N_{2, 8}]}      \varphi(h) \theta_{\psi}^\Phi (\alpha_T(v) (1, h))     f (  v   (1,  h))  dvdh.
\end{split}
\end{equation*}
We consider the following normal subgroup of $N_{2,8}$, 
\begin{equation*}
    N_{2,8}^c:=\left\{ v(0,0, z)=\left(\begin{smallmatrix}
I_2 &  &  & z\\
& I_2 & &  \\
& &I_2 & \\
& & &I_2
\end{smallmatrix}\right)\in N_{2,8} \right\}.
\end{equation*}
Note that for any $v(0,0,z)\in N_{2,8}^c(\mathbb{A})$, $v^\prime\in N_{2,8}(\mathbb{A}), h\in \mathrm{Sp}_4(\mathbb{A})$, by \eqref{Preliminaries-eq-Weil-formula1} we have
\begin{equation}
\begin{split}
\theta_{\psi}^\Phi(\alpha_T(v(0, 0, z))\alpha_T(v^\prime)(1, h)) =\psi(\mathrm{tr}(Tz)) \theta_{\psi}^\Phi(\alpha_T(v^\prime)(1, h)) , 
\end{split}
\label{eq-global-unfolding-I(gamma_0)-theta}
\end{equation}
and $f ( v(0,0,z) v^\prime   (1,  h)) = f (   v^\prime   (1,  h)) $
since $v(0,0,z)\in N_{8}(\mathbb{A})$.
Then $I(\gamma_0;\varphi, \Phi, f_{\Delta(\tau\otimes \chi_T, 2), s}) $ is equal to
\begin{equation*}
\begin{split}
& \int\limits_{P_4(F)\backslash \mathrm{Sp}_4(\mathbb{A})} \int\limits_{ N_{2,8}^c(\mathbb{A}) N_{2, 8}(F)\backslash N_{2, 8}(\mathbb{A})}  \int\limits_{[\mathrm{Mat}_2^0]}     \varphi(h) \theta_{\psi}^\Phi ( \alpha_T(v(0,0,z)) \alpha_T(v) (1, h))     f (v(0,0,z)  v   (1,  h))  dz dvdh \\
 = & \int\limits_{P_4(F)\backslash \mathrm{Sp}_4(\mathbb{A})} \int\limits_{ N_{2,8}^c(\mathbb{A}) N_{2, 8}(F)\backslash N_{2, 8}(\mathbb{A})}  \int\limits_{[\mathrm{Mat}_2^0]} \psi(\mathrm{tr}(Tz))dz  \varphi(h) \theta_{\psi}^\Phi ( \alpha_T( v) (1, h))     f (  v   (1,  h))   dvdh.
\end{split}
\end{equation*}
We conclude that $I(\gamma_0;\varphi, \Phi, f_{\Delta(\tau\otimes \chi_T, 2), s})$ vanishes since $\displaystyle \int\limits_{[\mathrm{Mat}_2^0]} \psi(\mathrm{tr}(Tz))dz=0$.

\subsubsection{Contribution from $I(\gamma_1;\varphi, \Phi, f_{\Delta(\tau\otimes \chi_T, 2), s})$}

In this subsection, we show that $I(\gamma_1;\varphi, \Phi, f_{\Delta(\tau\otimes \chi_T, 2), s})=0$.

By Lemma~\ref{lemma-double-coset-decomposition}, we see that $I(\gamma_1;\varphi, \Phi, f_{\Delta(\tau\otimes \chi_T, 2), s})$ is equal to
\begin{equation*}
\begin{split}
\int\limits_{[\mathrm{Sp}_4]} \int\limits_{[N_{2, 8}]}     \varphi(h)     \theta_{\psi}^\Phi (\alpha_T(v)(1, h))     \sum_{\substack{g_0\in B_{\mathrm{GL}_2}^{-}(F)\backslash \mathrm{GL}_2(F)\\ h_0\in P_4(F)\backslash \mathrm{Sp}_4(F) \\ v_0  \in  N_{\gamma_1}(F) \backslash N_{2,8}(F)}} 
f (\gamma_1 v_0 (g_0, h_0)v(1, h))  dvdh.
\end{split}
\end{equation*}
Interchanging the integration over $N_{2,8}(F)\backslash N_{2,8}(\mathbb{A})$ with the summation in $g_0$ and $h_0$, we get
\begin{equation*}
\begin{split}
 \int\limits_{[\mathrm{Sp}_4]}        \sum_{\substack{g_0\in B_{\mathrm{GL}_2}^{-}(F)\backslash \mathrm{GL}_2(F)\\ h_0\in P_4(F)\backslash \mathrm{Sp}_4(F)}}   \int\limits_{[N_{2, 8}]}     \varphi(h)     \theta_{\psi}^\Phi (\alpha_T(v)(1, h))      \sum_{\substack{v_0  \in  N_{\gamma_1}(F) \backslash N_{2,8}(F)}}   f (\gamma_1 v_0 (g_0, h_0)v(1, h))  dvdh.
\end{split}
\end{equation*}
We write $v=v(x, y, z)$. Then 
\begin{equation*}
\begin{split}
    (g_0, h_0)v(x, y, z) &=(g_0, 1) v((x, y)h_0^{-1}, z) (1, h_0),  \\
    (1, h_0)\alpha_T(v((x, y)h_0, z)) &=\alpha_T(v(x, y, z))(1, h_0).
\end{split}
\end{equation*}
Changing the variables $v(x, y, z)\mapsto v((x, y)h_0, z)$ and using automorphy of $\theta_{\psi}^\Phi$, and collapsing the summation over $P_4(F)\backslash \mathrm{Sp}_4(F)$ with the integration over $\mathrm{Sp}_4(F)\backslash \mathrm{Sp}_4(\mathbb{A})$, we obtain that $I(\gamma_1;\varphi, \Phi, f_{\Delta(\tau\otimes \chi_T, 2), s})$ is equal to
\begin{equation*}
\begin{split}
\int\limits_{P_4(F)\backslash \mathrm{Sp}_4(\mathbb{A})}  \sum_{g_0\in B_{\mathrm{GL}_2}^{-}(F)\backslash \mathrm{GL}_2(F)}                  \int\limits_{[N_{2, 8}]}     \varphi(h)     \theta_{\psi}^\Phi (\alpha_T(v)(1, h))        \sum_{\substack{v_0  \in  N_{\gamma_1}(F) \backslash N_{2,8}(F)}}   f (\gamma_1 v_0 (g_0, 1)v(1, h))  dvdh.
\end{split}
\end{equation*}
Recall that $N_{2,8}^c$ is a normal subgroup of $N_{2,8}$. Using
\eqref{eq-global-unfolding-I(gamma_0)-theta}, we get
\begin{equation*}
\begin{split}
I(\gamma_1;\varphi, \Phi, f_{\Delta(\tau\otimes \chi_T, 2), s})   = & \int\limits_{P_4(F)\backslash \mathrm{Sp}_4(\mathbb{A})}  \sum_{g_0\in B_{\mathrm{GL}_2}^{-}(F)\backslash \mathrm{GL}_2(F)}    \int\limits_{ N_{2,8}^c(\mathbb{A}) N_{2, 8}(F)\backslash N_{2, 8}(\mathbb{A})}  \int\limits_{[\mathrm{Mat}_2^0]}        \varphi(h) \psi(\mathrm{tr}(Tz)) \\
&\theta_{\psi}^\Phi (\alpha_T(v)(1, h))   \sum_{\substack{v_0  \in  N_{\gamma_1}(F) \backslash N_{2,8}(F)}} f (\gamma_1 v_0 (g_0, 1) v(0,0,z) v(1, h))  dzdvdh.
\end{split}
\end{equation*}
Now we use the identity $(g_0, 1) v(0,0,z)=v(0,0, g_0 z J_2 {}^t g_0 J_2) (g_0, 1)$
to conjugate $(g_0, 1)$ to the right of $v(0,0,z)$.
After making a change of variable $z\mapsto g_0^{-1}zJ_2 {}^t g_0^{-1} J_2$, we obtain that $I(\gamma_1;\varphi, \Phi, f_{\Delta(\tau\otimes \chi_T, 2), s})$ is equal to
\begin{equation*}
\begin{split}
&     \int\limits_{P_4(F)\backslash \mathrm{Sp}_4(\mathbb{A})}  \sum_{g_0\in B_{\mathrm{GL}_2}^{-}(F)\backslash \mathrm{GL}_2(F)}    \int\limits_{ N_{2,8}^c(\mathbb{A}) N_{2, 8}(F)\backslash N_{2, 8}(\mathbb{A})}  \int\limits_{[\mathrm{Mat}_2^0]}        \varphi(h) \psi(\mathrm{tr}(T g_0^{-1}zJ_2 {}^t g_0^{-1} J_2)) \\
& \ \ \ \ \ \ \    \theta_{\psi}^\Phi (\alpha_T(v)(1, h))   \sum_{\substack{v_0  \in  N_{\gamma_1}(F) \backslash N_{2,8}(F)}} f (\gamma_1 \cdot v_0 \cdot  v(0,0,z) \cdot (g_0, 1) \cdot  v\cdot (1, h))  dzdvdh.
\end{split}
\end{equation*}
Recall that $\mathrm{Mat}_2^0=\left\{ \left(\begin{smallmatrix}
z_1 & z_2 \\
z_3 & z_1
\end{smallmatrix} \right)\right\}$. Since $N_{2,8}^c$ is a normal subgroup of $N_{2,8}$, we have
\begin{equation*}
\begin{split}
I(\gamma_1;\varphi, \Phi, f_{\Delta(\tau\otimes \chi_T, 2), s})   =     & \int\limits_{P_4(F)\backslash \mathrm{Sp}_4(\mathbb{A})}  \sum_{g_0\in B_{\mathrm{GL}_2}^{-}(F)\backslash \mathrm{GL}_2(F)}    
\int\limits_{ N_{2,8}^c(\mathbb{A}) N_{2, 8}(F)\backslash N_{2, 8}(\mathbb{A})}  \int\limits_{(F\backslash \mathbb{A})^3}          \\
&         \psi(\mathrm{tr}(T g_0^{-1}\left(\begin{smallmatrix} z_1 & z_2 \\ z_3 & z_1\end{smallmatrix}\right)J_2 {}^t g_0^{-1} J_2))      \varphi(h) \theta_{\psi}^\Phi (\alpha_T(v)(1, h))  \sum_{\substack{v_0  \in  N_{\gamma_1}(F) \backslash N_{2,8}(F)}}    \\
&           f \left(\gamma_1 \cdot   v\left(0,0, \left(\begin{smallmatrix} 0 & 0 \\z_3&0\end{smallmatrix}\right) \right) \cdot  v\left(0,0,\left(\begin{smallmatrix} z_1 & z_2 \\ 0 & z_1\end{smallmatrix}\right) \right)
  v_0 \cdot   (g_0, 1)    v  (1, h)
\right)  dz_3 dz_2 dz_1 dvdh.
\end{split}
\end{equation*}
Note that $\gamma_1$ commutes with $v\left(0,0, \left(\begin{smallmatrix} 0 & 0 \\z_3&0\end{smallmatrix}\right) \right)$, and $v\left(0,0, \left(\begin{smallmatrix} 0 & 0 \\z_3&0\end{smallmatrix}\right) \right)\in N_8(\mathbb{A})$. Therefore, we get
\begin{equation*}
\int\limits_{F\backslash \mathbb{A}}  \psi\left(\mathrm{tr}(Tg_0^{-1}  \left(\begin{smallmatrix}
0 & 0\\
z_3 &0
\end{smallmatrix}\right) J_2 {}^t g_0^{-1} J_2)\right) dz_3
\end{equation*}
as an inner integration, which is zero since $z_3\mapsto \psi\left(\mathrm{tr}(Tg_0^{-1}  \left(\begin{smallmatrix}
0 & 0\\
z_3 &0
\end{smallmatrix}\right) J_2 {}^t g_0^{-1} J_2)\right)$ is a non-trivial character on $F\backslash {\mathbb A}$. Thus, $I(\gamma_1;\varphi, \Phi, f_{\Delta(\tau\otimes \chi_T, 2), s})=0$.

We now have shown that for $\mathrm{Re}(s)\gg 0$, we have
$
 \mathcal{Z}( \varphi,   \theta_{\psi}^\Phi   , E(\cdot, f)) = I(\gamma_2;\varphi, \Phi, f_{\Delta(\tau\otimes \chi_T, 2), s}).   
$

\subsubsection{Contribution from $I(\gamma_2;\varphi, \Phi, f_{\Delta(\tau\otimes \chi_T, 2), s})$}

Using Lemma~\ref{lemma-double-coset-decomposition}, we see that $I(\gamma_2;\varphi, \Phi, f_{\Delta(\tau\otimes \chi_T, 2), s})$ is equal to
\begin{equation*}
\begin{split}
\int\limits_{[\mathrm{Sp}_4]} \int\limits_{[N_{2, 8}]}     \varphi(h)     \theta_{\psi}^\Phi (\alpha_T(v)(1, h))      \sum_{\substack{ h_0\in P_4(F)\backslash \mathrm{Sp}_4(F) \\ v_0  \in  N_{\gamma_2}(F) \backslash N_{2,8}(F)}} f (\gamma_2 v_0 (1, h_0)v(1, h))  dvdh.
\end{split}
\end{equation*}
We interchange the summation over $P_4(F)\backslash \mathrm{Sp}_4(F)$ with the $dv$-integration, then make a change of variables $v(x, y, z)\mapsto v((x, y)h_0, z)$, and then use the automorphy of $\varphi$ and $\theta_{\psi}^\Phi$, and then collapse the integrals with summations, to get
\begin{equation*}
\begin{split}
I(\gamma_2;\varphi, \Phi, f_{\Delta(\tau\otimes \chi_T, 2), s})   = \int\limits_{P_4(F)\backslash \mathrm{Sp}_4(\mathbb{A})} \int\limits_{N_{\gamma_2}(F) \backslash N_{2, 8}(\mathbb{A})}     \varphi(h)     \theta_{\psi}^\Phi (\alpha_T(v)(1, h))    f (\gamma_2 v (1, h))  dvdh.
\end{split}
\end{equation*}

Our next step is to unfold the theta series $\theta_{\psi}^\Phi$. Note that $\mathrm{GL}_2(F)$ acts on $\mathrm{Mat}_2 (F)$ by right multiplication.
There are four $\mathrm{GL}_2(F)$-orbits in $\mathrm{Mat}_2 (F)$, which are given below:
\begin{itemize}
\item   $\mathrm{GL}_2(F)$, represented by $\xi=I_2$;
\item   $\{0\}$, represented by $\xi=0$;
\item   $\left\{\left(\begin{smallmatrix}
a & b\\
\lambda a &\lambda b
\end{smallmatrix}\right):(a ,b)\not=(0,0), \lambda\in F \right\}$, represented by $\xi_\lambda=\left(\begin{smallmatrix}
0 & 1\\
0 &\lambda
\end{smallmatrix}\right)$;
\item   $\left\{\left(\begin{smallmatrix}
0 & 0\\
a & b
\end{smallmatrix}\right):(a ,b)\not=(0,0) \right\}$, represented by $\xi=\left(\begin{smallmatrix}
0 & 0\\
0 &1
\end{smallmatrix}\right)$.
\end{itemize}
For $\xi\in \mathrm{Mat}_2(F)$, we denote the stabilizer of $\xi$ by $G_\xi$, i.e., $G_\xi =\{g\in \mathrm{GL}_2(F):\xi g=\xi\}$. Then 
\begin{equation*}
\begin{split}
 \theta_{\psi, 8}^\Phi (\alpha_T( v)(1, h)) = \sum_{\xi\in \mathrm{Mat}_2(F)/\mathrm{GL}_2(F)} \sum_{a\in G_{\xi}\backslash \mathrm{GL}_2(F)} \omega_\psi(\alpha_T(v)(1, h))\Phi(\xi a),
\end{split}
\end{equation*}
and hence $I(\gamma_2;\varphi, \Phi, f_{\Delta(\tau\otimes \chi_T, 2), s})$ is equal to
\begin{equation*}
\begin{split}
     \int\limits_{P_4(F)\backslash \mathrm{Sp}_4(\mathbb{A})}  \int\limits_{N_{\gamma_2}(F)\backslash N_{2, 8}(\mathbb{A})} 
          \sum_{\xi\in \mathrm{Mat}_2(F)/\mathrm{GL}_2(F)} \sum_{a\in G_{\xi}\backslash \mathrm{GL}_2(F)}  \varphi(h)  \omega_\psi(\alpha_T(v)(1, h))\Phi(\xi a)   f (\gamma_2  v(1, h))  dvdh.
\end{split} 
\end{equation*}
After interchanging the summations with the integration over $N_{\gamma_2}(F)\backslash N_{2, 8}(\mathbb{A})$, we obtain
\begin{equation*}
\begin{split}
   \int\limits_{P_4(F)\backslash \mathrm{Sp}_4(\mathbb{A})}  \sum_{\xi\in \mathrm{Mat}_2(F)/\mathrm{GL}_2(F)} \sum_{a\in G_{\xi}\backslash \mathrm{GL}_2(F)}  \int\limits_{N_{\gamma_2}(F) \backslash N_{2, 8}(\mathbb{A})}   \varphi(h) \omega_\psi(\alpha_T(v)(1, h))\Phi(\xi a)   f (\gamma_2  v(1, h))  dvdh.
\end{split} 
\end{equation*}
We further interchange the summation over $\mathrm{Mat}_2(F)/\mathrm{GL}_2(F)$ with the integration over $P_4(F)\backslash \mathrm{Sp}_4(\mathbb{A})$, to get
\begin{equation*}
\begin{split}
I(\gamma_2;\varphi, \Phi, f_{\Delta(\tau\otimes \chi_T, 2), s})=  \sum_{\xi\in \mathrm{Mat}_2(F)/\mathrm{GL}_2(F)} I(\gamma_2, \xi;\varphi, \Phi, f_{\Delta(\tau\otimes \chi_T, 2), s})     ,
\end{split} 
\end{equation*}
where $I(\gamma_2, \xi;\varphi, \Phi, f_{\Delta(\tau\otimes \chi_T, 2), s})$ is equal to
\begin{equation*}
\begin{split}
 \int\limits_{P_4(F)\backslash \mathrm{Sp}_4(\mathbb{A})}    \sum_{a\in G_{\xi}\backslash \mathrm{GL}_2(F)}        \int\limits_{N_{\gamma_2}(F)\backslash N_{2, 8}(\mathbb{A})} \varphi(h) \omega_\psi(\alpha_T(v)(1, h))\Phi(\xi a)   f (\gamma_2  v(1, h))  dvdh.
 \end{split}
\end{equation*}

Next, we prove the following statement.
\begin{proposition}
$I(\gamma_2, \xi;\varphi, \Phi, f_{\Delta(\tau\otimes \chi_T, 2), s})$ vanishes for all $\xi$ except when $\xi\in \mathrm{GL}_2(F)$.
\label{prop-global-unfolding-I(gamma_2, xi)-vanishes}
\end{proposition}

\begin{proof}
First, we show that $I(\gamma_2, 0;\varphi, \Phi, f_{\Delta(\tau\otimes \chi_T, 2), s})=0$. The stabilizer of $\xi=0$ is the full group $\mathrm{GL}_2(F)$. We see that $I(\gamma_2, 0;\varphi, \Phi, f_{\Delta(\tau\otimes \chi_T, 2), s})$ is equal to
\begin{equation*}
\begin{split}
  \int\limits_{N_4(\mathbb{A})M_4(F)\backslash \mathrm{Sp}_4(\mathbb{A})}    \int\limits_{[N_4]}    \int\limits_{N_{\gamma_2}(F)\backslash N_{2, 8}(\mathbb{A})}    \varphi(n h) \omega_\psi(\alpha_T(v)(1, nh))\Phi(0)   f (\gamma_2  v(1, nh))  dvdndh.
 \end{split}
\end{equation*}
We observe that for any $n\in N_4(\mathbb{A})$, $\Phi^\prime\in \mathcal{S}(\mathrm{Mat}_2(\mathbb{A}))$, $g\in \mathrm{Sp}_8(\mathbb{A})$, we have
\begin{equation*}
\begin{split}
    \omega_\psi \left( (1, n)  \right) \Phi^\prime(0) =\Phi^\prime(0), \, \, 
    f \left(\gamma_2 (1, n)g\right)&=f\left((1, n) \gamma_2 g\right)=f(\gamma_2 g).
\end{split}
\end{equation*}
We write $v=v(x, y, z)$. Then  
\begin{equation*}
\begin{split}
&I(\gamma_2, 0;\varphi, \Phi, f_{\Delta(\tau\otimes \chi_T, 2), s}) \\
=&\int\limits_{N_4(\mathbb{A})M_4(F)\backslash \mathrm{Sp}_4(\mathbb{A})}    \int\limits_{[N_4]}    \int\limits_{N_{\gamma_2}(F)\backslash N_{2, 8}(\mathbb{A})}  \\
&       \varphi(n h) \omega_\psi((1, n) \alpha_T(v(x, y)n, z))(1, h))\Phi(0)   f (  \gamma_2 (1, n)  v((x, y)n, z) (1, h))  dvdndh \\
=&\int\limits_{N_4(\mathbb{A})M_4(F)\backslash \mathrm{Sp}_4(\mathbb{A})}    \int\limits_{[N_4]}    \int\limits_{N_{\gamma_2}(F)\backslash N_{2, 8}(\mathbb{A})}  \\
& \ \ \       \varphi(n h) \omega_\psi( \alpha_T(v(x, y)n, z))(1, h))\Phi(0)   f (  \gamma_2    v((x, y)n, z) (1, h))  dvdndh.
 \end{split}
\end{equation*}
After a change of variables $v(x, y, z)\mapsto v((x, y)n^{-1}, z)$, we see that $I(\gamma_2, 0;\varphi, \Phi, f_{\Delta(\tau\otimes \chi_T, 2), s})$ is equal to
\begin{equation*}
\begin{split}
 \int\limits_{N_4(\mathbb{A})M_4(F)\backslash \mathrm{Sp}_4(\mathbb{A})} &    \int\limits_{N_{\gamma_2}(F)\backslash N_{2, 8}(\mathbb{A})}    \int\limits_{[N_4]}      \varphi(n h) dn \omega_\psi( \alpha_T(v)(1, h))\Phi(0)   f (  \gamma_2    v (1, h))  dvdh.
 \end{split}
\end{equation*}
Note that 
$
\int\limits_{[N_4]}      \varphi(n h) dn=0
$
because $\varphi$ is a cusp form. Hence $I(\gamma_2, 0;\varphi, \Phi, f_{\Delta(\tau\otimes \chi_T, 2), s})=0$.

It remains to prove $ I(\gamma_2, \xi;\varphi, \Phi, f_{\Delta(\tau\otimes \chi_T, 2), s})=0$ for $\xi=\left(\begin{smallmatrix}
0 & 1\\
0 &\lambda
\end{smallmatrix}\right)$ and $\xi=\left(\begin{smallmatrix}
0 & 0\\
0 &1
\end{smallmatrix}\right)$. Let $\xi$ be either $\left(\begin{smallmatrix}
0 & 1\\
0 &\lambda
\end{smallmatrix}\right)$ or $\left(\begin{smallmatrix}
0 & 0\\
0 &1
\end{smallmatrix}\right)$. In both cases, 
$
G_\xi=\left\{ \left(\begin{smallmatrix}
* & *\\
0 &1
\end{smallmatrix}\right)\in \mathrm{GL}_2(F)\right\}
$
is the mirabolic subgroup of $\mathrm{GL}_2(F)$, and 
$
\mathrm{tr}\left( T \  {}^t \xi \left(\begin{smallmatrix}
z_1 & z_2\\
z_3 & z_1
\end{smallmatrix}\right) \xi \right)=0.
$
Similar to the case $I(\gamma_2, 0;\varphi, \Phi, f_{\Delta(\tau\otimes \chi_T, 2), s})$, we observe that for any $n=\left(\begin{smallmatrix} I_2 & z \\ & I_2\end{smallmatrix}\right)\in N_4(\mathbb{A})$ and any $a\in \mathrm{GL}_2(F)$, we have
\begin{equation}
\begin{split}
  \omega_\psi((1, n)  \alpha_T(v)(1, h))\Phi(\xi) &= \omega_\psi( \alpha_T(v)(1, h))\Phi(\xi),
\end{split}
\label{eq-global-unfolding-I(gamma2)-vanishing-1}
\end{equation}
\begin{equation}
\begin{split}
  f (  \gamma_2  (1, n)  g) = f (    (1, n) \gamma_2  g) 
  =f (    \gamma_2  g),
\end{split}
\label{eq-global-unfolding-I(gamma2)-vanishing-2}
\end{equation}
and
\begin{equation}
\begin{split}
  f (  \gamma_2    g) &= f (  (1, \mathrm{diag}(a, a^*))  \gamma_2  g)=f (    \gamma_2 (1, \mathrm{diag}(a, a^*))  g).
  \label{eq-global-unfolding-I(gamma2)-vanishing-3}
\end{split}
\end{equation}
We also note that 
\begin{equation}
    \omega_\psi(\alpha_T(v)(1, h))\Phi(\xi a) =  \omega_\psi( (1, \hat{a}) \alpha_T(v)(1, h))\Phi(\xi),
\label{eq-global-unfolding-I(gamma2)-vanishing-4}
\end{equation}
where $\hat{a}=\mathrm{diag}(a, a^*)$. For a subgroup $G\subset \mathrm{GL}_2(F)$, we denote by $\widehat{G}=\left\{ \mathrm{diag}(a, a^*): a\in G\right\}$ its embedding into $\mathrm{Sp}_4(F)$. Using \eqref{eq-global-unfolding-I(gamma2)-vanishing-4}, we see that $I(\gamma_2, \xi;\varphi, \Phi, f_{\Delta(\tau\otimes \chi_T, 2), s})$ is equal to
\begin{equation*}
\begin{split}
   \int\limits_{P_4(F)\backslash \mathrm{Sp}_4(\mathbb{A})}    \sum_{a\in G_{\xi}\backslash \mathrm{GL}_2(F)}        \int\limits_{N_{\gamma_2}(F) \backslash N_{2, 8}(\mathbb{A})}  \varphi(h) \omega_\psi((1, \hat{a}) \alpha_T(v)(1, h))\Phi(\xi)   f (\gamma_2  v(1, h))  dvdh.
 \end{split}
\end{equation*}
We then make a change of variables $v(x, y, z)\mapsto v((x, y)\hat{a}, z)$, and use \eqref{eq-global-unfolding-I(gamma2)-vanishing-3}, to get that 
\begin{equation*}
\begin{split}
  \int\limits_{P_4(F)\backslash \mathrm{Sp}_4(\mathbb{A})}    \sum_{a\in G_{\xi}\backslash \mathrm{GL}_2(F)}        \int\limits_{N_{\gamma_2}(F)\backslash N_{2, 8}(\mathbb{A})}  \varphi(h) \omega_\psi( \alpha_T(v) (1, \hat{a}) (1, h))\Phi(\xi) 
  f (\gamma_2  v (1, \hat{a})(1, h))  dvdh.
\end{split}
\end{equation*}
Now we use the automorphy of $\varphi$ and then collapse the integral over $P_4(F)\backslash \mathrm{Sp}_4(\mathbb{A})$ with the summation over $G_{\xi}\backslash \mathrm{GL}_2(F)$,
to obtain
\begin{equation*}
\int\limits_{\hat{G_\xi} N_4(F)\backslash \mathrm{Sp}_4(\mathbb{A})}            \int\limits_{N_{\gamma_2}(F)\backslash N_{2, 8}(\mathbb{A})}  \varphi(h) \omega_\psi( \alpha_T(v)   (1, h))\Phi(\xi)   f (\gamma_2  v  (1, h))  dvdh.
\end{equation*}
Note that $N_4(F)$ is a normal subgroup of $\widehat{G_\xi} N_4(F)$. We see that 
\begin{equation*}
\begin{split}
&I(\gamma_2, \xi;\varphi, \Phi, f_{\Delta(\tau\otimes \chi_T, 2), s}) \\
=&   \int\limits_{\widehat{G_\xi} N_4(\mathbb{A})\backslash \mathrm{Sp}_4(\mathbb{A})}   \int\limits_{[N_4]}          \int\limits_{N_{\gamma_2}(F)\backslash N_{2, 8}(\mathbb{A})}  \varphi(n h) \omega_\psi( \alpha_T(v)   (1, n h))\Phi(\xi)   f (\gamma_2  v  (1, n h))  dvdndh \\ 
=& \int\limits_{\widehat{G_\xi} N_4(\mathbb{A})\backslash \mathrm{Sp}_4(\mathbb{A})}   \int\limits_{[N_4]}          \int\limits_{N_{\gamma_2}(F)\backslash N_{2, 8}(\mathbb{A})}  \varphi(n h) \omega_\psi( (1, n) \alpha_T(v ((x, y)n, z))   (1,  h))\Phi(\xi)  \\
& \ \ \ \ \ \ \ \ \ \ \ \ \ \ \ \ \ \ \ \ \ \ \ \ \ f (\gamma_2 (1, n)  v((x, y)n, z)  (1,  h))  dvdndh.
\end{split}    
\end{equation*}
Then we make a change of variables $v(x, y, z)\mapsto v((x, y)n^{-1}, z)$ and use \eqref{eq-global-unfolding-I(gamma2)-vanishing-1} and \eqref{eq-global-unfolding-I(gamma2)-vanishing-2} to obtain that $I(\gamma_2, \xi;\varphi, \Phi, f_{\Delta(\tau\otimes \chi_T, 2), s})$ is equal to
\begin{equation*}
\begin{split}
\int\limits_{\widehat{G_\xi} N_4(\mathbb{A})\backslash \mathrm{Sp}_4(\mathbb{A})}   \int\limits_{[N_4]}          \int\limits_{N_{\gamma_2}(F)\backslash N_{2, 8}(\mathbb{A})}   \varphi(n h) \omega_\psi(  \alpha_T(v )   (1,  h))\Phi(\xi)    f (\gamma_2   v  (1,  h))  dvdndh.
\end{split}    
\end{equation*}
As in the case of $I(\gamma_2, 0;\varphi, \Phi, f_{\Delta(\tau\otimes \chi_T, 2), s})$, we obtain
$
  \int\limits_{[N_4]} \varphi(nh)dn  
$
as an inner integration, which is zero by the cuspidality of $\varphi$. Thus $I(\gamma_2, \xi;\varphi, \Phi, f_{\Delta(\tau\otimes \chi_T, 2), s})=0$.
\end{proof}

\subsubsection{Completion of the proof}
Now we are ready to finish the proof of Theorem~\ref{theorem-global-unfolding}. We have proved that $\mathcal{Z}( \varphi,   \theta_{\psi}^\Phi   , E(\cdot, f)) = I(\gamma_2, I_2;\varphi, \Phi, f_{\Delta(\tau\otimes \chi_T, 2), s})$ for $\mathrm{Re}(s)$ sufficiently large.
Note that for the representative $\xi=I_2$, the embedding of the stabilizer $G_{I_2}$ is $\widehat{G_{I_2}}=\{I_4\}$. We see that $\mathcal{Z}( \varphi,   \theta_{\psi}^\Phi   , E(\cdot, f))$ is equal to (for $\mathrm{Re}(s)$ sufficiently large)
\begin{equation*}
\begin{split}
     \int\limits_{N_4(\mathbb{A})\backslash \mathrm{Sp}_4(\mathbb{A})}     \int\limits_{[N_4]}       \int\limits_{N_{\gamma_2}(F)\backslash N_{2, 8}(\mathbb{A})}  \varphi( n h) \omega_\psi(\alpha_T(v)(1,  n h))\Phi(I_2)    f (\gamma_2  v(1, n h))  dvdndh.
\end{split}
\end{equation*}
Note that for $n=\left(\begin{smallmatrix} I_2 & z \\ & I_2\end{smallmatrix}\right)\in N_4(\mathbb{A})$, we have
\begin{equation}
\begin{split}
  \omega_\psi((1, n)  \alpha_T(v)(1, h))\Phi(I_2) &= \psi \left( \mathrm{tr}\left( T z \right)\right)   \omega_\psi( \alpha_T(v)(1, h))\Phi(I_2).
\end{split}
\label{eq-global-unfolding-Weil-action-1}
\end{equation}
Making a change of variables $v(x, y, z)\mapsto v((x, y)n^{-1}, z)$ and using \eqref{eq-global-unfolding-I(gamma2)-vanishing-2} and  \eqref{eq-global-unfolding-Weil-action-1}, we see that $\mathcal{Z}( \varphi,   \theta_{\psi}^\Phi   , E(\cdot, f))$ is equal to (for $\mathrm{Re}(s)$ sufficiently large)
\begin{equation*}
\begin{split}
     \int\limits_{N_4(\mathbb{A})\backslash \mathrm{Sp}_4(\mathbb{A})}       \int\limits_{N_{\gamma_2} (\mathbb{A})\backslash N_{2, 8}(\mathbb{A})} \int\limits_{[N_{\gamma_2}]}  \varphi_{\psi, T}( h) \omega_\psi(\alpha_T(v_0) \alpha_T(v)(1,   h))\Phi(I_2)   f (\gamma_2 v_0  v(1,  h))  dv_0 dvdh.
\end{split}
\end{equation*}
Recall that 
\begin{equation*}
    N_{\gamma_2}=\left\{ v(0,y, 0)\in N_{2,8}\right\}, N_{2,8}^0=\left\{ v(x,0,z)\in N_{2,8}\right\}.
\end{equation*}
For fixed $v\in N_{2,8}(\mathbb{A})$ and $h\in \mathrm{Sp}_4(\mathbb{A})$, we have
\begin{equation*}
\begin{split}
    &\int\limits_{[N_{\gamma_2}]} \omega_\psi(\alpha_T(v_0) \alpha_T(v)(1,   h))\Phi(I_2)   f (\gamma_2 v_0  v(1,  h)) dv_0 \\
   = &  \omega_\psi( \alpha_T(v)(1,   h))\Phi(I_2) \int\limits_{\mathrm{Mat}_2(F)\backslash \mathrm{Mat}_2(\mathbb{A})}    f \left(    \left(\begin{smallmatrix}&I_2&&\\I_2 &&&\\&&&I_2\\&&I_2&
     \end{smallmatrix}\right) \gamma_2 \left(\begin{smallmatrix} I_2 &&y&\\ &I_2 &&y^*\\ &&I_2&\\ &&&I_2\end{smallmatrix}\right)   v(1,  h) \right)    \psi(\mathrm{tr}(2Ty)) dy\\
                = &  \omega_\psi( \alpha_T(v)(1,   h))\Phi(I_2) \int\limits_{\mathrm{Mat}_2(F)\backslash \mathrm{Mat}_2(\mathbb{A})} f \left(     \left(\begin{smallmatrix} I_2 &y&&\\ &I_2 &&\\ &&I_2&y^*\\ &&&I_2\end{smallmatrix}\right) \gamma   v(1,  h) \right) \psi^{-1}(\mathrm{tr}(2Ty)) dy \\
                =&  \omega_\psi( \alpha_T(v)(1,   h))\Phi(I_2) f_{\mathcal{W}(\tau\otimes \chi_T, 2, \psi_{2T}), s}(\gamma v (1, h)).
\end{split}
\end{equation*}
We have completed the proof of Theorem~\ref{theorem-global-unfolding}.

\section{The unramified computation}
\label{section-Unramified computation}

Recall that by Theorem~\ref{theorem-global-unfolding}, for $\mathrm{Re}(s)\gg 0$, the global integral $\mathcal{Z}( \varphi,   \theta_{\psi}^\Phi   , E^{*, S}(\cdot, f_{\Delta(\tau\otimes \chi_T, 2), s}))$ unfolds to (see \eqref{eq-thm-global-integral-unfolding-final})
\begin{equation*}
\begin{split}
 \int\limits_{ N_4(\mathbb{A})  \backslash  \mathrm{Sp}_4(\mathbb{A})}     \int\limits_{N_{2,8}^0(\mathbb{A})}      \varphi_{\psi, T}(h) \omega_\psi\left( \alpha_T(v)(1,h)\right) )\Phi(I_2 )     f^*_{\mathcal{W}(\tau\otimes \chi_T, 2, \psi_{2T}), s}(\gamma v(1, h) )      dv dh.
\end{split}
\end{equation*}
We remark that,
if $f_{\Delta(\tau\otimes \chi_T, 2), s}$ is decomposable, then 
by \cite[Theorem 13]{CaiFriedbergGinzburgKaplan2019} and \cite[Theorem 4]{CaiFriedbergGourevitchKaplan2023}, we can write
\begin{equation*}
  f_{\mathcal{W}(\tau\otimes \chi_T, 2, \psi_{2T}), s}(g) = \prod_{\nu} f_{\mathcal{W}(\tau_\nu\otimes \chi_{T}, 2, \psi_{2T}), s}(g_\nu), 
\end{equation*}
where $f_{\mathcal{W}(\tau_\nu\otimes \chi_{T}, 2, \psi_{2T}), s}\in \mathrm{Ind}_{P_8(F_\nu)}^{\mathrm{Sp}_8(F_\nu)}(\mathcal{W}(\tau_\nu\otimes\chi_{T}, 2, \psi_{2T})\vert \det \vert^s)$. 
We caution the reader that we have used $\psi_{2T}$ (respectively, $\chi_T$) for both the global additive (respectively, multiplicative) character and its local component. We hope that the meanings of these characters are evident from the context. 
In general, the Fourier coefficient $\varphi_{\psi, T}$ does not factor into an Euler product, due to the lack of uniqueness of corresponding linear functionals on the space of representations. However, there is still a technique, developed by  Piatetski-Shapiro and Rallis
\cite{Piatetski-ShapiroRallis1988}, to handle this situation.

In this section, we carry out the local unramified computation. The local unramified integral at a finite unramified place $\nu$ takes the form
\begin{equation*}
\begin{split}
 \int\limits_{N_4({F_\nu})\backslash \mathrm{Sp}_4({F_\nu})}   \int\limits_{N_{2,8}^0({F_\nu})}  l_{T,\nu}({\pi_\nu}(h)v_0)    \omega_{{\psi}}\left( \alpha_T(v)(1,h)\right) )\Phi^0_\nu(I_2 )
       f^*_{\mathcal{W}({\tau} \otimes\chi_{T}, 2, \psi_{2T}), s}( \gamma v (1, h) )    dv     dh .
\end{split}
\end{equation*}
This integral is obtained by replacing the global Fourier coefficient $\varphi_{\psi, T}$ in \eqref{eq-thm-global-integral-unfolding-final} by its local analogue $l_{T, \nu}(\pi_\nu(h)v_0)$ where $v_0$ is an unramified vector in $V_{\pi_\nu}$ and $l_{T,\nu}$ is a linear functional on $V_{\pi_\nu}$ satisfying the analogous equivariance property as $\varphi_{\psi, T}$, and by replacing all the other terms in \eqref{eq-thm-global-integral-unfolding-final} by their local counterparts (note that all the other terms are factorizable). We assume $\nu\nmid 2, 3$.

Henceforth until the end this section, 
we drop the reference to the local place $\nu$ to ease notation. Throughout Section~\ref{section-Unramified computation}, we let $F$ be a non-archimedean local field of characteristic zero with ring $\mathcal{O}_F$ of integers, with a fixed  uniformizer $\varpi$. Let $q$ be the cardinality of the residue field. We assume that the residual characteristic is not equal to $2$ or $3$. The absolute value $\vert\cdot \vert$ on $F$ is normalized so that $\vert\varpi\vert=q^{-1}$. We fix a  non-trivial additive unramified character $\psi$ of $F$. For any $a\in F^\times$, the character $\psi_a$ is defined by $\psi_a(x)=\psi(ax)$. The local Hilbert symbol is denoted by $(\ , \ )_F$. Let $T_0=\mathrm{diag}(t_1, t_2)$, $T=J_2 T_0$, and we assume that the diagonal coordinates of $T_0$ are in $\mathcal{O}_F^\times$. The character $\chi_T$ is the quadratic character on $F^\times$ defined by $\chi_T(a)=(a, \det(T))_F$. 
Given a character $\chi$ on $F^\times$, we extend it to a character on $\mathrm{GL}_n(F)$ by compositing it with the determinant map, and we still denote it by $\chi$, so that $\chi=\chi\circ \det$ is a character on $\mathrm{GL}_n(F)$.
Let $\Phi^0=\textbf{1}_{\mathrm{Mat}_2(\mathcal{O}_F)}$ be the characteristic function of $\mathrm{Mat}_2(\mathcal{O}_F)$. 
Let $(\pi, V_\pi)$ be an irreducible admissible unramified representation of $\mathrm{Sp}_4(F)$, with a non-zero unramified vector $v_0\in V_\pi$. 
Let $(\tau, V_\tau)$ be an irreducible unramified principle series representation $\mathrm{Ind}_{\mathrm{B}_{\mathrm{GL}_2}(F)}^{\mathrm{GL}_2(F)}(\chi_1 \otimes \chi_2)$, where $\chi_1, \chi_2$ are unramified quasi-characters of $F^\times$.

\subsection{Preliminaries for the unramified computation}\label{Preliminaries for the unramified computation}
In this subsection we prove a relation between unramified sections belonging to different induced representations that we will use later. 

The local Speh representation $\Delta(\tau\otimes \chi_T, 2)$ of $\mathrm{GL}_4(F)$ associated to $\tau\otimes\chi_T$ is given by (see \cite[Claim 9]{CaiFriedbergGinzburgKaplan2019})
\begin{equation}
\Delta(\tau\otimes\chi_T, 2)=\mathrm{Ind}_{P_{(2^2)}(F)}^{\mathrm{GL}_4(F)}(\chi_1 \chi_T \otimes \chi_2 \chi_T).
\label{eq-local-Speh-representation-GL4}
\end{equation}
It is the unique irreducible quotient of 
\begin{equation*}
\mathrm{Ind}_{P_{(2^2)}(F)}^{\mathrm{GL}_4(F)} (( (\tau\otimes\chi_T)\otimes (\tau\otimes\chi_T) )\delta_{P_{(2^2)}}^{1/4})
\end{equation*}
and the unique irreducible subrepresentation of 
\begin{equation*}
\mathrm{Ind}_{P_{(2^2)}(F)}^{\mathrm{GL}_4(F)} (( (\tau\otimes\chi_T)\otimes (\tau\otimes\chi_T) )\delta_{P_{(2^2)}}^{-1/4}).
\end{equation*}
This representation affords a unique non-zero linear functional $\Lambda$ on $\Delta(\tau\otimes \chi_T, 2)$ such that 
\begin{equation}
\label{eq-local-Speh-linear-functional}
\Lambda\left(\Delta(\tau\otimes \chi_T, 2)\left(\begin{smallmatrix}
I_2 &Y\\
&I_2
\end{smallmatrix}\right) \xi \right) =\psi(\mathrm{tr}(Y))\Lambda(\xi).
\end{equation}
Put $w_{2,2}=\left(\begin{smallmatrix} & I_2\\I_2 & \end{smallmatrix}\right)\in \mathrm{GL}_4.$
Then $\Lambda$ can be realized by the Jacquet integral (see  \cite[(3.4)]{CaiFriedbergGinzburgKaplan2019})
\begin{equation}
\xi  \mapsto  \int\limits_{U_{(2^2)}(F)} \xi(w_{2,2} u)\psi^{-1}(u)du.
\label{eq-local-Speh-linear-functional-explicit-GL4-psi}
\end{equation}

\begin{remark}
\label{remark-Jacquet-integral}
Since the integral \eqref{eq-local-Speh-linear-functional-explicit-GL4-psi} may not converge absolutely, we provide some explanation. 
Twisting the inducing data using auxiliary complex parameters, i.e., replacing $\chi_i\chi_T$ by $|\cdot|^{\zeta_i}\chi_i\chi_T$ with $\zeta_i\in \mathbb{C}$ for $1\le i\le 2$, there is a cone where the integral \eqref{eq-local-Speh-linear-functional-explicit-GL4-psi} is absolutely convergent. One can also choose data such that \eqref{eq-local-Speh-linear-functional-explicit-GL4-psi} is absolutely convergent and equals 1, for all choices of $\zeta_i$. Since the space of linear functionals on $\Delta(\tau\otimes \chi_T, 2)$ satisfying \eqref{eq-local-Speh-linear-functional} is one-dimensional, by invoking Bernstein's continuation principle (in \cite{Banks1998}) one sees that \eqref{eq-local-Speh-linear-functional-explicit-GL4-psi} admits analytic continuation in the parameters $\zeta_i$, and it also follows that it is a nonzero functional on $\Delta(\tau\otimes \chi_T, 2)$ for all $\zeta_i$, in particular  when setting $\zeta_1=\zeta_2=0$. We refer the reader to \cite[pp. 1020]{CaiFriedbergGinzburgKaplan2019} for more details. 
\end{remark}

The unique model $\mathcal{W}( \tau\otimes \chi_T, 2, \psi)$ corresponding to the linear functionals satisfying \eqref{eq-local-Speh-linear-functional}
consists of complexed-valued functions $W_\xi$ on $\mathrm{GL}_4(F)$ of the form 
\begin{equation*}
    W_\xi(g)=\Lambda( \Delta(\tau\otimes \chi_T, 2)(g)\xi)
\end{equation*} where $\xi\in \Delta(\tau\otimes \chi_T, 2)$. To simplify notation, we will also write $g\cdot \xi$ for $\Delta(\tau\otimes \chi_T, 2)(g)\xi$.

Let $\psi_{2T}$ be the character on $U_{(2^2)}(F)=\left\{ \left(\begin{smallmatrix} I_2 & Y \\&I_2\end{smallmatrix}\right)\in \mathrm{GL}_4(F)\right\}$ defined by
$
\psi_{2T} \left(\begin{smallmatrix} I_2 & Y \\&I_2\end{smallmatrix}\right) =\psi(\mathrm{tr}(2TY)). $
We also have a unique model $\mathcal{W}( \tau\otimes \chi_T, 2, \psi_{2T})\subset  \mathrm{Ind}_{U_{(2^2)}(F)}^{\mathrm{GL}_4(F)}(\psi_{2T})$, corresponding to the linear functional
\begin{equation}
\xi  \mapsto  \int\limits_{U_{(2^2)}(F)} \xi(w_{2,2} u)\psi_{2T}^{-1}(u)du.
\label{eq-local-Speh-linear-functional-explicit-GL4-psi-2T}
\end{equation}
More generally, for any $g_1, g_2\in \mathrm{GL}_2(F)$, we define the character $\psi_{2g_1 T g_2^{-1}}$ on $U_{(2^2)}(F)$ by
\begin{equation*}
   \psi_{2g_1 T g_2^{-1}}\left(\begin{smallmatrix}I_2 &Y\\ &I_2\end{smallmatrix}\right)=\psi(\mathrm{tr}(2g_1 T g_2^{-1}Y)),
\end{equation*}
and we get a unique model $\mathcal{W}( \tau\otimes \chi_T, 2, \psi_{2g_1 T g_2^{-1}})$ corresponding to the character  $\psi_{2g_1 T g_2^{-1}}$ and the linear functional
\begin{equation}
\label{eq-local-Speh-linear-functional-explicit-GL4-psi-g1g2}
 \xi  \mapsto  \int\limits_{U_{(2^2)}(F)} \xi( w_{2,2}  u)\psi_{2g_1 T g_2^{-1}}^{-1}(u)du.
\end{equation}
The integrals \eqref{eq-local-Speh-linear-functional-explicit-GL4-psi-2T} and \eqref{eq-local-Speh-linear-functional-explicit-GL4-psi-g1g2} are understood in a similar way as \eqref{eq-local-Speh-linear-functional-explicit-GL4-psi}, see Remark~\ref{remark-Jacquet-integral}.

We have the following result. 

\begin{lemma}
Let $W_\xi\in \mathcal{W}(\tau\otimes\chi_T, 2, \psi_{2T})$ be the function corresponding to $\xi\in \Delta(\tau\otimes\chi_T, 2)$ and the linear functional given by \eqref{eq-local-Speh-linear-functional-explicit-GL4-psi-2T}. Let $g_1, g_2\in \mathrm{GL}_2(F)$. There exists some $W^{g_1,g_2}_\xi\in  \mathcal{W}(\tau\otimes\chi_T, 2, \psi_{2g_2^{-1}T g_1 })$ corresponding to $\xi$
such that for any $g\in \mathrm{GL}_4(F)$, we have
\begin{equation*}
W_\xi\left( \left(\begin{smallmatrix}g_1 &\\
 &g_2\end{smallmatrix}\right) g    \right)=\left \vert \frac{\det(g_1)}{\det(g_2)} \right \vert  \chi_1(\det(g_2)) \chi_T(\det(g_2)) \chi_2(\det(g_1)) \chi_T(\det(g_1)) W^{g_1, g_2}_\xi(g).
\end{equation*}
Moreover, if $W_\xi$ is unramified (i.e., $\xi$ is unramified), then $W_\xi^{g_1, g_2}$ is also unramified. 
\label{lemma-local-Speh-model-equivariance-GL4-different-characters}
\end{lemma}

\begin{proof}
We realize $W_\xi$ through \eqref{eq-local-Speh-linear-functional-explicit-GL4-psi-2T}, i.e., $W_\xi(g)=\Lambda(g \cdot \xi)$ where $\Lambda$ is the linear functional defined in \eqref{eq-local-Speh-linear-functional-explicit-GL4-psi-2T} (we omit the twisting parameters in our following computation). We have
\begin{equation*}
\begin{split}
W_\xi\left( \left(\begin{smallmatrix}g_1 &\\ &g_2\end{smallmatrix}\right) g    \right)   &        =\int\limits_{\mathrm{Mat}_{2}(F)} g\cdot \xi \left( w_{2,2}  \left(\begin{smallmatrix}I_2&Y \\&I_2\end{smallmatrix}\right) \left(\begin{smallmatrix}g_1 &\\&g_2\end{smallmatrix}\right)  \right)\psi^{-1}(\mathrm{tr}(2TY))dY \\
  &             =\int\limits_{\mathrm{Mat}_{2}(F)} g\cdot \xi \left( w_{2,2} \left(\begin{smallmatrix}g_1 &\\&g_2\end{smallmatrix}\right) \left(\begin{smallmatrix}I_2&g_1^{-1}Y g_2\\ &I_2\end{smallmatrix}\right)   \right)\psi^{-1}(\mathrm{tr}(2TY))dY\\
 &             =\int\limits_{\mathrm{Mat}_{2}(F)} g\cdot \xi \left( \left(\begin{smallmatrix}g_2 &\\&g_1\end{smallmatrix}\right) w_{2,2} \left(\begin{smallmatrix}I_2&g_1^{-1}Y g_2\\ &I_2\end{smallmatrix}\right)   \right)\psi^{-1}(\mathrm{tr}(2TY))dY.
\end{split}
\end{equation*}
By a change of variable $Y\mapsto g_1Yg_2^{-1}$, $dY\mapsto  \vert\det(g_1)\vert^2 \vert\det(g_2)\vert^{-2}dY$, we obtain 
\begin{equation*}
\begin{split}
W_\xi\left( \left(\begin{smallmatrix}g_1 &\\ &g_2\end{smallmatrix}\right) g    \right)        =   \left \vert \frac{\det(g_1)}{\det(g_2)} \right \vert^2    \int\limits_{\mathrm{Mat}_{2}(F)} g\cdot \xi \left( \left(\begin{smallmatrix}g_2 &\\&g_1\end{smallmatrix}\right) w_{2,2} \left(\begin{smallmatrix}I_2&Y\\ &I_2\end{smallmatrix}\right)   \right)\psi^{-1}(\mathrm{tr}(2Tg_1 Yg_2^{-1}))dY .
\end{split}
\end{equation*}
Since $g\cdot \xi\in \Delta(\tau\otimes \chi_T, 2)$, we have
\begin{multline*}
g\cdot \xi \left( \left(\begin{smallmatrix}g_2 &\\&g_1\end{smallmatrix}\right)  w_{2,2} \left(\begin{smallmatrix}I_2&Y\\ &I_2\end{smallmatrix}\right)   \right) \\
= \chi_1(\det(g_2))\chi_T(\det(g_2)) \chi_2(\det(g_1))\chi_T(\det(g_1))   \left \vert \frac{\det(g_2)}{\det(g_1)} \right \vert g\cdot \xi \left(  w_{2,2} \left(\begin{smallmatrix}I_2&Y\\ &I_2\end{smallmatrix}\right)   \right).
\end{multline*}
Now take
$$
W_\xi^{g_1, g_2}(g)= \int\limits_{\mathrm{Mat}_{2}(F)} g\cdot \xi \left( w_{2,2} \left(\begin{smallmatrix}I_2&Y \\&I_2\end{smallmatrix}\right)   \right)\psi^{-1}(\mathrm{tr}(2 g_2^{-1} Tg_1 Y))dY .
$$
The result follows. 
\end{proof}

We realize $\mathrm{Ind}_{P_8(F)}^{\mathrm{Sp}_8(F)}(\mathcal{W}( \tau\otimes \chi_T, 2, \psi_{2T})\vert \det \vert ^s)$ in the space  of smooth functions $f_{\mathcal{W}(\tau\otimes \chi_T, 2, \psi_{2T}), s}$ on $\mathrm{Sp}_8(F)$ which take values in the model $\mathcal{W}(\tau\otimes \chi_T, 2,\psi_{2T})$, so we may regard $f_{\mathcal{W}(\tau\otimes \chi_T, 2, \psi_{2T}), s}(m,g)$ as a function of two variables where $m\in \mathrm{GL}_4(F)$, $g\in \mathrm{Sp}_8(F)$,   satisfying
$$
f_{\mathcal{W}(\tau\otimes \chi_T, 2, \psi_{2T}), s}(m, \hat{a} ug)= \delta_{P_8}^{\frac{1}{2}}(\hat{a})   \vert\det(a)\vert^{s} f_{\mathcal{W}(\tau\otimes \chi_T, 2, \psi_{2T}), s}(ma, g),
$$
where $\hat{a}=\mathrm{diag}(a, a^*)\in M_8(F), u\in N_8(F)$.
Here $\delta_{P_8}(\hat{a})=\vert \det(a)\vert^{5}$.
To ease notation, we also re-denote
$
f_{\mathcal{W}(\tau\otimes \chi_T, 2, \psi_{2T}), s}(g)=f_{\mathcal{W}(\tau\otimes \chi_T, 2, \psi_{2T}), s}(I_4, g).
$
Let $f^0_{\mathcal{W}(\tau\otimes \chi_T, 2, \psi_{2T}), s}$ be the unramified section in $\mathrm{Ind}_{P_8(F)}^{\mathrm{Sp}_8(F)}(\mathcal{W}( \tau\otimes \chi_T, 2, \psi_{2T})\vert \det \vert^s)$, such that its value at identity is 
\begin{equation}
\label{eq-local-unramified-section-Sp8-section-at-identity}
    f^0_{\mathcal{W}(\tau\otimes \chi_T, 2, \psi_{2T}), s}(I_8)=\frac{d_{\tau}^{\mathrm{Sp}_{16}}(s)}{d_{\tau\otimes\chi_T}^{\mathrm{Sp}_{8}}(s)},
\end{equation}
where
\begin{equation*}
         d_{\tau}^{\mathrm{Sp}_{16}}(s) =   L( s+\frac{5}{2}, \tau)     \prod_{1\le j \le 2} L(2s+2j, \chi_\tau)L(2s+2j-1, \tau, \mathrm{Sym}^2), 
\end{equation*}
\begin{equation*}
           d_{\tau\otimes \chi_T}^{\mathrm{Sp}_{8}}(s) =L(s+\frac{3}{2}, \tau\otimes\chi_T)L(2s+2, \chi_\tau) L(2s+1, \tau\otimes\chi_T, \mathrm{Sym}^2).
\end{equation*}
This is a meromorphic function in $s$.
Then we have
\begin{equation}
\begin{split}
f^0_{\mathcal{W}(\tau\otimes \chi_T, 2, \psi_{2T}), s}\left(    m u g  \right)
=  \vert \det(a_m)\vert^{s+\frac{5}{2}} f^0_{\mathcal{W}(\tau\otimes \chi_T, 2, \psi_{2T}), s}(a_m, g)      
\label{eq-local-unramified-section-Sp8-property1}
\end{split}
\end{equation}
where $m=\left(\begin{smallmatrix}
a_m &\\
&a_m^*
\end{smallmatrix}\right)\in M_8(F), u\in N_8(F), g\in \mathrm{Sp}_8(F).$
We also have
\begin{equation}
\begin{split}
f^0_{\mathcal{W}(\tau\otimes \chi_T, 2, \psi_{2T}), s}\left(    \left(\begin{smallmatrix}
I_2 &Y & &\\
 & I_2 & &\\
& &I_2 &-J_2 {}^t Y J_2\\
&& & I_2
\end{smallmatrix}\right) g  \right)=\psi(\mathrm{tr}(2 TY))f^0_{\mathcal{W}(\tau\otimes \chi_T, 2, \psi_{2T}), s}\left(   g \right) 
\label{eq-local-unramified-section-Sp8-property3}
\end{split}
\end{equation}
for all $Y\in \mathrm{Mat}_2(F), g\in \mathrm{Sp}_8(F)$, due to the equivariance property of elements in $\mathcal{W}( \tau\otimes \chi_T, 2, \psi_{2T})$. Set 
 \begin{equation} 
 \label{eq-local-unramified-section-Sp8-section-normalized}
 f^*_{\mathcal{W}(\tau\otimes \chi_{T}, 2, \psi_{2T}), s}(g )= d_{\tau\otimes\chi_T}^{\mathrm{Sp}_8}(s) f^0_{\mathcal{W}(\tau\otimes \chi_{T}, 2, \psi_{2T}), s}(g).
 \end{equation}

As a consequence of Lemma \ref{lemma-local-Speh-model-equivariance-GL4-different-characters}, we have the following relation between unramified sections belonging to different induced representations. 
\begin{lemma}
Let $g_1, g_2\in \mathrm{GL}_2(F)$. Then there exists an unramified section $f^0_{\mathcal{W}(\tau\otimes \chi_T, 2, \psi_{2g_2^{-1}Tg_1}), s}$ in $\mathrm{Ind}_{P_8(F)}^{\mathrm{Sp}_8(F)}(\mathcal{W}(\tau\otimes\chi_T, 2, \psi_{2g_2^{-1}Tg_1})\vert \det \vert^s )$, which is determined by the section $f^0_{\mathcal{W}(\tau\otimes \chi_T, 2, \psi_{2T}), s}$, such that, 
for any
 $g\in \mathrm{Sp}_8(F)$, we have
\begin{equation}
\begin{split}
&f^0_{\mathcal{W}(\tau\otimes \chi_T, 2, \psi_{2T}), s} \left(  \left(\begin{smallmatrix}
g_1 &&&\\
&g_2 && \\
&&g_2^*& \\
&&&g_1^*
\end{smallmatrix}\right)g \right)=  \vert \det(g_1)\vert^{s+\frac{7}{2}} \vert \det(g_2)\vert^{s+\frac{3}{2}} \\
& \ \ \ \ \chi_1(\det(g_2))\chi_2(\det(g_1))  \chi_T(\det(g_1 g_2))  f^0_{\mathcal{W}(\tau\otimes \chi_T, 2, \psi_{2g_2^{-1}Tg_1}), s}\left(  g \right).
\end{split}
\label{eq-lemma-local-unramified-section-Sp8-property5}
\end{equation}
\label{lemma-local-unramified-section-Sp8-property5}
\end{lemma}

\begin{proof}
We have
\begin{equation*}
\begin{split}
                    f^0_{\mathcal{W}(\tau\otimes \chi_T, 2, \psi_{2T}), s}\left(  \left(\begin{smallmatrix}
g_1 &&&\\
&g_2 && \\
&&g_2^*& \\
&&&g_1^*
\end{smallmatrix}\right)g\right) 
 =&                        \delta_{P_8}^{\frac{1}{2}}\left( \left(\begin{smallmatrix}
g_1 &&&\\
&g_2 && \\
&&g_2^*& \\
&&&g_1^*
\end{smallmatrix}\right) \right)  \left \vert\det \left(\begin{smallmatrix}
g_1&\\
&g_2 \end{smallmatrix}\right)  \right \vert^s f^0_{\mathcal{W}(\tau\otimes \chi_T, 2, \psi_{2T}), s} \left(  \left(\begin{smallmatrix}
g_1 &\\
&g_2 \end{smallmatrix}\right), g \right)\\
=& \vert \det(g_1g_2)\vert^{s+\frac{5}{2}}  f^0_{\mathcal{W}(\tau\otimes \chi_T, 2, \psi_{2T}), s} \left( \left(\begin{smallmatrix}
g_1&\\
&g_2 \end{smallmatrix}\right) , g\right).
\end{split}
\end{equation*}
Then by Lemma  \ref{lemma-local-Speh-model-equivariance-GL4-different-characters}, we see that 
\begin{equation*}
\begin{split}
 f^0_{\mathcal{W}(\tau\otimes \chi_T, 2, \psi_{2T}), s} \left( \left(\begin{smallmatrix}
g_1&\\
&g_2 \end{smallmatrix}\right), g \right) 
=       \left \vert \frac{\det(g_1)}{\det(g_2)} \right \vert  \chi_1(\det(g_2)) \chi_2(\det(g_1))   \chi_T(\det(g_1 g_2))  f^0_{\mathcal{W}(\tau\otimes \chi_T, 2, \psi_{2g_2^{-1}Tg_1}), s}(g),
\end{split}
\end{equation*}
for an unramified section $f^0_{\mathcal{W}(\tau\otimes \chi_T, 2, \psi_{2g_2^{-1}Tg_1}), s}$ in $\mathrm{Ind}_{P_8(F)}^{\mathrm{Sp}_8(F)}(\mathcal{W}(\tau\otimes\chi_T, 2, \psi_{2g_2^{-1}Tg_1})\vert \det \vert^s )$. Then \eqref{eq-lemma-local-unramified-section-Sp8-property5} follows. 
\end{proof}

\subsection{Local generalized doubling integral for $\mathrm{Sp}_4\times\mathrm{GL}_2$} 
\label{subsection-generalized-doubling-integral}
We recall the explicit determination of the local $L$-function $L(s, \pi\times \tau)$ from  \cite{CaiFriedbergGinzburgKaplan2019}. This will be a major tool for the unramified computation. 
We introduce more notations. 

Let $t:\mathrm{Sp}_4(F)\times \mathrm{Sp}_4(F) \to \mathrm{Sp}_{16}(F)$ be the embedding given by
\begin{equation*}
(g_1, g_2) \mapsto t(g_1, g_2)= \left(\begin{smallmatrix}
g_1 & & & & \\
      &g_{1,1} & &g_{1,2} & \\
      & &g_2 & &\\
      &g_{1,3} & &g_{1,4} & \\
      & & & &J_4 {}^t g_1^{-1} J_4
\end{smallmatrix}\right)
\label{eq-Sp4xSp4 embedding}
\end{equation*}
where $g_1=\left(\begin{smallmatrix}
g_{1,1} & g_{1,2}\\
g_{1,3} & g_{1,4}
\end{smallmatrix}\right)\in \mathrm{Sp}_4(F), g_2\in \mathrm{Sp}_4(F)$. In particular, for $g\in \mathrm{Sp}_4(F)$,
we have 
$
t(1, g)= \left(\begin{smallmatrix}
I_6 &&\\
&g&\\
&&I_6
\end{smallmatrix}\right).
$
For $g\in \mathrm{Sp}_4(F)$, we denote
$
    g^\iota=\left(\begin{smallmatrix}&I_2\\I_2 &\end{smallmatrix}\right) g \left(\begin{smallmatrix}&I_2\\I_2 &\end{smallmatrix}\right).
$
Note that if we write $g=\left(\begin{smallmatrix}
g_{1} & g_{2}\\
g_{3} & g_{4}
\end{smallmatrix}\right)\in \mathrm{Sp}_4(F)$, then  $g^\iota=\left(\begin{smallmatrix}g_{4} & g_{3}\\
g_{2}&g_{1}\end{smallmatrix}\right)$. An elementary matrix calculation shows that for $g_1, g_2\in \mathrm{Sp}_4(F)$, we have
$
t(1, g_1 g_2) t(g_1^\iota , 1) =t(g_1^\iota, g_1) t(1, g_2).
$

We let $\delta\in \mathrm{Sp}_{16}(F)$ be 
\begin{equation*}
\begin{split}
\delta=\left(\begin{smallmatrix}
0 & I_4 &0&0\\
0 &0&0& I_4\\
-I_4&0&0&0\\
0&I_4 &I_4&0
\end{smallmatrix}\right).
\end{split}
\end{equation*}
Let $N_{4,16}^0(F)$ be the subgroup of $N_{4,16}(F)$ given by
\begin{equation}
\begin{split}
&N_{4,16}^0(F)=\left\{     \left( \begin{smallmatrix} I_4 & x & &z\\ & I_4& &\\&&I_4 & x^* \\&&&I_4\end{smallmatrix}\right)\right\} 
\subset   N_{4,16}(F)=\left\{   
 \left( \begin{smallmatrix} I_4 & x &y &z\\ & I_4& &y^\prime\\&&I_4 & x^\prime \\&&&I_4\end{smallmatrix}\right)   \in \mathrm{Sp}_{16}(F)     \right\},
\end{split}
\label{eq-unipotent-subgroup-Sp16}
\end{equation}
realizing the quotient $\delta^{-1}P_{16}(F)\delta\cap N_{4,16}(F)\backslash N_{4,16}(F)$.
 The character $\psi_{N_{4,16}}$ on $N_{4,16}(F)$ is defined as 
\begin{equation}
\psi_{N_{4,16}}(u(x, y, z))=\psi(x_{11}+x_{22}+y_{33}+y_{44}), \ \ x=(x_{ij}), y=(y_{ij})\in \mathrm{Mat}_{4}(F),
\end{equation}
where we write $u(x, y, z)\in N_{4,16}(F)$ as in \eqref{eq-unipotent-subgroup-Sp16}.

Let $\Delta(\tau, 4)$ be the local Speh representation of $\mathrm{GL}_8(F)$, given by
\begin{equation}
\Delta(\tau, 4)=\mathrm{Ind}_{P_{(4^2)}(F)}^{\mathrm{GL}_8(F)}(\chi_1 \otimes \chi_2),
\label{eq-local-Speh-representation-GL8}
\end{equation}
which affords a unique model $\mathcal{W}(\tau, 4, \psi^{-1})$ with respect to the unipotent subgroup $U_{(4^2)}(F)=\left\{ \left(\begin{smallmatrix} I_4 & x \\&I_4\end{smallmatrix}\right)\in \mathrm{GL}_8(F)\right\}$ and the character on $U_{(4^2)}(F)$ given by 
$
    \left(\begin{smallmatrix}I_4 &x\\&I_4\end{smallmatrix}\right)\mapsto \psi^{-1}(\mathrm{tr}(x)).
$
For $\xi\in \Delta(\tau, 4)$, we let $W_\xi\in \mathcal{W}(\tau, 4, \psi^{-1})$ be the function on $\mathrm{GL}_8(F)$ associated with the vector $\xi$, hence it satisfies the  equivariance property
$$
W_\xi\left( \left(\begin{smallmatrix}
I_4 & x\\
&I_4
\end{smallmatrix}\right) g  \right)=\psi^{-1}(\mathrm{tr}(x))W_\xi(g).
$$

We let $f_{\mathcal{W}(\tau, 4, \psi^{-1}), s}^0\in \mathrm{Ind}_{P_{16}(F)}^{\mathrm{Sp}_{16}(F)}(\mathcal{W}(\tau, 4, \psi^{-1})\vert \det \vert^s)$ be the unramified section, which is normalized, such that, for $a\in \mathrm{GL}_8(F)$, we have
$
    f_{\mathcal{W}(\tau, 4, \psi^{-1}), s}^0(a;I_{16})=W_{\Delta(\tau, 4)}^0(a),
$
where $W_{\Delta(\tau, 4)}^0$ is the unique unramified function in $\mathcal{W}(\tau, 4, \psi^{-1})$ such that its value at the identity is 1. We re-denote $f_{\mathcal{W}(\tau, 4, \psi^{-1}), s}^0(g)=f_{\mathcal{W}(\tau, 4, \psi^{-1}), s}^0(I_{8};g)$.
The section $f_{\mathcal{W}(\tau, 4, \psi^{-1}), s}^0$ satisfies the property that
\begin{equation}
\begin{split}
f_{\mathcal{W}(\tau, 4, \psi^{-1}), s}^0\left( \left(\begin{smallmatrix}I_4 & x & &\\
&I_4 &&\\
&&I_4 &-J_4 {}^t x J_4\\
&&&I_4\end{smallmatrix}\right) g\right) =\psi^{-1}(\mathrm{tr}(x)) f_{\mathcal{W}(\tau, 4, \psi^{-1}), s}^0\left( g\right).
\end{split}
\label{eq-local-unramified-Sp16-equivariant-property3}
\end{equation}
Note that  (see \cite[Proposition 24]{CaiFriedbergGinzburgKaplan2019} or \cite[Lemma 1.1]{CaiFriedbergKaplan2022}) for any $g_0\in \mathrm{SL}_4(F)$ and $g\in \mathrm{Sp}_{16}(F)$, we have
 \begin{equation}
  f^0_{\mathcal{W}(\tau, 4, \psi^{-1}), s}\left(  \mathrm{diag}(g_0, g_0, g_0^*, g_0^*) g
   \right)   =   f^0_{\mathcal{W}(\tau, 4, \psi^{-1}), s}\left( g \right).
 \label{eq-local-unramified-Sp16-equivariant-property2}
\end{equation}
Moreover, for any $h_0\in \mathrm{Sp}_4(F)$, we have
\begin{multline}
\int\limits_{N_{4,16}^0(F)}   f^0_{\mathcal{W}(\tau, 4, \psi^{-1}), s}\left( 
    \delta u_0 t\left(h_0 g, h_0^{\iota} h \right) \right)\psi^{-1}_{N_{4,16}}(u_0)du_0   
    =  \int\limits_{N_{4,16}^0(F)}   f^0_{\mathcal{W}(\tau, 4, \psi^{-1}), s}\left( \delta u_0 t\left( g, h \right) \right)\psi^{-1}_{N_{4,16}}(u_0)du_0 ;
 \label{eq-local-unramified-Sp16-equivariant-property1}
\end{multline}
see \cite[p. 1029]{CaiFriedbergGinzburgKaplan2019} or \cite[(2.7)]{GinzburgSoudry2020}. This equality is valid in the domain where both integrals are absolutely convergent and in general by meromorphic continuation. 
Denote
\begin{equation}
\label{eq-local-unramified-section-Sp16-section-normalized}
    f_{\mathcal{W}(\tau, 4, \psi^{-1}), s}^*(g)=d_{\tau}^{\mathrm{Sp}_{16}}(s)f_{\mathcal{W}(\tau, 4, \psi^{-1}), s}^0(g).
\end{equation}
Then by \eqref{eq-local-unramified-section-Sp8-section-normalized} and \eqref{eq-local-unramified-section-Sp16-section-normalized} we have
\begin{equation}
\label{eq-local-unramified-sections-equal-at-identity}
    f_{\mathcal{W}(\tau, 4, \psi^{-1}), s}^*(I_{16})=f^*_{\mathcal{W}(\tau\otimes \chi_{T}, 2, \psi_{2T}), s}(I_8 ).
\end{equation}

We write down the unramified local integral of the generalized doubling method.

\begin{theorem}(Cai-Friedberg-Ginzburg-Kaplan \cite[Theorem 29]{CaiFriedbergGinzburgKaplan2019})
Let $\omega^0_{\pi}$ be the unramified matrix coefficient of $\pi$, normalized so that $\omega^0_{\pi}(I_4)=1$. Then for $\mathrm{Re}(s)\gg 0$, we have
\begin{equation}
\int\limits_{\mathrm{Sp}_4(F)} \int\limits_{N_{4,16}^0(F)} \omega^0_{\pi}(h)   f^*_{\mathcal{W}(\tau, 4, \psi^{-1}), s}(\delta u_0 t(1, h) )\psi_{N_{4,16}}^{-1}(u_0)du_0      dh
= L(s+\frac{1}{2}, \pi\times \tau) .
\label{eq-Local L function - Twisted Doubling}
\end{equation}
\label{thm-Local L-function-Twisted Doubling}
\end{theorem}

\begin{remark}
We refer the reader to \cite[Proposition 4.8]{GinzburgSoudry2021} for more versions of explicit local integrals, all of which can be deduced from \cite[Theorem 29]{CaiFriedbergGinzburgKaplan2019}. 
\end{remark}

\subsection{Proof of the unramified computation}

In this subsection, we prove the following unramified computation. 
All forthcoming manipulations are justified for $\mathrm{Re}(s)\gg 0$.

\begin{theorem}[Unramified computation]
Suppose that ${\pi}$ is an irreducible unramified representation of $\mathrm{Sp}_4({F})$, ${\tau}$ is an irreducible unramified generic representation of  $\mathrm{GL}_2({F})$, $\Phi^0=\textbf{1}_{\mathrm{Mat}_2(\mathcal{O}_{F})}$ is the characteristic function of $\mathrm{Mat}_2(\mathcal{O}_{F})$, and $f^*_{\mathcal{W}({\tau} \otimes\chi_{T}, 2, \psi_{2T}), s}$ is the normalized unramified  section in $\mathrm{Ind}_{P_8({F})}^{\mathrm{Sp}_8({F})}(\mathcal{W}( {\tau}\otimes \chi_{T}, 2, \psi_{2T})\vert \det \vert^s)$ defined in \eqref{eq-local-unramified-section-Sp8-section-normalized}.
Let $v_0\in V_{{\pi}}$ be a non-zero unramified vector in $V_{{\pi}}$. Let $l_{T}:V_{\pi}\to \mathbb{C}$ be a linear functional on $V_{{\pi}}$ such that 
\begin{equation}
l_{T}\left( {\pi} \left(\begin{smallmatrix}
I_2 &z\\
& I_2
\end{smallmatrix}\right) v \right)={\psi}^{-1}(\mathrm{tr}(Tz))l_{T}(v), \  \  \text{ for all } v\in V_{{\pi}}, z\in \mathrm{Mat}_2^0({F}).
\label{eq-l_T}
\end{equation}
We denote
\begin{equation*}
\begin{split}
    \mathcal{Z}^*(l_{T}, s) =  \int\limits_{N_4({F})\backslash \mathrm{Sp}_4({F})}   \int\limits_{N_{2,8}^0({F})}  l_{T}({\pi}(h)v_0)    \omega_{{\psi}}\left( \alpha_T(v)(1,h)\right) )\Phi^0(I_2 )
       f^*_{\mathcal{W}({\tau} \otimes\chi_{T}, 2, \psi_{2T}), s}( \gamma v (1, h) )    dv     dh .
\end{split}
\end{equation*}
The integral converges absolutely for $\mathrm{Re}(s)$ sufficiently large (depending on $\pi_v$, $\tau_v$ only). 
In this domain,  we have
\begin{equation}
\begin{split}
 \mathcal{Z}^*(l_{T}, s)        = L(s+\frac{1}{2},  {\pi}\times {\tau}) \cdot l_{T}(v_0)  .
\label{eq-local-unramified-identity}
\end{split}
\end{equation}
\label{theorem-unramified-computation}
\end{theorem}

The main idea of the proof is to compare the integral $\mathcal{Z}^*(l_{T}, s)$ with a reformulated version (see the left-hand side of \eqref{eq-Local-L-function-linear-functional}) of the local unramified integral from the generalized doubling method in \eqref{eq-Local L function - Twisted Doubling}. To compare these two integrals, we derive simpler formulas for both integrals.

Recall that we have the unramified local integral \eqref{eq-Local L function - Twisted Doubling} of the generalized doubling method, which represents the local $L$-function $L(s+\frac{1}{2}, \pi\times\tau)$. 
In the following lemma we show that we may replace the matrix coefficient $\omega^0_{\pi}(h)$ on the left-hand side of \eqref{eq-Local L function - Twisted Doubling} by $l_{T}(\pi(h)v_0)$, to get the following result.

\begin{lemma}
Let $v_0$ be an unramified vector in $V_\pi$ and let $l_T$ be any linear functional on $V_\pi$ satisfying \eqref{eq-l_T}. Then for $\mathrm{Re}(s)\gg 0$, 
\begin{equation}
\int\limits_{\mathrm{Sp}_4(F)}  \int\limits_{N_{4, 16}^0(F)} l_T(\pi(h) v_0)  f^*_{\mathcal{W}(\tau, 4, \psi^{-1}), s}(\delta u_0 t(1, h))\psi_{N_{4,16}}^{-1}(u_0)du_0      dh    =   L( s+\frac{1}{2}, \pi\times \tau)   \cdot l_T(v_0) .
\label{eq-Local-L-function-linear-functional}
\end{equation}
\label{lemma-Local-L-function-linear-functional}
\end{lemma}

\begin{proof}
Since $v_0\in V_\pi$ is unramified, the function
\begin{equation}
\label{eq-lemma-Local-L-function-linear-functional}
h\mapsto \int\limits_{\mathrm{Sp}_4(\mathcal{O}_F)} l_T(\pi(kh) v_0)dk
\end{equation}
is a function on $\mathrm{Sp}_4(F)$ which is bi-invariant under $\mathrm{Sp}_4(\mathcal{O}_F)$,
and as such is a constant multiple of $\omega_\pi^0$. Note that the value of the function \eqref{eq-lemma-Local-L-function-linear-functional} at identity is $l_T(v_0)$. It follows from the normalization of $\omega_\pi^0$ that
\begin{equation*}
 \int\limits_{\mathrm{Sp}_4(\mathcal{O}_F)} l_T(\pi(kh) v_0)dk=l_T(v_0)\omega_{\pi}^0(h).
\end{equation*}
Note that the function 
\begin{equation*}
h\mapsto     \int\limits_{N_{4, 16}^0(F)}   f^*_{\mathcal{W}(\tau, 4, \psi^{-1}), s}(\delta u_0 t(1, h))\psi_{N_{4,16}}^{-1}(u_0)du_0  
\end{equation*}
is bi-invariant under $\mathrm{Sp}_4(\mathcal{O}_F)$, thus we obtain
\begin{equation*}
\begin{split}
&\int\limits_{\mathrm{Sp}_4(F)}  \int\limits_{N_{4, 16}^0(F)} l_T(\pi(h) v_0)  f^*_{\mathcal{W}(\tau, 4, \psi^{-1}), s}(\delta u_0 t(1, h))\psi_{N_{4,16}}^{-1}(u_0)du_0      dh  \\
=& \int\limits_{  \mathrm{Sp}_4(\mathcal{O}_F) \setminus \mathrm{Sp}_4(F) }  \int\limits_{\mathrm{Sp}_4(\mathcal{O}_F)}  l_T(\pi( k h) v_0) dk   \int\limits_{N_{4, 16}^0(F)}    f^*_{\mathcal{W}(\tau, 4, \psi^{-1}), s}(\delta u_0 t(1, h))\psi_{N_{4,16}}^{-1}(u_0)du_0      dh \\
=&l_T(v_0)  \int\limits_{\mathrm{Sp}_4(F)}  \int\limits_{N_{4, 16}^0(F)} \omega_\pi^0(h)  f^*_{\mathcal{W}(\tau, 4, \psi^{-1}), s}(\delta u_0 t(1, h))\psi_{N_{4,16}}^{-1}(u_0)du_0      dh \\
=& l_T(v_0)  \cdot L( s+\frac{1}{2}, \pi\times \tau). 
\end{split}
\end{equation*}
where we have used Theorem~\ref{thm-Local L-function-Twisted Doubling}  in the last equality.
\end{proof}

Now we prove the unramified computation in Theorem~\ref{theorem-unramified-computation}. 

\begin{proof}[Proof of Theorem~\ref{theorem-unramified-computation}]
By comparing \eqref{eq-local-unramified-identity} with \eqref{eq-Local-L-function-linear-functional}, it suffices to show that, for $\mathrm{Re}(s)$ sufficiently large,
\begin{equation}
\mathcal{Z}^*(l_T, s)=\int\limits_{\mathrm{Sp}_4(F)} \int\limits_{N_{4,16}^0(F)} l_T(\pi(h)v_0)   f^*_{\mathcal{W}(\tau, 4, \psi^{-1}), s}(\delta u_0 t(1, h) )\psi_{N_{4,16}}^{-1}(u_0)du_0      dh  .
\label{eq-Local-L-function-bridge-between-two-identities}
\end{equation}

We first consider the right-hand side of \eqref{eq-Local-L-function-bridge-between-two-identities}.
By the Iwasawa decomposition,  we see that the right-hand side of \eqref{eq-Local-L-function-bridge-between-two-identities} is equal to 
 \begin{equation*}
 \begin{split}
    \int\limits_{\mathrm{Sp}_4(\mathcal{O}_F)} \int\limits_{M_4(F)} \int\limits_{N_4(F)}      \int\limits_{N_{4,16}^0(F)}  l_T(\pi(umk) v_0) f^*_{\mathcal{W}(\tau, 4, \psi^{-1}), s}(\delta u_0 t(1, umk))\psi^{-1}_{N_{4,16}}(u_0)du_0   \delta_{P_4}^{-1}(m)    dudmdk.
 \end{split}
 \end{equation*}
Since $v_0$ and $f^*_{\mathcal{W}(\tau, 4, \psi^{-1}), s}$ are unramified, the $dk$-integration evaluates to 1, so we obtain
 \begin{equation}
   \int\limits_{M_4(F)} \int\limits_{N_4(F)}      \int\limits_{N_{4,16}^0(F)}   l_T(\pi(um) v_0) f^*_{\mathcal{W}(\tau, 4, \psi^{-1}), s}(\delta u_0 t(1, um))\psi^{-1}_{N_{4,16}}(u_0)du_0   \delta_{P_4}^{-1}(m)    dudm.
 \end{equation}
 Using \eqref{eq-l_T}, we see that  the right-hand side of \eqref{eq-Local-L-function-bridge-between-two-identities} is equal to
\begin{equation}
\begin{split}
     \int\limits_{\mathrm{GL}_2(F)}           \int\limits_{\mathrm{Mat}_2^0(F)}     \int\limits_{N_{4,16}^0(F)}   l_T\left( \pi   (\hat{a})  v_0 \right)     f^*_{\mathcal{W}(\tau, 4, \psi^{-1}), s}\left(     \delta u_0 t\left( 1, \left(\begin{smallmatrix}
I_2 & z\\
& I_2
\end{smallmatrix}\right)\hat{a}\right) \right)   \psi^{-1}_{N_{4,16}}(u_0)  \psi^{-1}(\mathrm{tr}(Tz))  \vert \det(a)\vert^{-3} du_0     dz  da .
\end{split}
\label{eq-integral-l_T-simplification1}
\end{equation}
Here, $\hat{a}=\left(\begin{smallmatrix}
a & \\
& J_2{}^t a^{-1} J_2
\end{smallmatrix}\right)$.

Using \eqref{eq-local-unramified-Sp16-equivariant-property1}, we see that the inner integration in \eqref{eq-integral-l_T-simplification1} becomes
\begin{equation*}
\begin{split}
&      \int\limits_{N_{4,16}^0(F)}     f^*_{\mathcal{W}(\tau, 4, \psi^{-1}), s}\left(     \delta u_0 t\left( 1, \left(\begin{smallmatrix}
I_2 & z\\
& I_2
\end{smallmatrix}\right)\hat{a}\right) \right)\psi^{-1}_{N_{4,16}}(u_0) du_0   \\
=&    \int\limits_{N_{4,16}^0(F)}     f^*_{\mathcal{W}(\tau, 4, \psi^{-1}), s}\left(     \delta u_0 t\left( \left(\begin{smallmatrix}
I_2 & \\
-z& I_2
\end{smallmatrix}\right) , \hat{a}\right) \right)\psi^{-1}_{N_{4,16}}(u_0) du_0    \\
=&    \int\limits_{N_{4,16}^0(F)}     f^*_{\mathcal{W}(\tau, 4, \psi^{-1}), s}\left(     \delta u_0 \delta^{-1} \delta t\left( \left(\begin{smallmatrix}
I_2 & \\
-z& I_2
\end{smallmatrix}\right) , \hat{a}\right) \delta^{-1} \right)\psi^{-1}_{N_{4,16}}(u_0) du_0 \\
=& \int\limits_{\overline{N}_{4,16}^0(F)}  
f^*_{\mathcal{W}(\tau, 4, \psi^{-1}), s}\left(  
\overline{u}_0(A, 0, C)
\left(\begin{smallmatrix}
I_2 &&&&\\
&a& &&\\
&&I_8 & &\\
&& & a^* & \\
& & &&I_2
\end{smallmatrix}\right)
\left(\begin{smallmatrix}
I_2 &&&&&&&\\
&I_2&&&&&&\\
&&I_2 &&&&&\\
&&z&I_2&&&&\\
&&&&I_2&&&\\
&&&&-z&I_2 &&\\
J_2{}^t a J_2-I_2&0&&&& & I_2 & \\
-z&a-I_2&&&&&&I_2
\end{smallmatrix}\right)
\right)  
\psi(\mathrm{tr}(A_1)) d\overline{u}_0
\end{split}
\end{equation*}
where
\begin{equation*}
\overline{N}_{4,16}^0(F)= \delta N_{4,16}(F) \delta^{-1}=  \left\{ \overline{u}_0(A,0, C)=\left(\begin{smallmatrix}
I_2 &&&&&&&\\
&I_2&&&&&&\\
&&I_2 &&&&&\\
&&&I_2&&&&\\
A_1&A_2&c_1&c_2&I_2&&&\\
A_3&A_4&c_3&c_1^*&&I_2 &&\\
0&0&A_4^*&A_2^*&& & I_2 & \\
0&0& A_3^*&A_1^*&&&&I_2
\end{smallmatrix}\right)
  \right\}.
\end{equation*}
The above integral is absolutely convergent for $\mathrm{Re}(s)$ sufficiently large.

Let $z\in \mathrm{Mat}_2^0(F)$ be fixed. We claim that the function on $a$ given by
\begin{equation}
\begin{split}
 \int\limits_{\overline{N}_{4,16}^0(F)} 
 f^*_{\mathcal{W}(\tau, 4, \psi^{-1}), s}\left(  
\overline{u}_0(A, 0, C)
\left(\begin{smallmatrix}
I_2 &&&&\\
&a& &&\\
&&I_8 & &\\
&& & a^* & \\
& & &&I_2
\end{smallmatrix}\right)
\left(\begin{smallmatrix}
I_2 &&&&&&&\\
&I_2&&&&&&\\
&&I_2 &&&&&\\
&&z&I_2&&&&\\
&&&&I_2&&&\\
&&&&-z&I_2 &&\\
J_2{}^t a J_2-I_2&0&&&& & I_2 & \\
-z&a-I_2&&&&&&I_2
\end{smallmatrix}\right)
\right) 
 \psi(\mathrm{tr}(A_1)) d\overline{u}_0
\end{split}
\label{eq-local-unramified-Sp16-inner-integral-support-in-a}
\end{equation}
is supported in $\mathrm{GL}_2(F)\cap \mathrm{Mat}_2(\mathcal{O}_F)$. To prove this, we use the right translations by the following elements fixing the section $f^*_{\mathcal{W}(\tau, 4, \psi^{-1}), s}$, for all $r\in \mathrm{Mat}_2(\mathcal{O}_F)$, 
\begin{equation*}
\begin{split}
\left(\begin{smallmatrix}
I_2 &&&&&&&\\
&I_2&&r&&&&\\
&&I_2 &&&&&\\
&&&I_2&&&&\\
&&&&I_2&&-J_2{}^t r J_2&\\
&&&&&I_2 &&\\
&&&&& & I_2 & \\
&&&&&&&I_2
\end{smallmatrix}\right).
\end{split}
\end{equation*}
We conjugate the above matrix to the left and obtain
\begin{equation*}
\begin{split}
\left(\begin{smallmatrix}
I_2 &&&&&&&\\
&I_2&-arz&ar&&&&\\
&&I_2 &&&&&\\
&&&I_2&&&&\\
&&&&I_2&&-J_2{}^t r J_2&\\
&&&&&I_2 &J_2 {}^t (arz) J_2&\\
&&&&& & I_2 & \\
&&&&&&&I_2
\end{smallmatrix}\right),
\end{split}
\end{equation*}
which contributes $\psi^{-1}(\mathrm{tr}(ar))$. The change of variables in $\overline{u}_0$ contributes $\psi^{-1}(\mathrm{tr}(ar))\psi(\mathrm{tr}(r))=\psi^{-1}(\mathrm{tr}(ar))$. Thus, in order for \eqref{eq-local-unramified-Sp16-inner-integral-support-in-a} to be non-vanishing, we must have $\psi^{-1}(\mathrm{tr}(2ar))=1$ for all $r\in \mathrm{Mat}_2^0(F)$, and hence we must have $a\in \mathrm{Mat}_2(\mathcal{O}_F)$. This shows that the function \eqref{eq-local-unramified-Sp16-inner-integral-support-in-a} is supported in $\mathrm{GL}_2(F)\cap \mathrm{Mat}_2(\mathcal{O}_F)$.

Note that for $a\in \mathrm{GL}_2(F)\cap \mathrm{Mat}_2(\mathcal{O}_F)$, the inner integration in \eqref{eq-integral-l_T-simplification1} equals
\begin{equation*}
\begin{split}
   \int\limits_{\overline{N}_{4,16}^0(F)}  
f^*_{\mathcal{W}(\tau, 4, \psi^{-1}), s}\left(  
\overline{u}_0(A, 0, C)
\left(\begin{smallmatrix}
I_2 &&&&\\
&a& &&\\
&&I_8 & &\\
&& & a^* & \\
& & &&I_2
\end{smallmatrix}\right)
\left(\begin{smallmatrix}
I_2 &&&&&&&\\
&I_2&&&&&&\\
&&I_2 &&&&&\\
&&z&I_2&&&&\\
&&&&I_2&&&\\
&&&&-z&I_2 &&\\
 &&&&& & I_2 & \\
-z& &&&&&&I_2
\end{smallmatrix}\right)
\right) 
\psi(\mathrm{tr}(A_1)) d\overline{u}_0 .
\end{split}
\end{equation*}
For fixed $a\in \mathrm{GL}_2(F)\cap \mathrm{Mat}_2(\mathcal{O}_F)$ and $z\in \mathrm{Mat}_2^0(F)$, we consider the following function on $A_2\in \mathrm{Mat}_2(F)$,
\begin{equation}
\begin{split}
 \int\limits  
f^*_{\mathcal{W}(\tau, 4, \psi^{-1}), s}\left(  
\left(\begin{smallmatrix}
I_2 &&&&&&&\\
&I_2&&&&&&\\
&&I_2 &&&&&\\
&&&I_2&&&&\\
A_1&A_2&c_1&c_2&I_2&&&\\
A_3&A_4&c_3&c_1^*&&I_2 &&\\
0&0&A_4^*&A_2^*&& & I_2 & \\
0&0& A_3^*&A_1^*&&&&I_2
\end{smallmatrix}\right)
\left(\begin{smallmatrix}
I_2 &&&&\\
&a& &&\\
&&I_8 & &\\
&& & a^* & \\
& & &&I_2
\end{smallmatrix}\right)  
\left(\begin{smallmatrix}
I_2 &&&&&&&\\
&I_2&&&&&&\\
&&I_2 &&&&&\\
&&z&I_2&&&&\\
&&&&I_2&&&\\
&&&&-z&I_2 &&\\
 &&&&& & I_2 & \\
-z& &&&&&&I_2
\end{smallmatrix}\right)
\right) 
\psi(\mathrm{tr}(A_1)) dA_1 dA_3 dA_4 dC.
\label{eq-local-unramified-Sp16-inner-integral-support-in-a_2}
\end{split}
\end{equation}
The above integral is absolutely convergent for $\mathrm{Re}(s)$ sufficiently large.
This time, we use the left translations by the following elements fixing the section $f^*_{\mathcal{W}(\tau, 4, \psi^{-1}), s}$ according to \eqref{eq-local-unramified-Sp16-equivariant-property2}, for all $r\in \mathrm{Mat}_2(\mathcal{O}_F)$, 
\begin{equation*}
\begin{split}
\left(\begin{smallmatrix}
I_2 &&&&&&&\\
ar&I_2&&&&&&\\
&&I_2 &&&&&\\
&&ar&I_2&&&&\\
&&&&I_2&&&\\
&&&&-J_2{}^t (ar) J_2&I_2 &&\\
&&&&& & I_2 & \\
&&&&&&-J_2{}^t (ar) J_2&I_2
\end{smallmatrix}\right).
\end{split}
\end{equation*}
We conjugate the above matrices to the right, to obtain matrices of the form
\begin{equation*}
\begin{split}
\left(\begin{smallmatrix}
I_2 &&&&&&&\\
r&I_2&&&&&&\\
&&I_2 &&&&&\\
&&ar&I_2&&&&\\
&&&&I_2&&&\\
&&&&-J_2{}^t (ar) J_2&I_2 &&\\
&&&&& & I_2 & \\
&&&&&&-J_2{}^t r J_2&I_2
\end{smallmatrix}\right),
\end{split}
\end{equation*}
which belong to $\mathrm{Sp}_{16}(\mathcal{O}_F)$ and hence fixing the section $f^*_{\mathcal{W}(\tau, 4, \psi^{-1}), s}$ because it is unramified.
The change of variables in the coordinates $(A_1, A_3, A_4, C)$ produces a factor $\psi(\mathrm{tr}(A_2 ar))$. Therefore, if \eqref{eq-local-unramified-Sp16-inner-integral-support-in-a_2} is non-vanishing, then $A_2 a\in \mathrm{Mat}_2(\mathcal{O}_F)$. Using this, we see that 
\begin{equation*}
\begin{split}
&    \int\limits_{\overline{N}_{4,16}^0(F)}  
f^*_{\mathcal{W}(\tau, 4, \psi^{-1}), s}\left(  
\left(\begin{smallmatrix}
I_2 &&&&&&&\\
&I_2&&&&&&\\
&&I_2 &&&&&\\
&&&I_2&&&&\\
A_1&A_2&c_1&c_2&I_2&&&\\
A_3&A_4&c_3&c_1^*&&I_2 &&\\
0&0&A_4^*&A_2^*&& & I_2 & \\
0&0& A_3^*&A_1^*&&&&I_2
\end{smallmatrix}\right)
\left(\begin{smallmatrix}
I_2 &&&&\\
&a& &&\\
&&I_8 & &\\
&& & a^* & \\
& & &&I_2
\end{smallmatrix}\right)  
\left(\begin{smallmatrix}
I_2 &&&&&&&\\
&I_2&&&&&&\\
&&I_2 &&&&&\\
&&z&I_2&&&&\\
&&&&I_2&&&\\
&&&&-z&I_2 &&\\
 &&&&& & I_2 & \\
-z& &&&&&&I_2
\end{smallmatrix}\right)
\right)  \psi(\mathrm{tr}(A_1)) dAdC \\
=&    \int\limits_{\overline{N}_{4,16}^0(F)}  
f^*_{\mathcal{W}(\tau, 4, \psi^{-1}), s}\left(  
\left(\begin{smallmatrix}
I_2 &&&&&&&\\
&I_2&&&&&&\\
&&I_2 &&&&&\\
&&&I_2&&&&\\
A_1&0&c_1&c_2&I_2&&&\\
A_3&A_4&c_3&c_1^*&&I_2 &&\\
0&0&A_4^*&0&& & I_2 & \\
0&0& A_3^*&A_1^*&&&&I_2
\end{smallmatrix}\right)
\left(\begin{smallmatrix}
I_2 &&&&&&&\\
&a&&&&&&\\
&&I_2 &&&&&\\
&&z&I_2&&&&\\
&&&&I_2&&&\\
&&&&-z&I_2 &&\\
 &&&&& & a^* & \\
-z& &&&&&&I_2
\end{smallmatrix}\right)  \right.\\
& \ \ \ \ \ \ \left.
\left(\begin{smallmatrix}
I_2 &&&&&&&\\
&I_2&&&&&&\\
&&I_2 &&&&&\\
&&&I_2&&&&\\
0&A_2 a&0&0&I_2&&&\\
0&0&0&0&&I_2 &&\\
0&0&0&(A_2 a)^*&& & I_2 & \\
0&0& 0&0&&&&I_2
\end{smallmatrix}\right)
\right) 
\psi(\mathrm{tr}(A_1)) dAdC \\
=&  \vert \det(a)\vert^{-2}     \int
f^*_{\mathcal{W}(\tau, 4, \psi^{-1}), s}\left(  
\left(\begin{smallmatrix}
I_2 &&&&&&&\\
&I_2&&&&&&\\
&&I_2 &&&&&\\
&&&I_2&&&&\\
A_1&0&c_1&c_2&I_2&&&\\
A_3&A_4&c_3&c_1^*&&I_2 &&\\
0&0&A_4^*&0&& & I_2 & \\
0&0& A_3^*&A_1^*&&&&I_2
\end{smallmatrix}\right) 
\left(\begin{smallmatrix}
I_2 &&&&&&&\\
&a&&&&&&\\
&&I_2 &&&&&\\
&&z&I_2&&&&\\
&&&&I_2&&&\\
&&&&-z&I_2 &&\\
 &&&&& & a^* & \\
-z& &&&&&&I_2
\end{smallmatrix}\right)
\right)  
\psi(\mathrm{tr}(A_1)) dA_1 dA_3 dA_4 dC.
\end{split}
\end{equation*}
Thus, for $\mathrm{Re}(s)\gg 0$, the right-hand side of \eqref{eq-Local-L-function-bridge-between-two-identities} is equal to
\begin{equation}
\begin{split}
                   \int\limits_{\mathrm{GL}_2(F)\cap \mathrm{Mat}_2(\mathcal{O}_F)} \int\limits_{\mathrm{Mat}_2^0(F)}      \int\limits_{\overline{N}^1_{4,16}(F)}   &l_T\left( \pi   (\hat{a})  v_0 \right)   
                   f^*_{\mathcal{W}(\tau, 4, \psi^{-1}), s}\left(  
\overline{u}_1 \cdot
\left(\begin{smallmatrix}
I_2 &&&&&&&\\
&a&&&&&&\\
&&I_2 &&&&&\\
&&z&I_2&&&&\\
&&&&I_2&&&\\
&&&&-z&I_2 &&\\
 &&&&& & a^* & \\
-z& &&&&&&I_2
\end{smallmatrix}\right)
\right) 
\\
& \ \ \      \psi(\mathrm{tr}(A_1))   \psi^{-1}(\mathrm{tr}(Tz)) \vert \det(a)\vert^{-5}  
 d\overline{u}_1  dz  da, 
\end{split}
\label{eq-integral-l_T-simplification2}
\end{equation}
where 
\begin{equation*}
\overline{N}^1_{4,16}(F)=    \left\{ 
 \overline{u}_1=   \left(\begin{smallmatrix}
I_2 &&&&&&&\\
&I_2&&&&&&\\
&&I_2 &&&&&\\
&&&I_2&&&&\\
A_1&0&c_1&c_2&I_2&&&\\
A_3&A_4&c_3&c_1^*&&I_2 &&\\
0&0&A_4^*&0&& & I_2 & \\
0&0& A_3^*&A_1^*&&&&I_2
\end{smallmatrix}\right) \in \overline{N}_{4,16}^0(F)
\right\}.
\end{equation*}

Next, we fix $a\in \mathrm{GL}_2(F)\cap \mathrm{Mat}_2(\mathcal{O}_F)$, and consider the following function on $z$,
\begin{equation}
\begin{split}
\int\limits_{\overline{N}^1_{4,16}(F)}             f^*_{\mathcal{W}(\tau, 4, \psi^{-1}), s}\left(  
\left(\begin{smallmatrix}
I_2 &&&&&&&\\
&I_2&&&&&&\\
&&I_2 &&&&&\\
&&&I_2&&&&\\
A_1&0&c_1&c_2&I_2&&&\\
A_3&A_4&c_3&c_1^*&&I_2 &&\\
0&0&A_4^*&0&& & I_2 & \\
0&0& A_3^*&A_1^*&&&&I_2
\end{smallmatrix}\right) 
\left(\begin{smallmatrix}
I_2 &&&&&&&\\
&a&&&&&&\\
&&I_2 &&&&&\\
&&z&I_2&&&&\\
&&&&I_2&&&\\
&&&&-z&I_2 &&\\
 &&&&& & a^* & \\
-z& &&&&&&I_2
\end{smallmatrix}\right)
\right)   \psi(\mathrm{tr}(A_1))   
d\overline{u}_1.
\end{split}
\label{eq-integral-l_T-simplification--3}
\end{equation}
The above integral is absolutely convergent for $\mathrm{Re}(s)$ sufficiently large.
We use the right translations by the following elements fixing the section $f^*_{\mathcal{W}(\tau, 4, \psi^{-1}), s}$, 
for all $r\in \mathrm{Mat}_2(\mathcal{O}_F)$, 
\begin{equation*}
\begin{split}
\left(\begin{smallmatrix}
I_2 &&&r&&&&\\
&I_2&&&&&&\\
&&I_2 &&&&&\\
&&&I_2&&&&\\
&&&&I_2&&&-J_2{}^t r J_2\\
&&&&&I_2 &&\\
&&&&& & I_2 & \\
&&&&&&&I_2
\end{smallmatrix}\right).
\end{split}
\end{equation*} 
We conjugate the above matrices to the left and obtain
\begin{equation*}
\begin{split}
\left(\begin{smallmatrix}
I_2 &&-rz&r&&&&\\
&I_2&&&&&&\\
&&I_2 &&&&&\\
&&&I_2&&&&\\
&&&&I_2&&&-J_2{}^t r J_2\\
&&&&&I_2 &&J_2{}^t(rz)J_2\\
&&&&& & I_2 & \\
&&&&&&&I_2
\end{smallmatrix}\right),
\end{split}
\end{equation*}
which contributes $\psi(\mathrm{tr}(rz))$. The change of variables in $A_1$ contributes $\psi(\mathrm{tr}(rz))$ as well. In order for \eqref{eq-integral-l_T-simplification--3} to be non-vanishing, we must have $\psi(\mathrm{tr}(2rz))=1$ for all $r\in \mathrm{Mat}_2(\mathcal{O}_F)$.
We see that the function \eqref{eq-integral-l_T-simplification--3} is supported in $z\in \mathrm{Mat}_2^0(\mathcal{O}_F)$, and hence the $dz$-integration in \eqref{eq-integral-l_T-simplification2} evaluates to 1. Thus, 
the right-hand side of \eqref{eq-Local-L-function-bridge-between-two-identities} is equal to
\begin{equation}
\begin{split}
                   \int\limits_{\mathrm{GL}_2(F)\cap \mathrm{Mat}_2(\mathcal{O}_F)}      \int\limits_{\overline{N}^1_{4,16}(F)}   l_T\left( \pi   (\hat{a})  v_0 \right)  f^*_{\mathcal{W}(\tau, 4, \psi^{-1}), s}\left( 
\overline{u}_1 \cdot
\left(\begin{smallmatrix}
I_2 &&&&&&&\\
&a&&&&&&\\
&&I_2 &&&&&\\
&&&I_2&&&&\\
&&&&I_2&&&\\
&&&&&I_2 &&\\
 &&&&& & a^* & \\
& &&&&&&I_2
\end{smallmatrix}\right)
\right) 
 \psi(\mathrm{tr}(A_1))    \vert \det(a)\vert^{-5}  
 d\overline{u}_1    da.
\end{split}
\label{eq-integral-l_T-simplification--4}
\end{equation}

Next, we consider the inner integration in \eqref{eq-integral-l_T-simplification--4}. For $g= \left(\begin{smallmatrix}
 g_1 & g_2 \\
 g_3 & g_4\end{smallmatrix}\right) \in \mathrm{Sp}_4(F)$ with $g_i\in \mathrm{Mat}_2(F)$, we define 
\begin{equation*}
\begin{split}
\lambda(f^*_{\mathcal{W}(\tau, 4, \psi^{-1}), s})(g)=\int\limits_{\overline{N}^1_{4,16}(F)}  f^*_{\mathcal{W}(\tau, 4, \psi^{-1}), s}\left(
\left(\begin{smallmatrix}
I_2 &&&&&&&\\
&I_2&&&&&&\\
&&I_2 &&&&&\\
&&&I_2&&&&\\
A_1&0&c_1&c_2&I_2&&&\\
A_3&A_4&c_3&c_1^*&&I_2 &&\\
0&0&A_4^*&0&& & I_2 & \\
0&0& A_3^*&A_1^*&&&&I_2
\end{smallmatrix}\right) 
\left(\begin{smallmatrix}
I_2 &&&&&&&\\
&g_1&&&&&g_2&\\
&&I_2 &&&&&\\
&&&I_2&&&&\\
&&&&I_2&&&\\
&&&&&I_2 &&\\
 &g_3&&&& & g_4 & \\
& &&&&&&I_2
\end{smallmatrix}\right)
\right) 
  \psi(\mathrm{tr}(A_1))  d\overline{u}_1.
\end{split}
\end{equation*}
The above integral defining $\lambda(f^*_{\mathcal{W}(\tau, 4, \psi^{-1}), s})(g) $ is absolutely convergent for $\mathrm{Re}(s)\gg 0$.
We caution the reader that $f^*_{\mathcal{W}(\tau, 4, \psi^{-1}), s}$ is a function on $\mathrm{Sp}_{16}(F)$, while $\lambda(f^*_{\mathcal{W}(\tau, 4, \psi^{-1}), s})$ is a function on $\mathrm{Sp}_4(F)$ defined above.
We claim that $\lambda(f^*_{\mathcal{W}(\tau, 4, \psi^{-1}), s})(g)$ is an unramified section of the induced representation 
$
     \mathrm{Ind}_{P_4(F)}^{\mathrm{Sp}_4(F)} ( \chi_1(\det) \vert \det \vert^{s+3}).
$  
To prove this, we need to consider a matrix of the form $\left(\begin{smallmatrix}
a &b\\
&a^*
\end{smallmatrix}\right)\in P_4(F)$ and show appropriate equivariance properties.
 First, we check that $\lambda(f^*_{\mathcal{W}(\tau, 4, \psi^{-1}), s})(g)$ is left-invariant under  $\left(\begin{smallmatrix}
I_2 &b\\
&I_2
\end{smallmatrix}\right)$, by conjugating the matrix 
\begin{equation*}
\left(\begin{smallmatrix}
I_2 &&&&&&&\\
& I_2&&&&&b&\\
&&I_2 &&&&&\\
&&&I_2&&&&\\
&&&&I_2&&&\\
&&&&&I_2 &&\\
 & &&&& &I_2 & \\
& &&&&&&I_2
\end{smallmatrix}\right)
\end{equation*}
to the left and do a change of variable in the coordinates $(c_1, c_2, c_3)$. Second, we plug in $g_0=\left(\begin{smallmatrix}
a &\\
&a^*
\end{smallmatrix}\right)$ where $a\in \mathrm{GL}_2(F)$, to $\lambda(f^*_{\mathcal{W}(\tau, 4, \psi^{-1}), s})(g_0 g)$, and conjugate the embedding of $g_0$ into $\mathrm{Sp}_{16}(F)$ to the left in the integral defining $\lambda(f^*_{\mathcal{W}(\tau, 4, \psi^{-1}), s})$. We obtain $\vert \det(a)\vert^{-2}$ from the change of variable in $A_4$. 
We also obtain a factor of 
$\chi_1(\det(a))|\det(a)|^{2} |\det(a)|^{s+\frac{9}{2}}$ 
from the section $f^*_{\mathcal{W}(\tau, 4, \psi^{-1}), s}$. This is explained as follows.  For any $h\in \mathrm{Sp}_{16}(F)$, we have
\begin{equation*}
\begin{split}
f^*_{\mathcal{W}(\tau, 4, \psi^{-1}), s}\left(
\left(\begin{smallmatrix}
I_2 &&&&&&&\\
&a&&&&&&\\
&&I_2 &&&&&\\
&&&I_2&&&&\\
&&&&I_2&&&\\
&&&&&I_2 &&\\
 &&&&& & a^* & \\
& &&&&&&I_2
\end{smallmatrix}\right) h
\right)  
=&  |\det(a)|^{s+\frac{9}{2}} \Delta(\tau, 4) \left(\begin{smallmatrix}
I_2 &&&\\
&a&&\\
&&I_2 &\\
&&&I_2\\
\end{smallmatrix}\right) f^*_{\mathcal{W}(\tau, 4, \psi^{-1}), s}\left(h\right)  \\
=& |\det(a)|^{s+\frac{9}{2}} |\det(a)|^{2}\chi_1(\det(a)) f^*_{\mathcal{W}(\tau, 4, \psi^{-1}), s}\left(h\right).
\end{split}
\end{equation*}
Here, the second equality follows from the action of the Speh representation $\Delta(\tau, 4)$ (which is given in \eqref{eq-local-Speh-representation-GL8}). Thus, $\lambda(f^*_{\mathcal{W}(\tau, 4, \psi^{-1}), s})\in \mathrm{Ind}_{P_4(F)}^{\mathrm{Sp}_4(F)} ( \chi_1(\det) \vert \det \vert^{s+3}).$ 
Finally, $\lambda(f^*_{\mathcal{W}(\tau, 4, \psi^{-1}), s})(gk)=\lambda(f^*_{\mathcal{W}(\tau, 4, \psi^{-1}), s})(g)$ for any $k\in \mathrm{Sp}_4(\mathcal{O}_F)$ because the section $f^*_{\mathcal{W}(\tau, 4, \psi^{-1}), s}$ is unramified. This proves the claim. 

Next, we claim that 
\begin{equation}
    \lambda(f^*_{\mathcal{W}(\tau, 4, \psi^{-1}), s})\left( 
I_4
\right)= d_{\tau}^{\mathrm{Sp}_{16}}(s).
\label{eq-Sp(16)-evaluation-of-local-unramified-integral}
\end{equation}
This can be verified by matrix multiplication, as follows. Specifically, one can use \eqref{eq-local-unramified-Sp16-equivariant-property3}, \eqref{eq-local-unramified-Sp16-equivariant-property2},  and the right-invariance of $f^*_{\mathcal{W}(\tau, 4, \psi^{-1}), s}$ by $\mathrm{Sp}_{16}(\mathcal{O}_F)$, to show that 
\begin{equation*}
\begin{split}
\lambda(f^*_{\mathcal{W}(\tau, 4, \psi^{-1}), s})\left(
I_4
\right) =     
\int\limits_{\overline{N}^1_{4,16}(\mathcal{O}_F)}    
   f^*_{\mathcal{W}(\tau, 4, \psi^{-1}), s} 
\left(\begin{smallmatrix}
I_2 &&&&&&&\\
&I_2&&&&&&\\
&&I_2 &&&&&\\
&&&I_2&&&&\\
A_1&0&c_1&c_2&I_2&&&\\
A_3&A_4&c_3&c_1^*&&I_2 &&\\
0&0&A_4^*&0&& & I_2 & \\
0&0& A_3^*&A_1^*&&&&I_2
\end{smallmatrix} 
\right) 
  \psi(\mathrm{tr}(A_1))  d\overline{u}_1.
  \end{split}
\end{equation*}
First, we use conjugation by the following elements, for $r\in \mathrm{Mat}_2(\mathcal{O}_F)$, 
\begin{equation*}
\left(\begin{smallmatrix}
I_2 &r&&&&&&\\
& I_2&&&&&&\\
&&I_2 &r&&&&\\
&&&I_2&&&&\\
&&&&I_2&r^*&&\\
&&&&&I_2 &&\\
 & &&&& &I_2 &r^* \\
& &&&&&&I_2
\end{smallmatrix}\right),
\end{equation*}
which produces the character $\psi(\mathrm{tr}(J_2{}^t r J_2 A_3)$, to conclude that 
\begin{equation*}
    A_3\mapsto \int    
   f^*_{\mathcal{W}(\tau, 4, \psi^{-1}), s}  
\left(\begin{smallmatrix}
I_2 &&&&&&&\\
&I_2&&&&&&\\
&&I_2 &&&&&\\
&&&I_2&&&&\\
A_1&0&c_1&c_2&I_2&&&\\
A_3&A_4&c_3&c_1^*&&I_2 &&\\
0&0&A_4^*&0&& & I_2 & \\
0&0& A_3^*&A_1^*&&&&I_2
\end{smallmatrix} 
\right) 
  \psi(\mathrm{tr}(A_1))  dA_1 dA_4 dC
\end{equation*}
is supported in $\mathrm{Mat}_2(\mathcal{O}_F)$. For the variable $A_1$, we consider other unipotent matrices $g_0\in \mathrm{SL}_4(\mathcal{O}_F)$, and use conjugation by $\mathrm{diag}(g_0, g_0, g_0^*, g_0^*)$, to show that  
\begin{equation*}
    A_1\mapsto \int    
   f^*_{\mathcal{W}(\tau, 4, \psi^{-1}), s} 
\left(\begin{smallmatrix}
I_2 &&&&&&&\\
&I_2&&&&&&\\
&&I_2 &&&&&\\
&&&I_2&&&&\\
A_1&0&c_1&c_2&I_2&&&\\
0&A_4&c_3&c_1^*&&I_2 &&\\
0&0&A_4^*&0&& & I_2 & \\
0&0&0 &A_1^*&&&&I_2
\end{smallmatrix} 
\right) 
  \psi(\mathrm{tr}(A_1))   dA_4 dC
\end{equation*}
is supported in $\mathrm{Mat}_2(\mathcal{O}_F)$. Then we consider the variable $A_4$. This time, we use matrices of the form
\begin{equation*}
\left(\begin{smallmatrix}
I_2 &&&&&&r&\\
& I_2&&&&&&r^*\\
&&I_2 &&&&&\\
&&&I_2&&&&\\
&&&&I_2&&&\\
&&&&&I_2 &&\\
 & &&&& &I_2 & \\
& &&&&&&I_2
\end{smallmatrix}\right)
\end{equation*}
where $r\in \mathrm{Mat}_2(\mathcal{O}_F)$, to conclude that 
\begin{equation*}
    A_4\mapsto \int    
   f^*_{\mathcal{W}(\tau, 4, \psi^{-1}), s} 
\left(\begin{smallmatrix}
I_2 &&&&&&&\\
&I_2&&&&&&\\
&&I_2 &&&&&\\
&&&I_2&&&&\\
0&0&c_1&c_2&I_2&&&\\
0&A_4&c_3&c_1^*&&I_2 &&\\
0&0&A_4^*&0&& & I_2 & \\
0&0&0 &0&&&&I_2
\end{smallmatrix} 
\right) 
    dC
\end{equation*}
is supported in $\mathrm{Mat}_2(\mathcal{O}_F)$. The next step is to use matrices of the form
$
    \left(\begin{smallmatrix}
        I_4 & &r &\\
        & I_4 & &r^* \\
        & & I_4 & \\
        &&&I_4
    \end{smallmatrix}\right),
$
for $r\in \mathrm{Mat}_4(\mathcal{O}_F)$, to conclude that 
\begin{equation*}
    C\mapsto 
   f^*_{\mathcal{W}(\tau, 4, \psi^{-1}), s}  
\left(\begin{smallmatrix}
        I_4 & & &\\
        & I_4 & & \\
        & C& I_4 & \\
        &&&I_4
\end{smallmatrix}\right) 
\end{equation*}
is supported in $\mathrm{Mat}_4^0(\mathcal{O}_F)$. Therefore, 
\begin{equation*}
\begin{split}
\lambda(f^*_{\mathcal{W}(\tau, 4, \psi^{-1}), s})\left(
I_4
\right) =     
\int\limits_{\overline{N}^1_{4,16}(\mathcal{O}_F)}    
   f^*_{\mathcal{W}(\tau, 4, \psi^{-1}), s} 
\left(I_{16}
\right) 
   d\overline{u}_1=  f^*_{\mathcal{W}(\tau, 4, \psi^{-1}), s} 
\left(I_{16}
\right) =  d_{\tau}^{\mathrm{Sp}_{16}}(s).
  \end{split}
\end{equation*}
The above second equality follows from the fact that $\mathrm{vol}(\overline{N}^1_{4,16}(\mathcal{O}_F))=1$. This proves \eqref{eq-Sp(16)-evaluation-of-local-unramified-integral}. 
Thus, 
the right-hand side of \eqref{eq-Local-L-function-bridge-between-two-identities} is equal to
\begin{equation}
\begin{split}
                   \int\limits_{\mathrm{GL}_2(F)\cap \mathrm{Mat}_2(\mathcal{O}_F)}        l_T\left( \pi   (\hat{a})  v_0 \right)   
\lambda(f^*_{\mathcal{W}(\tau, 4, \psi^{-1}), s})\left( \begin{smallmatrix} a& \\ & a^*\end{smallmatrix}\right)  \vert \det(a)\vert^{-5}  
     da
\end{split}
\label{eq-integral-l_T-simplification--5}
\end{equation}
where
$\lambda(f^*_{\mathcal{W}(\tau, 4, \psi^{-1}), s})(g)$ is an unramified section of the induced representation 
\begin{equation*}
     \mathrm{Ind}_{P_4(F)}^{\mathrm{Sp}_4(F)} ( \chi_1(\det)\vert \det \vert^{s+3})
 \end{equation*}
and its value at identity is $\lambda(f^*_{\mathcal{W}(\tau, 4, \psi^{-1}), s})(I_4)=d_{\tau}^{\mathrm{Sp}_{16}}(s)$.
Thus, it suffices to prove that, for $\mathrm{Re}(s)$ sufficiently large,
\begin{equation}
\label{eq-Local-L-function-bridge-between-two-identities-2}
\begin{split}
 \mathcal{Z}^*(l_T, s) =                   \int\limits_{\mathrm{GL}_2(F)\cap \mathrm{Mat}_2(\mathcal{O}_F)}        l_T\left( \pi   (\hat{a})  v_0 \right)   
\lambda(f^*_{\mathcal{W}(\tau, 4, \psi^{-1}), s})\left( \begin{smallmatrix} a& \\ & a^*\end{smallmatrix}\right)  \vert \det(a)\vert^{-5}  
     da.
\end{split}
\end{equation}

Now we consider $\mathcal{Z}^*(l_T, s)$. Recall that $\mathcal{Z}^*(l_T, s)$ involves a section $f^*_{\mathcal{W}(\tau \otimes\chi_{T}, 2, \psi_{2T}), s}$ on $\mathrm{Sp}_8(F)$.
Using the Iwasawa decomposition again, we compute that, for $\mathrm{Re}(s)$ sufficiently large,
\begin{align} \label{eq-local-unramified-LHS-reduction1}
	\mathcal{Z}^*(l_T, s)	&=        \int\limits_{N_4(F)\backslash \mathrm{Sp}_4(F)} \int\limits_{N_{2,8}^0(F)}    l_T(\pi(h)v_0)  \omega_{\psi}\left( \alpha_T(v)(1,h)\right) )\Phi^0(I_2 )
f^*_{\mathcal{W}(\tau \otimes\chi_{T}, 2, \psi_{2T}), s}( \gamma v (1, h) )    dv   dh   \nonumber\\
		&=  \int\limits_{\mathrm{GL}_2(F)} l_T(\pi(\hat{a} )v_0) \vert \det(a)\vert^{-3}   \int\limits_{N_{2,8}^0(F)}   
    \omega_{\psi}\left( \alpha_T(v)(1,\hat{a} )\right) )\Phi^0(I_2 )   f^*_{\mathcal{W}(\tau \otimes\chi_{T}, 2, \psi_{2T}), s}( \gamma v (1, \hat{a}  ) )   dv   da.
\end{align}

Recall that 
$$N_{2,8}^0(F)=\left\{v(x, 0, z)=\left(\begin{smallmatrix}
I_2 &x  &0  & z\\
& I_2 & &0  \\
& &I_2 & -J_2{}^t z J_2\\
& & &I_2
\end{smallmatrix}\right)\in N_{2,8}(F): x\in \mathrm{Mat}_2(F), z\in \mathrm{Mat}_2^0(F)  \right\}.$$
According to the action of the Weil representation, we have
\begin{equation*}
\begin{split}
 \omega_{\psi}\left( \alpha_T(v(x,0,z))(1,\hat{a} )\right) )\Phi^0(I_2 )      =\chi_T(\det(a)) \vert \det(a)\vert   \omega_{\psi}\left(   \alpha_T(v(xa,0,z))\right) )\Phi^0(a).
\end{split}
\end{equation*}
Moreover, by direct matrix computation we see that
$v(x,0,z)  \cdot  (1, \hat{a}  ) = (1, \hat{a}) \cdot  v(xa, 0, z).$
Then the inner $dv$-integration of \eqref{eq-local-unramified-LHS-reduction1} equals
\begin{equation*}
\begin{split}
    & \chi_T(\det(a)) \vert \det(a)\vert     \iint    
 \omega_{\psi}\left( \alpha_T(v(xa,0,z)) \right) \Phi^0(a)   
f^*_{\mathcal{W}(\tau \otimes\chi_{T}, 2, \psi_{2T}), s}(    \gamma (1, \hat{a}  )  v(xa,0,z)  )         \psi(\mathrm{tr}(Tz))  dxdz \\
=     & \chi_T(\det(a)) \vert \det(a)\vert^{-1}    \iint    
 \omega_{\psi}\left( \alpha_T(v(x,0,z)) \right) \Phi^0(a)   
f^*_{\mathcal{W}(\tau \otimes\chi_{T}, 2, \psi_{2T}), s}(     \gamma (1, \hat{a}  )   v(x,0,z)  )         \psi(\mathrm{tr}(Tz))  dxdz.
\end{split}
\end{equation*}
Note that
$
\omega_{\psi}  \left( \alpha_T(v(x,0,z)) \right) \Phi^0(a) = \psi  (\mathrm{tr}(Tz)) \Phi^0(a+x).
$
It follows that the inner $dv$-integration of \eqref{eq-local-unramified-LHS-reduction1} equals
\begin{equation*}
\begin{split}
 \chi_T(\det(a)) \vert \det(a)\vert^{-1}    \iint    
 \Phi^0(a+x) f^*_{\mathcal{W}(\tau \otimes\chi_{T}, 2, \psi_{2T}), s}( \gamma (1, \hat{a}  )     v(x,0,z)  )         \psi(\mathrm{tr}(Tz))  dxdz .
 \end{split}
 \end{equation*}
 We apply a change of variable $x\mapsto x-a$ to get
 \begin{equation*}
 \begin{split}
 \chi_T(\det(a)) \vert \det(a)\vert^{-1}    \iint    
  \Phi^0(x) f^*_{\mathcal{W}(\tau \otimes\chi_{T}, 2, \psi_{2T}), s}( \gamma (1, \hat{a}  )    v(x-a,0,z)  )         \psi(\mathrm{tr}(Tz))  dxdz .
\end{split}
\end{equation*}
Since  $\Phi^0=1_{\mathrm{Mat}_2(\mathcal{O}_F)}$ and $ f^*_{\mathcal{W}(\tau \otimes\chi_{T}, 2, \psi_{2T}), s}$ is invariant under right translation by $v(x, 0, 0)$ for $x\in \mathrm{Mat}_2(\mathcal{O}_F)$, the $dx$-integration evaluates to 1. Thus the inner integration of \eqref{eq-local-unramified-LHS-reduction1} equals
\begin{equation*}
\begin{split}
       & \chi_T(\det(a)) \vert \det(a)\vert^{-1}    \int\limits_{\mathrm{Mat}_2^0(F)}    
  f^*_{\mathcal{W}(\tau \otimes\chi_{T}, 2, \psi_{2T}), s}( \gamma (1, \hat{a}  )  \gamma^{-1} \gamma    v(-a,0,z) \gamma^{-1} )         \psi(\mathrm{tr}(Tz))  dz  \\
  =&    \chi_T(\det(a)) \vert \det(a)\vert^{-1}     \int\limits_{\mathrm{Mat}_2^0(F)}     
  f^*_{\mathcal{W}(\tau \otimes\chi_{T}, 2, \psi_{2T}), s} \left(     
  \left(\begin{smallmatrix}
a &&&\\ &I_2 &&\\&&I_2&\\&&&a^*
\end{smallmatrix}\right)
  \left(\begin{smallmatrix}
I_2 &&&\\ &I_2 &&\\-a&z&I_2&\\&-J_2 {}^t a J_2&&I_2
\end{smallmatrix}\right)
  \right)      \psi^{-1}(\mathrm{tr}(Tz))  dz  .
\end{split}
\end{equation*}
Therefore,  for $\mathrm{Re}(s)$ sufficiently large, $\mathcal{Z}^*(l_T, s)$ is equal to
\begin{equation}
\begin{split}
&\int\limits_{\mathrm{GL}_2(F)}  
l_T(\pi(\hat{a} )v_0) \chi_T(\det(a)) \vert \det(a)\vert^{-4}   \int\limits_{\mathrm{Mat}_2^0(F)}\\ 
&         f^*_{\mathcal{W}(\tau \otimes\chi_{T}, 2, \psi_{2T}), s} \left(     
      \left(\begin{smallmatrix}
a &&&\\ &I_2 &&\\&&I_2&\\&&&a^*
\end{smallmatrix}\right)
       \left(\begin{smallmatrix}
I_2 &&&\\ &I_2 &&\\-a&z&I_2&\\&-J_2 {}^t a J_2&&I_2
\end{smallmatrix}\right) \right)  
 \psi^{-1}(\mathrm{tr}(Tz))  dz  da.
 \label{eq-local-unramified-LHS-reduction2}
\end{split}
\end{equation}

Next we look at the inner integration in \eqref{eq-local-unramified-LHS-reduction2}.
We use the relation between unramified sections of different induced representations in Lemma \ref{lemma-local-unramified-section-Sp8-property5}, and apply it to $g_1=I_2, g_2=-\frac{T}{2}$. 
Since $\det(T)\in \mathcal{O}_F^\times$, we have $\vert \det(T/2)\vert=1$  and
$\chi_1(\det(T/2))=1$. Also, since $(\mathcal{O}_F^\times, \mathcal{O}_F^\times)_F=1$ (see \cite[Chapter V, Proposition 3.4]{Neukirch1999}), we have
\begin{equation*}
    \chi_T(\det(T/2))=(\det(T), \det(T/2))_F=\left(\det(T), \frac{\det(T)}{4}\right)_F=1.
\end{equation*}
Thus, 
$
     \vert \det(T/2)\vert^{s+\frac{3}{2}} \chi_1(\det(T/2))\chi_T (\det(T/2) )=1.
$
We obtain that the inner integration in \eqref{eq-local-unramified-LHS-reduction2} is equal to
\begin{equation*}
\begin{split}             
\int\limits_{\mathrm{Mat}_2^0(F)}  
  f^*_{\mathcal{W}(\tau \otimes\chi_{T}, 2, \psi^{-1} ), s} \left(     
  \left(\begin{smallmatrix}
I_2 &&&\\
&-2T^{-1} && \\
&&-\frac{1}{2}J_2{}^t T J_2 & \\
&&&I_2
\end{smallmatrix}\right) 
  \left(\begin{smallmatrix}
a &&&\\ &I_2 &&\\&&I_2&\\&&&a^*
\end{smallmatrix}\right)   \left(\begin{smallmatrix}
I_2 &&&\\ &I_2 &&\\-a&z&I_2&\\&-J_2 {}^t a J_2&&I_2
\end{smallmatrix}\right)
  \right)     \psi^{-1}(\mathrm{tr}(Tz)) dz,
\end{split}
\end{equation*}
where $f^*_{\mathcal{W}(\tau\otimes \chi_T, 2, \psi^{-1}), s}$ is an unramified section in $\mathrm{Ind}_{P_8(F)}^{\mathrm{Sp}_8(F)}(\mathcal{W}(\tau\otimes\chi_T, 2, \psi^{-1})\vert \det \vert^s )$.
Also, if we plug in $g=I_8$ in \eqref{eq-lemma-local-unramified-section-Sp8-property5}, 
we see that 
$
    f^*_{\mathcal{W}(\tau\otimes \chi_T, 2, \psi^{-1}), s}(I_8)=f^*_{\mathcal{W}(\tau \otimes\chi_{T}, 2, \psi_{2T}), s}(I_8).
$
We conjugate the matrix involving $T$ to the right, and do a change of variable in $z$, to get
\begin{equation}
\begin{split}
 \int\limits_{\mathrm{Mat}_2^0(F)}  
  f^*_{\mathcal{W}(\tau \otimes\chi_{T}, 2, \psi^{-1} ), s} \left(     
    \left(\begin{smallmatrix}
a &&&\\ &I_2 &&\\&&I_2&\\&&&a^*
\end{smallmatrix}\right)
  \left(\begin{smallmatrix}
I_2 &&&\\ &I_2 &&\\\frac{1}{2} J_2 {}^t T J_2 a&z&I_2&\\&\frac{1}{2} J_2 {}^t a J_2 T&&I_2
\end{smallmatrix}\right)
  \right)        
   \psi^{-1}(\mathrm{tr}(4zT^{-1})) dz.
    \label{eq-local-unramified-LHS-reduction2-inner}
\end{split}
\end{equation}

We claim that the integral in \eqref{eq-local-unramified-LHS-reduction2-inner} vanishes if $a\not\in \mathrm{Mat}_2(\mathcal{O}_F)$. To see this, we let $r\in \mathrm{Mat}_2(\mathcal{O}_F)$. Then $f^*_{\mathcal{W}(\tau \otimes\chi_{T}, 2, \psi^{-1}), s}=v(r,0,0)\cdot f^*_{\mathcal{W}(\tau \otimes\chi_{T}, 2, \psi^{-1}), s}$ since $f^*_{\mathcal{W}(\tau \otimes\chi_{T}, 2, \psi^{-1}), s}$ is unramified. Conjugating the element $v(r,0,0)$ to the left-hand side, we see that 
\begin{equation*}
\begin{split}
 &\int\limits_{\mathrm{Mat}_2^0(F)}  
  f^*_{\mathcal{W}(\tau \otimes\chi_{T}, 2, \psi^{-1}), s} \left(     
    \left(\begin{smallmatrix}
a &&&\\ &I_2 &&\\&&I_2&\\&&&a^*
\end{smallmatrix}\right)
  \left(\begin{smallmatrix}
I_2 &&&\\ &I_2 &&\\\frac{1}{2} J_2 {}^t T J_2 a&z&I_2&\\&\frac{1}{2} J_2 {}^t a J_2 T&&I_2
\end{smallmatrix}\right) v(r,0,0)
  \right)         
   \psi^{-1}(\mathrm{tr}(4zT^{-1})) dz \\
 =   &           \int\limits_{\mathrm{Mat}_2^0(F)}
  f^*_{\mathcal{W}(\tau \otimes\chi_{T}, 2, \psi^{-1}), s} \left(
v(ar, 0, 0)
    \left(\begin{smallmatrix}
a &&&\\ &I_2 &&\\&&I_2&\\&&&a^*
\end{smallmatrix}\right)
  \left(\begin{smallmatrix}
I_2 &&&\\ &I_2 &&\\\frac{1}{2} J_2 {}^t T J_2 a&z+\frac{1}{2} J_2 ({}^t r {}^t a J_2 T+{}^t T J_2 a r)&I_2&\\&\frac{1}{2} J_2 {}^t a J_2 T&&I_2
\end{smallmatrix}\right) 
   \right)    
   \psi^{-1}(\mathrm{tr}(4zT^{-1})) dz .
\end{split}
\end{equation*}
The element $v(ar, 0, 0)$ contributes $\psi^{-1}(\mathrm{tr}(ar))$, and the change of variable $z\mapsto z-\frac{1}{2} J_2 ({}^t r {}^t a J_2 T+{}^t T J_2 a r)$ contributes $\psi(\mathrm{tr}(4ar))$. Thus, we see that 
\begin{equation*}
\begin{split}
&\int\limits_{\mathrm{Mat}_2^0(F)}  
  f^*_{\mathcal{W}(\tau \otimes\chi_{T}, 2, \psi^{-1}), s} \left(     
    \left(\begin{smallmatrix}
a &&&\\ &I_2 &&\\&&I_2&\\&&&a^*
\end{smallmatrix}\right)
  \left(\begin{smallmatrix}
I_2 &&&\\ &I_2 &&\\\frac{1}{2} J_2 {}^t T J_2 a&z&I_2&\\&\frac{1}{2} J_2 {}^t a J_2 T&&I_2
\end{smallmatrix}\right) 
  \right)        
   \psi^{-1}(\mathrm{tr}(4zT^{-1})) dz \\
=&   \psi(\mathrm{tr}(3ar))     \int\limits_{\mathrm{Mat}_2^0(F)}  
  f^*_{\mathcal{W}(\tau \otimes\chi_{T}, 2, \psi^{-1}), s} \left(     
    \left(\begin{smallmatrix}
a &&&\\ &I_2 &&\\&&I_2&\\&&&a^*
\end{smallmatrix}\right)
  \left(\begin{smallmatrix}
I_2 &&&\\ &I_2 &&\\\frac{1}{2} J_2 {}^t T J_2 a&z&I_2&\\&\frac{1}{2} J_2 {}^t a J_2 T&&I_2
\end{smallmatrix}\right)  
  \right)           
   \psi^{-1}(\mathrm{tr}(4zT^{-1})) dz      
\end{split}
\end{equation*}
 for any $r\in \mathrm{Mat}_2(\mathcal{O}_F)$. Therefore, in order for the integral in \eqref{eq-local-unramified-LHS-reduction2-inner} to be non-vanishing, we must have
  $\psi(\mathrm{tr}(3ar))=1$ for all $r\in \mathrm{Mat}_2(\mathcal{O}_F)$, and consequently we must have $a\in \mathrm{Mat}_2(\mathcal{O}_F)$. 
 Therefore, for $\mathrm{Re}(s)$ sufficiently large, we have
 \begin{equation}
 \label{eq-local-unramified-LHS-reduction3}
\begin{split}
& \mathcal{Z}^*(l_T, s) =   \int\limits_{\mathrm{GL}_2(F)\cap \mathrm{Mat}_2(\mathcal{O}_F)}   l_T(\pi(\hat{a} )v_0)  \chi_T(\det(a)) \vert \det(a)\vert^{-4}  \int\limits_{\mathrm{Mat}_2^0(F)}    \\ 
&    f^*_{\mathcal{W}(\tau \otimes\chi_{T}, 2, \psi^{-1}), s} \left(     
  \left(\begin{smallmatrix}
I_2 &&&\\ &I_2 &&\\ &z&I_2&\\& &&I_2
\end{smallmatrix}\right)  
    \left(\begin{smallmatrix}
a &&&\\ &I_2 &&\\&&I_2&\\&&&a^*
\end{smallmatrix}\right)
  \right) 
    \psi^{-1}(\mathrm{tr}(4zT^{-1})) dz  da.
\end{split}
\end{equation}

For $g= \left(\begin{smallmatrix}
 g_1 & g_2 \\
 g_3 & g_4\end{smallmatrix}\right) \in \mathrm{Sp}_4(F)$ with $g_i\in \mathrm{Mat}_2(F)$, we define $\tilde{\lambda}(f^*_{\mathcal{W}(\tau \otimes\chi_{T}, 2, \psi^{-1}), s})(g)$ to be
\begin{equation*}
\begin{split}
\int\limits_{\mathrm{Mat}_2^0(F)}      f^*_{\mathcal{W}(\tau \otimes\chi_{T}, 2, \psi^{-1}), s} \left(     
  \left(\begin{smallmatrix}
I_2 &&&\\ &I_2 &&\\ &z&I_2&\\& &&I_2
\end{smallmatrix}\right)  
    \left(\begin{smallmatrix}
g_1 &&&g_2\\ &I_2 &&\\&&I_2&\\g_3&&&g_4
\end{smallmatrix}\right)
  \right)     \psi^{-1}(\mathrm{tr}(4zT^{-1})) dz.
\end{split}
\end{equation*}
The above integral is absolutely convergent for $\mathrm{Re}(s)\gg 0$. 
 We claim that $\tilde{\lambda}(f^*_{\mathcal{W}(\tau \otimes\chi_{T}, 2, \psi^{-1}), s})$ is an unramified section of the induced representation 
 $
     \mathrm{Ind}_{P_4(F)}^{\mathrm{Sp}_4(F)} ( \chi_1\chi_T(\det)\vert \det \vert^{s+2}).
$
The proof is similar to that of $\lambda(f^*_{\mathcal{W}(\tau, 4, \psi^{-1}), s})$.
 We consider a matrix of the form $g_0=\left(\begin{smallmatrix}
A & B \\
 & A^*\end{smallmatrix}\right)\in P_4(F)$, where $A\in \mathrm{GL}_2(F)$. First, by conjugating 
 the embedding of $\left(\begin{smallmatrix}
I_2 & B \\
 & I_2\end{smallmatrix}\right)\in N_4(F)$ to the left and using the left-invariance of $f^*_{\mathcal{W}(\tau \otimes\chi_{T}, 2, \psi^{-1}), s}$ under $N_8(F)$, we see 
that $\tilde{\lambda}(f^*_{\mathcal{W}(\tau \otimes\chi_{T}, 2, \psi^{-1}), s})$ is left invariant under $\left(\begin{smallmatrix}
I_2 & B \\
 & I_2\end{smallmatrix}\right)\in N_4(F)$.
Next, we plug in   $g_0=\left(\begin{smallmatrix}
A &  \\
 & A^*\end{smallmatrix}\right)$ and conjugate to the left, to get
 \begin{equation*}
 \begin{split}
 &\tilde{\lambda}(f^*_{\mathcal{W}(\tau \otimes\chi_{T}, 2, \psi^{-1}), s})\left( 
\left(\begin{smallmatrix}
A &  \\
 & A^*\end{smallmatrix}\right)
 g
\right) \\
= & \int\limits_{\mathrm{Mat}_2^0(F)}    f^*_{\mathcal{W}(\tau \otimes\chi_{T}, 2, \psi^{-1}), s} \left(     
   \left(\begin{smallmatrix}
A &&&\\ &I_2 &&\\&&I_2&\\ &&&A^*
\end{smallmatrix}\right)
  \left(\begin{smallmatrix}
I_2 &&&\\ &I_2 &&\\ &z&I_2&\\& &&I_2
\end{smallmatrix}\right)  
    \left(\begin{smallmatrix}
g_1 &&&g_2\\ &I_2 &&\\&&I_2&\\g_3&&&g_4
\end{smallmatrix}\right)
  \right)     \psi^{-1}(\mathrm{tr}(4zT^{-1})) dz \\
 =&   \int\limits_{\mathrm{Mat}_2^0(F)}    \vert \det(A)\vert^{s+\frac{5}{2}} \Delta(\tau\otimes\chi_T, 2)   \left(\begin{smallmatrix}
A &\\ &I_2 
\end{smallmatrix}\right) f^*_{\mathcal{W}(\tau \otimes\chi_{T}, 2, \psi^{-1}), s} \left(     
  \left(\begin{smallmatrix}
I_2 &&&\\ &I_2 &&\\ &z&I_2&\\& &&I_2
\end{smallmatrix}\right)  
    \left(\begin{smallmatrix}
g_1 &&&g_2\\ &I_2 &&\\&&I_2&\\g_3&&&g_4
\end{smallmatrix}\right)
  \right)     \psi^{-1}(\mathrm{tr}(4zT^{-1})) dz \\
  =&  \vert \det(A)\vert^{s+\frac{5}{2}}  \chi_1\chi_T(\det(A))|\det(A)|  \int\limits_{\mathrm{Mat}_2^0(F)}       f^*_{\mathcal{W}(\tau \otimes\chi_{T}, 2, \psi^{-1}), s} \left(     
  \left(\begin{smallmatrix}
I_2 &&&\\ &I_2 &&\\ &z&I_2&\\& &&I_2
\end{smallmatrix}\right)  
    \left(\begin{smallmatrix}
g_1 &&&g_2\\ &I_2 &&\\&&I_2&\\g_3&&&g_4
\end{smallmatrix}\right)
  \right)     \psi^{-1}(\mathrm{tr}(4zT^{-1})) dz.
   \end{split}
 \end{equation*}
 Here, the last equality follows from the action of the Speh representation $\Delta(\tau\otimes\chi_T, 2)$ (see \eqref{eq-local-Speh-representation-GL4}). Thus
 \begin{equation*}
 \tilde{\lambda}(f^*_{\mathcal{W}(\tau \otimes\chi_{T}, 2, \psi^{-1}), s})\in \mathrm{Ind}_{P_4(F)}^{\mathrm{Sp}_4(F)} ( \chi_1\chi_T(\det)\vert \det \vert^{s+2}).
  \end{equation*}
 Finally, since the section $f^*_{\mathcal{W}(\tau \otimes\chi_{T}, 2, \psi^{-1}), s}$ is unramified, we have $\tilde{\lambda}(f^*_{\mathcal{W}(\tau \otimes\chi_{T}, 2, \psi^{-1}), s})\left( 
 gk
\right)=\tilde{\lambda}(f^*_{\mathcal{W}(\tau \otimes\chi_{T}, 2, \psi^{-1}), s})\left( 
 g
\right)$
for any $k\in \mathrm{Sp}_4(\mathcal{O}_F)$. This proves the claim. 

Next, we claim that 
\begin{equation}
\label{eq-local-unramified-computation-lambdatilde-at-identity}
\tilde{\lambda}(f^*_{\mathcal{W}(\tau \otimes\chi_{T}, 2, \psi^{-1}), s})(I_4)=f^*_{\mathcal{W}(\tau \otimes\chi_{T}, 2, \psi^{-1}), s}(I_8).
\end{equation}
Here, $f^*_{\mathcal{W}(\tau \otimes\chi_{T}, 2, \psi^{-1}), s}$ is a function on $\mathrm{Sp}_8(F)$ while $\tilde{\lambda}(f^*_{\mathcal{W}(\tau \otimes\chi_{T}, 2, \psi^{-1}), s})$ is a function on $\mathrm{Sp}_4(F)$ defined through an integral involving $f^*_{\mathcal{W}(\tau \otimes\chi_{T}, 2, \psi^{-1}), s}$. Let $r\in \mathrm{Mat}_2(\mathcal{O}_F)$. Using the identity
\begin{equation*}
     \left(\begin{smallmatrix}
I_2 &&&\\ &I_2 &&\\ &z&I_2&\\& &&I_2
\end{smallmatrix}\right)  
\left(\begin{smallmatrix}
I_2 &&r&\\ &I_2 &&J_2{}^t r J_2\\ &&I_2&\\& &&I_2
\end{smallmatrix}\right) =
\left(\begin{smallmatrix}
I_2 &-rz&r&\\ &I_2 &&J_2{}^t r J_2\\ &&I_2&zJ_2{}^tr J_2\\& &&I_2
\end{smallmatrix}\right)\left(\begin{smallmatrix}
I_2 &&&\\ &I_2 &&\\ &z&I_2&\\& &&I_2
\end{smallmatrix}\right) 
\end{equation*}
and the right $\mathrm{Sp}_8(\mathcal{O}_F)$-invariance of $f^*_{\mathcal{W}(\tau \otimes\chi_{T}, 2, \psi^{-1}), s}$, we have
\begin{equation*}
    \begin{split}
        f^*_{\mathcal{W}(\tau \otimes\chi_{T}, 2, \psi^{-1}), s} \left(     
  \left(\begin{smallmatrix}
I_2 &&&\\ &I_2 &&\\ &z&I_2&\\& &&I_2
\end{smallmatrix}\right)  
  \right)  
                   &= f^*_{\mathcal{W}(\tau \otimes\chi_{T}, 2, \psi^{-1}), s} \left(    
  \left(\begin{smallmatrix}
I_2 &&&\\ &I_2 &&\\ &z&I_2&\\& &&I_2
\end{smallmatrix}\right)  
\left(\begin{smallmatrix}
I_2 &&r&\\ &I_2 &&J_2{}^t r J_2\\ &&I_2&\\& &&I_2
\end{smallmatrix}\right)
  \right)   \\
                &= \psi(\mathrm{tr}(rz)) f^*_{\mathcal{W}(\tau \otimes\chi_{T}, 2, \psi^{-1}), s} \left(   
      \left(\begin{smallmatrix}
I_2 &&&\\ &I_2 &&\\ &z&I_2&\\& &&I_2
\end{smallmatrix}\right) 
       \right) .
    \end{split}
\end{equation*}
This means that if $f^*_{\mathcal{W}(\tau \otimes\chi_{T}, 2, \psi^{-1}), s} \left( 
  \left(\begin{smallmatrix}
I_2 &&&\\ &I_2 &&\\ &z&I_2&\\& &&I_2
\end{smallmatrix}\right)  
  \right)  \not=0$, then $\psi(\mathrm{tr}(rz))=1$ for all $r\in \mathrm{Mat}_2(\mathcal{O}_F)$, and consequently we must have $z\in \mathrm{Mat}_2^0(\mathcal{O}_F)$. So the function
$
      z\mapsto f^*_{\mathcal{W}(\tau \otimes\chi_{T}, 2, \psi^{-1}), s} \left(     
  \left(\begin{smallmatrix}
I_2 &&&\\ &I_2 &&\\ &z&I_2&\\& &&I_2
\end{smallmatrix}\right)  
  \right)
$
 is supported in $\mathrm{Mat}_2^0(\mathcal{O}_F)$. Then \eqref{eq-local-unramified-computation-lambdatilde-at-identity} follows from the above computation.
Therefore,  for $\mathrm{Re}(s)$ sufficiently large, we have
 \begin{equation}
 \label{eq-local-unramified-LHS-reduction4}
\begin{split}
\mathcal{Z}^*(l_T, s) = \int\limits_{\mathrm{GL}_2(F)\cap \mathrm{Mat}_2(\mathcal{O}_F)}    l_T(\pi(\hat{a} )v_0)  \tilde{\lambda}(f^*_{\mathcal{W}(\tau \otimes\chi_{T}, 2, \psi^{-1}), s})
\left(\begin{smallmatrix}a& \\ & a^*\end{smallmatrix}\right) 
     \chi_T(\det(a)) \vert \det(a)\vert^{-4}   da,
\end{split}
\end{equation}
where 
$\tilde{\lambda}(f^*_{\mathcal{W}(\tau \otimes\chi_{T}, 2, \psi^{-1}), s})(g)$ is an unramified section of the induced representation \begin{equation*}
     \mathrm{Ind}_{P_4(F)}^{\mathrm{Sp}_4(F)} ( \chi_1\chi_T(\det)\vert \det \vert^{s+2})
 \end{equation*}
whose value at identity is $\tilde{\lambda}(f^*_{\mathcal{W}(\tau \otimes\chi_{T}, 2, \psi^{-1}), s})(I_{4})=f^*_{\mathcal{W}(\tau \otimes\chi_{T}, 2, \psi^{-1}), s} \left(I_8\right)$.

Examining the two unramified sections $\lambda(f^*_{\mathcal{W}(\tau, 4, \psi^{-1}), s})(g)$ and $\tilde{\lambda}(f^*_{\mathcal{W}(\tau \otimes\chi_{T}, 2, \psi^{-1}), s})(g)$, it follows from \eqref{eq-local-unramified-sections-equal-at-identity} that 
their values at identity agree. Then, for $a\in \mathrm{GL}_2(F)$, we have
\begin{equation*}
\begin{split}
     \frac{\tilde{\lambda}(f^*_{\mathcal{W}(\tau \otimes\chi_{T}, 2, \psi^{-1}), s})
\left(\begin{smallmatrix}a& \\ & a^*\end{smallmatrix}\right)}{\lambda(f^*_{\mathcal{W}(\tau, 4, \psi^{-1}), s})\left(\begin{smallmatrix}a& \\ & a^*\end{smallmatrix}\right)}  =
\frac{\vert \det(a)\vert^{s+\frac{7}{2}} \chi_1(\det(a)) \chi_T(\det(a)) }{\vert \det(a)\vert^{s+\frac{9}{2}} \chi_1(\det(a))}.
\end{split}
\end{equation*}
This means that 
$
    \tilde{\lambda}(f^*_{\mathcal{W}(\tau \otimes\chi_{T}, 2, \psi^{-1}), s})
\left(\begin{smallmatrix}a& \\ & a^*\end{smallmatrix}\right)=\lambda(f^*_{\mathcal{W}(\tau, 4, \psi^{-1}), s})\left(\begin{smallmatrix}a& \\ & a^*\end{smallmatrix}\right) \cdot \vert \det(a)\vert^{-1} \chi_T(\det(a)).
$
Since $\chi_T$ is quadratic, it follows from \eqref{eq-local-unramified-LHS-reduction4} that, for $\mathrm{Re}(s)$ sufficiently large,
\begin{equation*}
\begin{split}
 \mathcal{Z}^*(l_T, s) =                   \int\limits_{\mathrm{GL}_2(F)\cap \mathrm{Mat}_2(\mathcal{O}_F)}        l_T\left( \pi   (\hat{a})  v_0 \right)   
\lambda(f^*_{\mathcal{W}(\tau, 4, \psi^{-1}), s})\left( \begin{smallmatrix} a& \\ & a^*\end{smallmatrix}\right)  \vert \det(a)\vert^{-5}  
     da.
\end{split}
\end{equation*}
This is exactly \eqref{eq-Local-L-function-bridge-between-two-identities-2}. This completes the proof of Theorem~\ref{theorem-unramified-computation}.

\end{proof}

\section{Finite ramified and archimedean non-vanishing}\label{section-Ramified computation}

In this section we consider local integrals at finite ramified places and archimedean places.

The local integral at a finite ramified place can be controlled by the following result.

\begin{proposition} 
Let $F$ be non-archimedean, and let $K_0$ be an open compact subgroup of $\mathrm{Sp}_4(\mathcal{O}_{F})$. There exist a Schwartz function $\Phi_0\in \mathcal{S}(\mathrm{Mat}_2(F))$ and a section $f_{0,s} \in \mathrm{Ind}_{P_8(F)}^{\mathrm{Sp}_8(F)}(\mathcal{W}( \tau\otimes \chi_T, 2, \psi_{2T})\vert \det \vert^s)$  such that for any irreducible admissible representation $\pi_\nu$ of $\mathrm{Sp}_4(F)$, any vector $v_0 \in V_{\pi_\nu}$ which is stabilized by $K_0$, and any choice of the functional $l_T$ satisfying \eqref{Introduction-eq-l_T}, we have
\begin{equation}
\int\limits_{N_4(F)\backslash \mathrm{Sp}_4(F)}     \int\limits_{N_{2,8}^0(F)}   l_T(\pi_\nu(h)v_0)   \omega_{\psi,\nu}\left( \alpha_T(v )(1,h)\right) ) 
 \Phi_0(I_2 )   f_{0, s}(\gamma v (1, h))dv dh =
 l_T(v_0).
\label{eq-ramified-calculation-finiteplace}
\end{equation} 
The integral on the left-hand side of \eqref{eq-ramified-calculation-finiteplace} converges absolutely for all $s$.
\label{proposition-ramified-finiteplace}
\end{proposition}

\begin{proof} 
The proof is similar to \cite[Proposition 6.6]{GinzburgRallisSoudry1998}.

Let $\overline{N}_4(F)=\left\{ \overline{n}_4(w)=
\left( \begin{smallmatrix}
I_2 & \\ w & I_2
\end{smallmatrix}\right): w\in \mathrm{Mat}_2^0(F)
  \right\}$. Then
$N_4(F)M_4(F)\overline{N}_4(F)$ is an open dense subset of $\mathrm{Sp}_4(F)$ whose complement has Haar measure 0. Recall that
\begin{equation*}
\begin{split}
\gamma = \left(\begin{smallmatrix}
&I_2 & &\\&&&-I_2\\I_2&&&\\&&I_2&
\end{smallmatrix}\right).
\end{split}
\end{equation*}
Let
\begin{equation}
\label{eq-local-ramified-N_2,8,0}
   \overline{N}_{2,8}^0(F)=\gamma N_{2,8}^0(F) \gamma^{-1}= \left\{ 
 \overline{v}(x, z)=\left(\begin{smallmatrix} I_2 &&&\\ & I_2&&\\ x & z &I_2 & \\ &x^* &&I_2\end{smallmatrix} \right): z\in \mathrm{Mat}_2^0(F), x\in \mathrm{Mat}_2(F)
  \right\}.
\end{equation}
We write the left-hand side of \eqref{eq-ramified-calculation-finiteplace} as
\begin{equation}
\begin{split}
   &\int\limits_{\mathrm{GL}_2(F)} \int\limits_{\mathrm{Mat}_2^0(F)}    \int\limits_{\overline{N}_{2,8}^0(F)}  l_T\left(\pi\left( \left(\begin{smallmatrix} a & \\ & a^*\end{smallmatrix}\right)  \overline{n}_4(w)   \right)  v_0 \right)   \psi^{-1}(\mathrm{tr}(Tz)) \chi_T(\det(a)) \vert \det(a)\vert^{-2} \\
    &   \omega_{\psi}\left( (1,\overline{n}_4(w) \right) \Phi_0( a+xa)
f_{0, s}
 \left(  
 \left(\begin{smallmatrix} I_2 &&&\\ & I_2&&\\ x & z &I_2 & \\ &x^* &&I_2\end{smallmatrix} \right)    
  \left(\begin{smallmatrix} a &&&\\ & I_2&&\\  &  &I_2 & \\ & &&a^*\end{smallmatrix} \right)  
   \left(\begin{smallmatrix} I_2 &&&\\ & I_2&&\\  &  &I_2 & \\ w& &&I_2\end{smallmatrix} \right) \gamma
 \right)    d\overline{v} dw  da.
\label{eq-local-ramified-eq1}
\end{split}
\end{equation}
We conjugate $  \left(\begin{smallmatrix} a &&&\\ & I_2&&\\  &  &I_2 & \\ & &&a^*\end{smallmatrix} \right) $ to the left of $ \left(\begin{smallmatrix} I_2 &&&\\ & I_2&&\\ x & z &I_2 & \\ &x^* &&I_2\end{smallmatrix} \right) $, and make the change of variable $x\mapsto xa^{-1}$, to get
\begin{equation}
\begin{split}
    &\int\limits_{\mathrm{GL}_2(F)} \int\limits_{\mathrm{Mat}_2^0(F)}    \int\limits_{\overline{N}_{2,8}^0(F)}  l_T\left(\pi\left( \left(\begin{smallmatrix} a & \\ & a^*\end{smallmatrix}\right)  \overline{n}_4(w)  \right)  v_0 \right)   \psi^{-1}(\mathrm{tr}(Tz)) \chi_T(\det(a)) \vert \det(a)\vert^{-4} \\
 &   \omega_{\psi}\left( (1,\overline{n}_4(w) \right) \Phi_0( a+x)
f_{0, s}
 \left(  
   \left(\begin{smallmatrix} a &&&\\ & I_2&&\\  &  &I_2 & \\ & &&a^*\end{smallmatrix} \right) 
 \left(\begin{smallmatrix} I_2 &&&\\ & I_2&&\\ x & z &I_2 & \\ w&x^* &&I_2\end{smallmatrix} \right)    
  \gamma
 \right)    d\overline{v} dw  da .
 \end{split}
 \label{eq-local-ramified-eq2}
 \end{equation}
 We write the section $f_{0,s}$ in Jacquet's style notation, viewing it as a function $f_{0,s}(a;g)$ of two variables where $a\in \mathrm{GL}_4(F)$, $g\in \mathrm{Sp}_8(F)$, 
 satisfying
$$
f_{0,s}(a; \hat{m} ug)=   \vert\det(m)\vert^{s+\frac{5}{2}} f_{0,s}(am;g),
$$
where $\hat{m}=\mathrm{diag}(m, m^*)\in M_8(F), u\in N_8(F)$.
Then we get
 \begin{equation}
 \begin{split}
     &\int\limits_{\mathrm{GL}_2(F)} \int\limits_{\mathrm{Mat}_2^0(F)}    \int\limits_{\overline{N}_{2,8}^0(F)}  l_T\left(\pi\left( \left(\begin{smallmatrix} a & \\ & a^*\end{smallmatrix}\right)  \overline{n}_4(w) \right)  v_0 \right)   \psi^{-1}(\mathrm{tr}(Tz)) \chi_T(\det(a)) \vert \det(a)\vert^{s-\frac{3}{2}} \\
     &     \omega_{\psi}\left( (1,\overline{n}_4(w)\right) \Phi_0( a+x)
f_{0, s}
 \left(  
   \left(\begin{smallmatrix} a & \\ & I_2 \end{smallmatrix} \right) ;
 \left(\begin{smallmatrix} I_2 &&&\\ & I_2&&\\ x & z &I_2 & \\ w&x^* &&I_2\end{smallmatrix} \right)    
  \gamma
 \right)    d\overline{v} dw  da .  
  \label{eq-local-ramified-eq3}
\end{split}
\end{equation}
Similar to \cite[Proposition 6.6]{GinzburgRallisSoudry1998}, we choose the right $\gamma$-translate of $f_{0,s}$ to have support in $P_8(F)\cdot \mathcal{U}$, where  $\mathcal{U}\subset \mathrm{Sp}_8(F)$ is a  small neighborhood of identity, such that  $f_{0,s}(m; u \gamma)=W_{0}(m)$ for $u\in \mathcal{U}$, $m\in \mathrm{GL}_4(F)\cong M_8(F)$. Here $W_{0}$ is a function in the generalized Whittaker-Speh-Shalika model $\mathcal{W}(\tau\otimes \chi_T, 2, \psi_{2T})$, to be chosen later. We denote
\begin{equation*}
\begin{split}
\mathcal{V}_1  =\left\{ x:   \left(\begin{smallmatrix} I_2 &&&\\ & I_2&&\\ x &  &I_2 & \\ &x^* &&I_2\end{smallmatrix} \right)  \in \mathcal{U} \right\}, \mathcal{V}_2=\left\{ z:   \left(\begin{smallmatrix} I_2 &&&\\ & I_2&&\\ & z &I_2 & \\  &  &&I_2\end{smallmatrix} \right)  \in \mathcal{U} \right\},   \mathcal{V}_3  =\left\{ w:   \left(\begin{smallmatrix} I_2 &&&\\ & I_2&&\\   & &I_2 & \\ w& &&I_2\end{smallmatrix} \right)  \in \mathcal{U} \right\}.
\end{split}
\end{equation*}
With this choice of $f_{0,s}$, 
$\mathcal{V}_1$ must be a small neighborhood of zero in $\mathrm{Mat}_2(F)$, $\mathcal{V}_2$ and $\mathcal{V}_3$ must be small neighborhoods of zero in $\mathrm{Mat}^0_2(F)$.
We choose $\mathcal{U}$ so small that $\pi(\overline{n}_4(w))v_0=v_0$ and $\omega_{\psi}\left( (1,\overline{n}_4(w)\right) \Phi_0=\Phi_0$ for $w\in \mathcal{V}_3$, $\psi^{-1}(\mathrm{tr}(Tz))=1$ for $z\in \mathcal{V}_2$, and $
    \Phi_0(a+x)=\omega_{\psi}(\alpha_T(v(x, 0, 0) ))\Phi_0(a)=\Phi_0(a)
$ for $x\in \mathcal{V}_1$. Then  \eqref{eq-local-ramified-eq3} is equal to
 \begin{equation}
 \begin{split}
\mathrm{vol}(\mathcal{U}\cap \overline{N}_{8}(F))  \int\limits_{\mathrm{GL}_2(F)}   l_T\left(\pi \left(\begin{smallmatrix} a & \\ & a^*\end{smallmatrix}\right)    v_0 \right) \chi_T(\det(a)) \vert \det(a)\vert^{s-\frac{3}{2}} \Phi_0(a) W_0\left(\begin{smallmatrix} a & \\ & I_2\end{smallmatrix} \right) da.
 \end{split}
   \label{eq-local-ramified-eq4}
 \end{equation}
Note that $v_0$ is stabilized by a compact open subgroup $K_0\subset \mathrm{Sp}_4(F)$.  
We choose $M_1$ to be a small neighborhood of $I_2$ in $\mathrm{GL}_2(F)$ such that for any $a\in M_1$, $\left(\begin{smallmatrix} a & \\ & a^*\end{smallmatrix}\right)\in K_0$. We choose $M_2$ to be a small neighborhood of $I_2$ in $\mathrm{GL}_2(F)$ such that for any $a\in M_2$, $\left(\begin{smallmatrix} a & \\ & I_2\end{smallmatrix}\right)$ belongs to the stabilizer of $W_0$. We choose $M_3$ to be a small neighborhood of $I_2$ in $\mathrm{GL}_2(F)$ such that for any $a\in M_3$, $\det(a)$ belongs to the intersection of $\mathcal{O}_F^\times$ and the conductor of $\chi_T$. Let $M_0=M_1\cap M_2\cap M_3$, and let $\Phi_0$ be the characteristic function of $M_0$. Then \eqref{eq-local-ramified-eq4} is equal to
\begin{equation}
 \begin{split}
&\mathrm{vol}(\mathcal{U}\cap \overline{N}_{8}(F)) \int\limits_{M_0}   l_T\left(\pi \left(\begin{smallmatrix} a & \\ & a^*\end{smallmatrix}\right)    v_0 \right) \chi_T(\det(a)) \vert \det(a)\vert^{s-\frac{3}{2}} W_0\left(\begin{smallmatrix} a & \\ & I_2\end{smallmatrix} \right) da\\
=& l_T(v_0) \cdot \mathrm{vol}(\mathcal{U}\cap \overline{N}_{8}(F)) \mathrm{vol}(M_0) W_0(I_4).
 \end{split}
   \label{eq-local-ramified-eq5}
 \end{equation}
Setting $W_0(I_4)=\mathrm{vol}(M_0)^{-1} \mathrm{vol}(\mathcal{U} \cap \overline{N}_8(F))^{-1}$ gives the proposition.
\end{proof}

The archimedean integral is given by
\begin{equation*}
\begin{split}
\mathcal{Z}_\infty(\varphi,  \Phi_\infty, f_{\infty, s}) = \int\limits_{N_4(\mathbb{A}_\infty)\backslash \mathrm{Sp}_4(\mathbb{A}_\infty)} \int\limits_{N_{2,8}^0(\mathbb{A}_\infty)}  \varphi_{\psi, T}(h)   \omega_{\psi, \infty}\left( \alpha_T(v)(1,h)\right) )\Phi_\infty(I_2 )
f_{\infty, s}( \gamma v (1, h) )    dv   dh,
\end{split}
\end{equation*}
where $\Phi_\infty\in \mathcal{S}(\mathrm{Mat}_{2}({\mathbb A}_\infty))$, $f_{\infty, s}\in \mathrm{Ind}_{P_8({\mathbb A}_\infty)}^{\mathrm{Sp}_8({\mathbb A}_\infty)}(\mathcal{W}( \tau_\infty\otimes \chi_T, 2, \psi_{2T})\vert \det \vert^s)$.
Note that here we use the actual Fourier coefficient $\varphi_{\psi, T}$, instead of the linear functional $l_T$ as in the finite place. We have the following result.

\begin{proposition} 
Let $s_0\in \mathbb{C}$ be a complex number. There is a choice of data $(\varphi_j, \Phi_j, f_{j,s})$, such that the finite sum $\sum_j \mathcal{Z}_\infty(\varphi_j, \Phi_j, f_{j,s})$ admits meromorphic continuation to the whole complex plane and its meromorphic continuation is holomorphic and non-zero at $s_0$.
\label{proposition-ramified-archimedean}
\end{proposition}

\begin{proof} 
The proof is similar to \cite[Proposition 6.7]{GinzburgRallisSoudry1998}. Similar to \eqref{eq-local-ramified-eq2}, $\mathcal{Z}_\infty(\varphi, \Phi_\infty, f_{\infty, s})$ is equal to
\begin{equation*}
    \begin{split}
            & \int\limits_{\mathrm{GL}_2(\mathbb{A}_\infty)} \int\limits_{\mathrm{Mat}_2^0(\mathbb{A}_\infty)}    \int\limits_{\overline{N}_{2,8}^0(\mathbb{A}_\infty)} \varphi_{\psi, T}
            \left(  \left(\begin{smallmatrix} a & \\ & a^*\end{smallmatrix}\right)  \overline{n}_4(w)      \right)   \psi^{-1}(\mathrm{tr}(Tz)) \chi_T(\det(a)) \vert \det(a)\vert^{-4} \\
  &  \ \ \ \ \ \     
        \omega_{\psi, \infty}\left( (1,\overline{n}_4(w) \right) \Phi_\infty( a+x)
            f_{\infty, s}
             \left(  
             \left(\begin{smallmatrix} a &&&\\ & I_2&&\\  &  &I_2 & \\ & &&a^*\end{smallmatrix} 
             \right) 
         \left(\begin{smallmatrix} I_2 &&&\\ & I_2&&\\ x & z &I_2 & \\ w&x^* &&I_2\end{smallmatrix} \right)    \gamma  \right)    d\overline{v} dw  da .
    \end{split}
\end{equation*}
We choose $f_{\infty, s}$ such that the right $\gamma$-translate of $f_{\infty, s}$ has support in $P_8(\mathbb{A}_\infty)\cdot \overline{N}_8(\mathbb{A}_\infty)$, where
$$
\overline{N}_8(\mathbb{A}_\infty)=\left\{  \overline{u}(x, z, w)=\left(\begin{smallmatrix} I_2 &&&\\ & I_2&&\\ x & z &I_2 & \\ w&x^* &&I_2\end{smallmatrix} \right)      \in \mathrm{Sp}_8(\mathbb{A}_\infty)  \right\}.$$
Assume that 
$
\gamma\cdot f_{\infty, s} \left(  \left(\begin{smallmatrix} I_4 & * \\ & I_4\end{smallmatrix}\right) \left(\begin{smallmatrix} m &  \\ & m^*\end{smallmatrix}\right)  \overline{u}    \right)  =\vert \det(m)\vert^{s+\frac{5}{2}} \phi(\overline{u}) W_\infty (m),
$
where $\overline{u}\in \overline{N}_8(\mathbb{A}_\infty)$, $m\in \mathrm{GL}_4(\mathbb{A}_\infty)$, $W_\infty\in \mathcal{W}(\tau_\infty\otimes \chi_{T}, 2, \psi_{2T})$, and $\phi$ is a smooth function with compact support in $C_c^\infty(\overline{N}_8(\mathbb{A}_\infty))$ given by 
$
\phi(\overline{u}(x, z, w))=\phi_1(x) \phi_2(z)\phi_3(w)
$
where $\phi_1\in C_c^\infty(\mathrm{Mat}_2(\mathbb{A}_\infty)) $, $\phi_2\in C_c^\infty(\mathrm{Mat}_2^0(\mathbb{A}_\infty)) $, $\phi_3\in C_c^\infty(\mathrm{Mat}_2^0(\mathbb{A}_\infty)) $.
With  these choices,  $\mathcal{Z}_\infty(\varphi, \Phi_\infty, f_{\infty, s})$ is equal to
\begin{equation*}
\begin{split}
    &\int\limits_{\mathrm{GL}_2(\mathbb{A}_\infty)} \int\limits_{\mathrm{Mat}_2^0(\mathbb{A}_\infty)}    \int\limits_{\overline{N}_{2,8}^0(\mathbb{A}_\infty)}   \pi\left( \begin{smallmatrix} I_2  & \\ w & I_2\end{smallmatrix} \right) \varphi_{\psi, T}   \left(\begin{smallmatrix} a & \\ & a^*\end{smallmatrix}\right)      \psi^{-1}(\mathrm{tr}(Tz)) \chi_T(\det(a)) \vert \det(a)\vert^{s-\frac{3}{2}} \\
    & \ \ \ \ \ \ \ \ \ \            \omega_{\psi, \infty}\left( (1,\overline{n}_4(w) \right) \Phi_\infty( a+x)
W_{\infty}
   \left(\begin{smallmatrix} a  &\\ & I_2 \end{smallmatrix} \right) 
  \phi_1(x) \phi_2(z)\phi_3(w)  d\overline{u}  da .
 \end{split}
\end{equation*}
Here, the variable $\overline{u}=\overline{u}(x, z, w)$ depends on $x, z, w$, where $x, z$ come from coordinates of elements of $\overline{N}_{2,8}^0(\mathbb{A}_\infty)$ (see \eqref{eq-local-ramified-N_2,8,0}) and $w$ belongs to $\mathrm{Mat}_2^0(\mathbb{A}_\infty)$.
The $dz$-integration gives a constant $\int \phi_2(z)\psi^{-1}(\mathrm{tr}(Tz))dz$, and we may choose $\phi_2$ such that this constant is non-zero (see \cite[pp. 232]{GinzburgRallisSoudry1998}). Thus, up to this non-zero constant, $\mathcal{Z}_\infty(\varphi, \Phi_\infty, f_{\infty, s})$ is equal to 
\begin{equation}
\begin{split}
    &\int\limits_{\mathrm{GL}_2(\mathbb{A}_\infty)} \int\limits_{\mathrm{Mat}_2(\mathbb{A}_\infty)}\int\limits_{\mathrm{Mat}_2^0(\mathbb{A}_\infty)}    \pi\left( \begin{smallmatrix} I_2  & \\ w & I_2\end{smallmatrix} \right) \varphi_{\psi, T}   \left(\begin{smallmatrix} a & \\ & a^*\end{smallmatrix}\right)      \chi_T(\det(a)) \vert \det(a)\vert^{s-\frac{3}{2}} \\
\ \ \     & \ \ \ \       \omega_{\psi, \infty}\left( (1,\overline{n}_4(w) \right) \Phi_\infty( a+x)
W_{\infty}
   \left(\begin{smallmatrix} a  &\\ & I_2 \end{smallmatrix} \right) 
  \phi_1(x)  \phi_3(w)  dw dx  da .
 \end{split}
 \label{eq-local-archimedean-eq4}
\end{equation}
Next, we consider the $dw$-integration
\begin{equation}
\int\limits_{\mathrm{Mat}_2^0(\mathbb{A}_\infty)}   \phi_3(w)  \pi\left( \begin{smallmatrix} I_2  & \\ w & I_2\end{smallmatrix} \right)  \varphi_{\psi, T}  \otimes   \omega_{\psi, \infty}\left( (1,\overline{n}_4(w) \right) \Phi_\infty dw.
 \label{eq-local-archimedean-eq5}
\end{equation}
This is a convolution of $\phi_3$ against $\varphi_{\psi, T}\otimes \Phi_\infty \in \mathcal{W}(\pi, \psi, T)\hat{\otimes} \omega_\psi$. Here, $  \mathcal{W}(\pi, \psi, T)$ is the space spanned by the global Fourier coefficient of the form $\varphi_{\psi, T}$ and we regard it as the representation space  under the right translation action of $\overline{N}_4(\mathbb{A}_\infty)$. The symbol $\hat{\otimes}$ stands for the completed projective topological tensor product. By the Dixmier-Malliavin Theorem \cite{DixmierMalliavin1978}, a linear combination of integrals of the form \eqref{eq-local-archimedean-eq5} represents a general element in the space $\mathcal{W}(\pi, \psi, T)\hat{\otimes} \omega_\psi$. Thus, up to a suitable linear combination of integrals of the form  \eqref{eq-local-archimedean-eq4}, we get
\begin{equation}
\begin{split}
    \int\limits_{\mathrm{GL}_2(\mathbb{A}_\infty)} \int\limits_{\mathrm{Mat}_2(\mathbb{A}_\infty)} \varphi^\prime_{\psi, T}   \left(\begin{smallmatrix} a & \\ & a^*\end{smallmatrix}\right)      \chi_T(\det(a)) \vert \det(a)\vert^{s-\frac{3}{2}} \Phi_\infty^\prime ( a+x)
W_{\infty}
   \left(\begin{smallmatrix} a  &\\ & I_2 \end{smallmatrix} \right) 
  \phi_1(x)  dx  da .
 \end{split}
 \label{eq-local-archimedean-eq6}
\end{equation}
where $\varphi^\prime_{\psi, T} \in \mathcal{W}(\pi, \psi, T)$, $\Phi_\infty^\prime\in \mathcal{S}(\mathrm{Mat}_2(\mathbb{A})_\infty)$.
By the Dixmier-Malliavin Theorem \cite{DixmierMalliavin1978} again, up to a linear combination, the $dx$-integration
\begin{equation}
\int\limits_{\mathrm{Mat}_2(\mathbb{A}_\infty)}  \phi_1(x) \Phi_\infty^\prime(a+x)dx = \int\limits_{\mathrm{Mat}_2(\mathbb{A}_\infty)} \phi_1(x) \omega_\psi(\alpha_T(v(x,0,0)))\Phi_\infty^\prime(a) dx
\label{eq-local-archimedean-eq7}
\end{equation}
represents a general element of  $\mathcal{S}(\mathrm{Mat}_2(\mathbb{A}_\infty)) $. Thus, after a suitable linear combination, \eqref{eq-local-archimedean-eq6} becomes
\begin{equation}
\begin{split}
    \int\limits_{\mathrm{GL}_2(\mathbb{A}_\infty)}     \varphi^\prime_{\psi, T}   \left(\begin{smallmatrix} a & \\ & a^*\end{smallmatrix}\right)      \chi_T(\det(a)) \vert \det(a)\vert^{s-\frac{3}{2}} \Phi_\infty^{\prime \prime} ( a)
W_{\infty}
   \left(\begin{smallmatrix} a  &\\ & I_2 \end{smallmatrix} \right) 
   da 
 \end{split}
 \label{eq-local-archimedean-eq8}
\end{equation}
where $\Phi_\infty^{\prime\prime}\in \mathcal{S}(\mathrm{Mat}_2(\mathbb{A}_\infty))$.
Observe that the support of $a\mapsto W_{\infty}
   \left(\begin{smallmatrix} a  &\\ & I_2 \end{smallmatrix} \right) $ is contained in the support of a Schwartz function (see \cite[Lemma 3.13]{CaiFriedbergKaplan2024}). We may replace $W_{\infty}
   \left(\begin{smallmatrix} a  &\\ & I_2 \end{smallmatrix} \right)$ by a Schwartz function to arrive at
\begin{equation}
\begin{split}
    \int\limits_{\mathrm{GL}_2(\mathbb{A}_\infty)}     \varphi^\prime_{\psi, T}   \left(\begin{smallmatrix} a & \\ & a^*\end{smallmatrix}\right)      \chi_T(\det(a)) \vert \det(a)\vert^{s-\frac{3}{2}} \Phi_\infty^{\prime \prime \prime} ( a)
   da,
 \end{split}
 \label{eq-local-archimedean-eq9}
\end{equation}
where $\Phi_\infty^{\prime\prime\prime}$ is another Schwartz function on $\mathrm{GL}_2({\mathbb A}_\infty)$. 
Note that \eqref{eq-local-archimedean-eq9} is an archimedean integral considered in
\cite{RallisSoudry1989}. By 
\cite{RallisSoudry1989} we know that there is a choice of data so that \eqref{eq-local-archimedean-eq9} admits meromorphic continuation and \eqref{eq-local-archimedean-eq9} is non-vanishing at $s_0$. This proves the proposition.
\end{proof}

\section{Proofs of the main results}
\label{section-proofs-main-results}

\subsection{Proof of Theorem~\ref{thm-conjecture-holds}}
\label{subsection-proof-of-main-result}
Let $S$ be the following finite set of places: $\nu\not\in S$ if and only if the following conditions are satisfied:
\begin{itemize}
\item $\nu \not\vert  \, 2, 3, \infty$, 
\item $\pi_\nu$ and $\tau_\nu$ are unramified, 
\item the local component of $\psi$ at $\nu$ is unramified,
\item the diagonal coordinates of $T_0$ are in $\mathcal{O}_{F_\nu}^\times$, where $\mathcal{O}_{F_\nu}^\times$ is the group of units of the ring of integers of $F_\nu$.
\end{itemize}
Note that $S$ depends on $T$.

For any place $\nu\not\in S$, we fix in the space of $\mathrm{Ind}_{P_{16}(F_\nu)}^{\mathrm{Sp}_{16}(F_\nu)}(\mathcal{W}(\tau_\nu, 4, \psi^{-1})\vert \det \vert^s)$, the spherical vector $f_{\mathcal{W}(\tau_\nu, 4, \psi^{-1}), s}^0$, which is normalized, such that, for $a\in \mathrm{GL}_8(F_\nu)$, we have
$
    f_{\mathcal{W}(\tau, 4, \psi^{-1}), s}^0(a;I_{16})=W^0_{\Delta(\tau_\nu, 4)}(a),
$
where $W^0_{\Delta(\tau_\nu, 4)}$ is the unique unramified function in the unique model $\mathcal{W}(\tau_\nu, 4, \psi^{-1})$, such that its value at $I_8$ is 1. 
Recall that $f_{\mathcal{W}(\tau, 4, \psi^{-1}), s}^0(a;g)$ is a function of two variables where 
$a\in \mathrm{GL}_8(F_{\nu})$, $g\in \mathrm{Sp}_{16}(F_{\nu})$, 
 satisfying
$$
f_{\mathcal{W}(\tau, 4, \psi^{-1}), s}^0(a; \hat{m} ug)=   \vert\det(m)\vert^{s+\frac{9}{2}} f_{\mathcal{W}(\tau, 4, \psi^{-1}), s}^0(am;g),
$$
where $\hat{m}=\mathrm{diag}(m, m^*)\in M_{16}(F_{\nu}), u\in N_{16}(F_{\nu})$.
To ease notation, we also denote $f_{\mathcal{W}(\tau_\nu, 4, \psi^{-1}), s}^0(g)=f_{\mathcal{W}(\tau_\nu, 4, \psi^{-1}), s}^0(I_8;g)$ for $g\in \mathrm{Sp}_{16}(F_\nu)$.
Let us fix an isomorphism
$
    \Delta(\tau, 4)\cong \otimes_\nu^\prime \Delta(\tau_\nu, 4)
$
with respect to a system $\{\xi_\nu^0 \}_{\nu\not\in S}$ of spherical vectors. For $\nu\not\in S$, $\Delta(\tau_\nu, 4)$ is an unramified representation of $\mathrm{GL}_8(F_\nu)$, and the vector $\xi_\nu^0$ is a spherical vector in the space of $\Delta(\tau_\nu, 4)$. We further fix a non-zero functional $\Lambda_\nu\in \mathrm{Hom}_{U_{(4^2)}(F_\nu)}(\Delta(\tau_\nu, 4), \psi^{-1})$ such that $W^0_{\Delta(\tau_\nu, 4)}(I_8)=\Lambda_\nu(\xi_\nu^0)=1$. Recall that $U_{(4^2)}(F_\nu)=\left\{ \left(\begin{smallmatrix}I_4 & u_{1,2}\\ & I_4\end{smallmatrix}\right): u_{1,2}\in \mathrm{Mat}_4(F_\nu) \right\}$.

Recall that
\begin{equation*}
    d_{\tau_\nu}^{\mathrm{Sp}_{16}}(s)=L( s+\frac{5}{2}, \tau_\nu)     \prod_{1\le j \le 2} L(2s+2j, \chi_{\tau_\nu})L(2s+2j-1, \tau_\nu, \mathrm{Sym}^2)
\end{equation*}
and
\begin{equation*}
    d_{\tau_\nu\otimes \chi_T}^{\mathrm{Sp}_{8}}(s)=L(s+\frac{3}{2}, \tau_\nu\otimes\chi_T)L(2s+2, \chi_{\tau_\nu}) L(2s+1, \tau_\nu\otimes\chi_T, \mathrm{Sym}^2).
\end{equation*}
For $\nu\not\in S$, let
\begin{equation*}
    f_{\mathcal{W}(\tau_\nu, 4, \psi^{-1}), s}^*(g)= d_{\tau_\nu}^{\mathrm{Sp}_{16}}(s) f_{\mathcal{W}(\tau_\nu, 4, \psi^{-1}), s}^0(g).
\end{equation*}
For $\nu\not\in S$, we consider the induced representation $\mathrm{Ind}_{P_8(F_\nu)}^{\mathrm{Sp}_8(F_\nu)}(\mathcal{W}( \tau_\nu\otimes \chi_T, 2, \psi_{2T})\vert \det \vert^s)$, and let  $f^0_{\mathcal{W}(\tau_\nu\otimes \chi_T, 2, \psi_{2T}), s}$ be an unramified section in the space of $\mathrm{Ind}_{P_8(F_\nu)}^{\mathrm{Sp}_8(F_\nu)}(\mathcal{W}( \tau_\nu\otimes \chi_T, 2, \psi_{2T})\vert \det \vert^s)$, 
such
that, when viewed as a complex-valued function, its value at the identity
is equal to 
$
    \frac{d_{\tau_\nu}^{\mathrm{Sp}_{16}}(s)}{d_{\tau_\nu\otimes\chi_T}^{\mathrm{Sp}_{8}}(s)}.
$
This is a meromorphic function in $s$.
Denote, for $\nu\not\in S$, \begin{equation*}    f^*_{\mathcal{W}(\tau_\nu\otimes \chi_{T}, 2, \psi_{2T}), s}(g_\nu)= d_{\tau_\nu\otimes\chi_T}^{\mathrm{Sp}_8}(s) f^0_{\mathcal{W}(\tau_\nu\otimes \chi_{T}, 2, \psi_{2T}), s}(g_\nu),\end{equation*} and, for $\nu \in S$, \begin{equation*}    f^*_{\mathcal{W}(\tau_\nu\otimes \chi_{T}, 2, \psi_{2T}), s}(g_\nu)=  f_{\mathcal{W}(\tau_\nu\otimes \chi_{T}, 2, \psi_{2T}), s}(g_\nu).\end{equation*}

For  a finite set  $\Omega$  of places, we denote
\begin{equation}
\label{eq-zeta_omega}
\begin{split}
    \mathcal{Z}_\Omega(\varphi, &\Phi, f^*_{\mathcal{W}(\tau\otimes \chi_{T}, 2, \psi_{2T}), s})= 
     \int\limits_{ N_4(\mathbb{A}_\Omega)  \backslash  \mathrm{Sp}_4(\mathbb{A}_\Omega)}     \int\limits_{N_{2,8}^0(\mathbb{A}_\Omega)} \\
     &\varphi_{\psi, T}(h) \omega_{\psi, \Omega}\left( \alpha_T(v)(1,h)\right) )\Phi_\Omega(I_2 )     f^*_{\mathcal{W}(\tau_\Omega\otimes \chi_{T}, 2, \psi_{2T}), s}(\gamma v (1, h) )      dv dh .
\end{split}
\end{equation}
Here, for an algebraic group $H$,  we write $H({\mathbb A}_\Omega)=\prod_{\nu\in \Omega}H(F_\nu)$. Also, $\Phi_\Omega$ and $f^*_{\mathcal{W}(\tau_\Omega\otimes \chi_{T}, 2, \psi_{2T}), s}$ respectively denote the product of the factors of  $\Phi$ and $f^*_{\mathcal{W}(\tau\otimes \chi_T, 2, \psi_{2T}), s}$ at the places in $\Omega$.

We have the following result.

\begin{theorem}
Fix an isomorphism $\pi\cong \otimes_\nu^\prime \pi_\nu$. Assume $\varphi\in V_\pi$ corresponds to a pure tensor $\otimes_\nu \xi_\nu$ with $\xi_\nu\in V_{\pi_\nu}$.
For $\mathrm{Re}(s)\gg 0$, we have
\begin{equation*}
    \mathcal{Z}( \varphi,   \theta_{\psi}^\Phi   , E^{*, S}(\cdot, f_{\Delta(\tau\otimes \chi_T, 2), s})) =  L^S(s+\frac{1}{2}, \pi\times \tau)  \cdot \mathcal{Z}_S(\varphi, \Phi, f^*_{\mathcal{W}(\tau\otimes \chi_{T}, 2, \psi_{2T}), s}) .
\end{equation*}
\label{theorem-global-1}
\end{theorem}

\begin{proof}
The proof is similar to \cite[pp. 117-118]{Piatetski-ShapiroRallis1988}, \cite[Theorem 1.2]{BumpFurusawaGinzburg1995} and \cite[Corollary 3.4]{PollackShah2017}. We prove some details for completeness.
Let $\Omega$ be a finite set of places containing $S$. By the definition of the adelic integral we have
\begin{equation*}
    \mathcal{Z}( \varphi,   \theta_{\psi}^\Phi   , E^{*, S}(\cdot, f_{\Delta(\tau\otimes \chi_T, 2), s}))= \varinjlim_{\Omega\supseteq S} \mathcal{Z}_\Omega(\varphi, \Phi, f^*_{\mathcal{W}(\tau\otimes \chi_{T}, 2, \psi_{2T}), s}).
\end{equation*}
If $\nu$ is a place not in $\Omega$, then 
\begin{multline}
     \mathcal{Z}_{\Omega\cup\{\nu\}}(\varphi, \Phi, f^*_{\mathcal{W}(\tau\otimes \chi_{T}, 2, \psi_{2T}), s})\\
 =  \int\limits_{ N_4(\mathbb{A}_{\Omega\cup\{\nu\}})  \backslash  \mathrm{Sp}_4(\mathbb{A}_{\Omega\cup\{\nu\}})}     \int\limits_{N_{2,8}^0(\mathbb{A}_{\Omega\cup\{\nu\}})}      \varphi_{\psi, T}(h) \omega_{\psi, {\Omega\cup\{\nu\}}}\left( \alpha_T(v)(1,h)\right) )  \Phi_{\Omega\cup\{\nu\}}(I_2 )  \\
     f^*_{\mathcal{W}(\tau_{\Omega\cup\{\nu\}}\otimes \chi_{T}, 2, \psi_{2T}), s}(\gamma v (1, h) )      dv dh \\
 =    \int\limits_{ N_4(\mathbb{A}_\Omega)  \backslash  \mathrm{Sp}_4(\mathbb{A}_\Omega)}     \int\limits_{N_{2,8}^0(\mathbb{A}_\Omega)}       \omega_{\psi, \Omega}\left( \alpha_T(v)(1,h)\right) )\Phi_\Omega(I_2 )     f^*_{\mathcal{W}(\tau_\Omega\otimes \chi_{T}, 2, \psi_{2T}), s}(\gamma v (1, h) )      dv  \\
 \times    \int\limits_{ N_4(F_\nu)  \backslash  \mathrm{Sp}_4(F_\nu)}  \int\limits_{N_{2,8}^0(F_\nu)}   \varphi_{\psi, T}(h_\nu h) \omega_{\psi, \nu}\left( \alpha_T(v_\nu)(1,h_\nu)\right) )\Phi_\nu(I_2 )    \\
  f^*_{\mathcal{W}(\tau_\nu\otimes \chi_{T}, 2, \psi_{2T}), s}(\gamma v_\nu (1, h_\nu) ) dv_\nu dh_\nu   dh
\label{eq-Main-Theorems-Basic-Lemma-argument1}
\end{multline}
since $\omega_\psi$, $\Phi$, and $ f^*_{\mathcal{W}(\tau \otimes \chi_T, 2, \psi_{2T}), s}$ are factorizable. 

Now we fix $h\in \mathrm{Sp}_4(\mathbb{A}_\Omega)$. Since $\pi_{\nu}$ is unramified (recall $\nu$ is not in $S$), there is a non-zero unramified vector $v_0$ in $V_{\pi_\nu}$. Let $i_\nu:V_{\pi_\nu}\to V_\pi$ be a
$\mathrm{Sp}_4(F_\nu)$-equivariant map that sends $v_0$ to the cusp form $\varphi$. For example, we may take
\begin{equation*}
    i_\nu(v_\nu)=\bigotimes_{w\not=\nu}\xi_w\otimes v_\nu.
\end{equation*}
Consider the linear functional $l_{T}:V_{\pi_\nu}\to \mathbb{C}$ given by
\begin{equation*}
    l_{T}(v_\nu)=   \int\limits_{\mathrm{Mat}_2^0(F)\backslash \mathrm{Mat}_2^0(\mathbb{A})} (\pi(h)i_\nu(v_\nu) )\left(\begin{smallmatrix}
I_2 & z\\
& I_2
\end{smallmatrix} \right) \psi(\mathrm{tr}(Tz))dz.
\end{equation*}
Note that $l_{T}(v_0)=\varphi_{\psi, T}(h)$. Since $h_\nu\in \mathrm{Sp}_4(F_\nu)$ commutes with $h\in \mathrm{Sp}_4(\mathbb{A}_\Omega)$,
$l_{T}$ satisfies \eqref{eq-l_T}. Moreover, we have
\begin{equation*}
    \varphi_{\psi, T}(h_\nu h)=l_{T}(\pi_\nu(h_\nu)v_0).
\end{equation*}
Now we use \eqref{eq-Main-Theorems-Basic-Lemma-argument1} and Theorem~\ref{theorem-unramified-computation} to obtain
\begin{equation*}
\begin{split}
    & \mathcal{Z}_{\Omega\cup\{\nu\}}(\varphi, \Phi, f^*_{\mathcal{W}(\tau\otimes \chi_{T}, 2, \psi_{2T}), s})\\
=&   \int\limits_{ N_4(\mathbb{A}_\Omega)  \backslash  \mathrm{Sp}_4(\mathbb{A}_\Omega)}     \int\limits_{N_{2,8}^0(\mathbb{A}_\Omega)}       \omega_{\psi, \Omega}\left( \alpha_T(v)(1,h)\right) )\Phi_\Omega(I_2 )     f^*_{\mathcal{W}(\tau_\Omega\otimes \chi_{T}, 2, \psi_{2T}), s}(\gamma v (1, h) )      dv  \\
&\hspace{0.4in} \times    \int\limits_{ N_4(F_\nu)  \backslash  \mathrm{Sp}_4(F_\nu)}  \int\limits_{N_{2,8}^0(F_\nu)}  l_{T}(\pi_\nu(h_\nu) v_0 ) \omega_{\psi, \nu}\left( \alpha_T(v_\nu)(1,h_\nu)\right) )\Phi_\nu(I_2 )    \\ 
&\hspace{1.8in}  f^*_{\mathcal{W}(\tau_\nu\otimes \chi_{T}, 2, \psi_{2T}), s}(\gamma v_\nu (1, h_\nu) ) dv_\nu dh_\nu dh \\
=&   \int\limits_{ N_4(\mathbb{A}_\Omega)  \backslash  \mathrm{Sp}_4(\mathbb{A}_\Omega)}     \int\limits_{N_{2,8}^0(\mathbb{A}_\Omega)}       \omega_{\psi, \Omega}\left( \alpha_T(v)(1,h)\right) )\Phi_\Omega(I_2 )     f^*_{\mathcal{W}(\tau_\Omega\otimes \chi_{T}, 2, \psi_{2T}), s}(\gamma v (1, h) )      dv \\
& \hspace{3in}   L(s+\frac{1}{2},  {\pi_\nu}\times {\tau_\nu}) \cdot  \varphi_{\psi, T}(h) dh \\
=&  L(s+\frac{1}{2},  {\pi_\nu}\times {\tau_\nu})  \mathcal{Z}_\Omega(\varphi, \Phi, f^*_{\mathcal{W}(\tau\otimes \chi_T, 2, \psi_{2T}), s}).
\end{split}
\end{equation*}
Finally, we take the limit to obtain the result. 
\end{proof}

We now restate and prove Theorem~\ref{thm-conjecture-holds}. 
\begin{theorem}
Given an irreducible automorphic cuspidal representation $\pi$ of $\mathrm{Sp}_4({\mathbb A})$, 
an irreducible unitary automorphic cuspidal representation $\tau$ of $\mathrm{GL}_2({\mathbb A})$, 
and a non-zero cusp form $\varphi\in V_\pi$ which corresponds to $\otimes_\nu \xi_\nu$ under the factorization $\pi\cong \otimes_\nu^\prime \pi_\nu$, there is a choice of a matrix $T=J_2 T_0\in \mathrm{GL}_2(F)$ where $T_0$ is symmetric, a section $f_{\Delta(\tau\otimes \chi_T, 2), s}\in \mathrm{Ind}_{P_8(\mathbb{A})}^{\mathrm{Sp}_8(\mathbb{A})}(\Delta(\tau\otimes \chi_T, 2)\vert \det \cdot \vert^s)$ and some $\Phi=\prod_{\nu}\Phi_\nu \in \mathcal{S}(\mathrm{Mat}_2({\mathbb A}))$ so that
\begin{equation*}
\mathcal{Z}( \varphi,   \theta_{\psi}^\Phi   , E^{*, S}(\cdot, f_{\Delta(\tau\otimes \chi_T, 2), s})) = 
L^S(s+\frac{1}{2}, \pi\times \tau)  \cdot \mathcal{Z}_S(\varphi, \Phi, f^*_{\mathcal{W}(\tau\otimes \chi_{T}, 2, \psi_{2T}), s}) ,
\end{equation*}
where 
$\mathcal{Z}_S(\varphi, \Phi, f^*_{\mathcal{W}(\tau\otimes \chi_{T}, 2, \psi_{2T}), s})$ 
is a meromorphic function in $s$.
Moreover, for any $s_0\in \mathbb{C}$, the data can be chosen so that $\mathcal{Z}_S(\varphi, \Phi, f^*_{\mathcal{W}(\tau\otimes \chi_{T}, 2, \psi_{2T}), s})$ is holomorphic and
non-zero in a neighborhood of $s=s_0$. 
\label{thm-conjecture-holds-restate}
\end{theorem}

\begin{proof}
This follows from Theorem~\ref{theorem-global-1}, Proposition~\ref{proposition-ramified-finiteplace} and Proposition~\ref{proposition-ramified-archimedean}.
Indeed, we choose $T$ such that $\varphi_{\psi, T}\not=0$. 
Let $S$ be the finite set of places as in Subsection~\ref{subsection-proof-of-main-result}. Let $\Phi=\prod_{\nu}\Phi_\nu$, $f_{\Delta(\tau\otimes \chi_T, 2), s} = \prod_{\nu} f_{\Delta(\tau_\nu\otimes \chi_T, 2), s}$ be factorizable. 
By Theorem~\ref{theorem-global-1}, we have
\begin{equation*}
    \mathcal{Z}( \varphi,   \theta_{\psi}^\Phi   , E^{*, S}(\cdot, f_{\Delta(\tau\otimes \chi_T, 2), s})) =  L^S(s+\frac{1}{2}, \pi\times \tau)  \cdot \mathcal{Z}_S(\varphi, \Phi, f^*_{\mathcal{W}(\tau\otimes \chi_{T}, 2, \psi_{2T}), s}) .
\end{equation*}
By Proposition~\ref{proposition-ramified-finiteplace} and Proposition~\ref{proposition-ramified-archimedean}, we can choose data so that the local integral at a finite place $\nu\in S$ is equal to $l_T(v_0)$, where $v_0$ is a vector stabilized by an open compact subgroup of $\mathrm{Sp}_4(\mathcal{O}_{F_\nu})$,
and the local integral at an archimedean place $\nu\in S$ admits meromorphic continuation to the whole complex plane which is holomorphic and non-zero at $s_0$. Hence $\mathcal{Z}_S(\varphi, \Phi, f^*_{\mathcal{W}(\tau\otimes \chi_{T}, 2, \psi_{2T}), s})$ admits meromorphic continuation to the whole complex plane which is holomorphic and can be made non-zero at $s_0$. This proves Theorem~\ref{thm-conjecture-holds-restate}.
\end{proof}

\subsection{Proofs of Theorem~\ref{thm-nonvanishing-of-period} and Theorem~\ref{thm-holomorphy-of-L(s,pixtau)}}
\label{subsection-applications}
In this subsection we prove Theorem~\ref{thm-nonvanishing-of-period} and Theorem~\ref{thm-holomorphy-of-L(s,pixtau)}.

We recall the following result of Jiang, Liu and Zhang on the possible location of the poles of the fully normalized Eisenstein series $E^*(\cdot, f_{\Delta(\tau\otimes \chi_T, 2), s})$, defined in \eqref{eq-normalized-Eisenstein}.

\begin{theorem}(Jiang-Liu-Zhang, \cite[Theorem 5.2]{JiangLiuZhang2013})
\label{theorem-JiangLiuZhang2013-Poles}
Let $\tau$ be an irreducible unitary automorphic cuspidal self-dual representation of ${\mathrm{GL}}_2({\mathbb{A}})$.
Then the normalized Eisenstein series $E^*(\cdot, f_{\Delta(\tau\otimes \chi_T, 2), s})$ at worst has simple poles at $s=-1$, $s=-\frac{1}{2}$, $s=\frac{1}{2}$, and $s=1$.
More specifically, for $\mathrm{Re}(s)\ge 0$, we have the following.\\
\begin{enumerate}
    \item[(i)] If the central character $\chi_\tau$ of $\tau$ is trivial, and $L(\frac{1}{2}, \tau\otimes\chi_T)\not=0$, then  $E^*(\cdot, f_{\Delta(\tau\otimes \chi_T, 2), s})$ has a simple pole at $1$.\\
    \item[(ii)] If the central character  $\chi_\tau$ of $\tau$ is trivial, and $L(\frac{1}{2}, \tau\otimes\chi_T)=0$, then   $E^*(\cdot, f_{\Delta(\tau\otimes \chi_T, 2), s})$ is holomorphic.\\
    \item[(iii)] If $L(s, \tau\otimes \chi_T, \mathrm{Sym}^2)$ has a pole at $s=1$, then
     $E^*(\cdot, f_{\Delta(\tau\otimes \chi_T, 2), s})$ has a simple pole at $\frac{1}{2}$.
\end{enumerate}
\end{theorem}

Now we prove Theorem~\ref{thm-nonvanishing-of-period} and Theorem~\ref{thm-holomorphy-of-L(s,pixtau)}.

\begin{proof}[Proof of Theorem~\ref{thm-nonvanishing-of-period}]
(i) Let $\mathcal{Z}( \varphi,   \theta_{\psi}^\Phi   , E^*(\cdot, f_{\Delta(\tau\otimes \chi_T, 2), s}))$ be the integral \eqref{eq-global-integral-0} with the Eisenstein series $E^{*,S}(\cdot, f_{\Delta(\tau\otimes \chi_T, 2), s})$ replaced by fully normalized Eisenstein series $E^*(\cdot, f_{\Delta(\tau\otimes \chi_T, 2), s})$.
By Theorem~\ref{thm-conjecture-holds}, we obtain that
\begin{equation}
\label{eq-Applications-integral-representation-formula}
\mathcal{Z}( \varphi,   \theta_{\psi}^\Phi   , E^*(\cdot, f_{\Delta(\tau\otimes \chi_T, 2), s}))= 
L^S(s+\frac{1}{2}, \pi\times \tau) \cdot \mathcal{Z}_S(\varphi, \Phi, f_{\mathcal{W}(\tau\otimes \chi_{T}, 2, \psi_{2T}), s}) \cdot  \prod_{\nu\in S} d_{\tau_\nu\otimes\chi_T}^{\mathrm{Sp}_8}(s)   ,
\end{equation}
and $\mathcal{Z}_S(\varphi, \Phi, f_{\mathcal{W}(\tau\otimes \chi_{T}, 2, \psi_{2T}), s}$ can be made non-vanishing for any $s$. 

Now we show that $\prod_{\nu\in S} d_{\tau_\nu\otimes\chi_T}^{\mathrm{Sp}_8}(s)$ is non-vanishing for any $s$. Recall that
\begin{equation*}
    d_{\tau_\nu\otimes\chi_T}^{{\mathrm{Sp}}_8}(s)=L (s+\frac{3}{2}, \tau_\nu\otimes\chi_T)L (2s+2, \chi_{\tau_\nu}) L (2s+1, \tau_\nu\otimes\chi_T, \mathrm{Sym}^2).
\end{equation*}
We divide into two cases: $\nu\in S$ is a finite place (Case I), and $\nu\in S$ is archimedean (Case II).

Case I: $\nu\in S$ is a finite place. The local $L$-function $L(s, \tau_\nu\otimes\chi_T)$ is defined as $L(s, \tau_\nu\otimes\chi_T)=\frac{1}{P_1(q^{-s})}$ for certain polynomial $P_1(X)\in \mathbb{C}[X]$ with $P_1(0)=1$, so $L (s+\frac{3}{2}, \tau_\nu\otimes\chi_T)$ is non-vanishing for any $s$.
Similarly, $L (2s+2, \chi_{\tau_\nu})$ and $L (2s+1, \tau_\nu\otimes\chi_T, \mathrm{Sym}^2)$ are non-vanishing for any $s$. We conclude that, when $\nu\in S$ is finite, we have $d_{\tau_\nu\otimes\chi_T}^{{\mathrm{Sp}}_8}(s)\not=0$ for any $s$.

Case II: $\nu\in S$ is an archimedean place. In this case, all the $L$-functions $L (s, \tau_\nu\otimes\chi_T)$, $L (s, \tau_\nu\otimes\chi_T, \mathrm{Ext}^2)$, $L (s, \tau_\nu\otimes\chi_T, \mathrm{Sym}^2)$ are given by a product of Gamma functions $\Gamma(s)$ up to a scalar (see, for example, \cite{Knapp1994motives}). Since $\Gamma(s)$ is never zero, it follows that $d_{\tau_\nu\otimes\chi_T}^{{\mathrm{Sp}}_8}(s)\not=0$ for any $s$. We conclude that $\prod_{\nu\in S} d_{\tau_\nu\otimes\chi_T}^{\mathrm{Sp}_8}(s)$ is non-vanishing for any $s$

Next  we use \eqref{eq-Applications-integral-representation-formula} with $s=1$, and choose data so that $\mathcal{Z}_S(\varphi, \Phi, f_{\mathcal{W}(\tau\otimes\chi_T, 2, \psi_{2T}, s})$ does not vanish at $s=1$. By the above analysis on the non-vanishing of the local $L$-functions, we see that $\prod_{\nu\in S} d_{\tau_\nu\otimes\chi_T}^{\mathrm{Sp}_8}(1)\not=0$.
Since
$
    \mathrm{Res}_{s=\frac{3}{2}} L^S(s, \pi\times \tau)\not=0,
$
it follows from \eqref{eq-Applications-integral-representation-formula} that $\mathcal{Z}( \varphi,   \theta_{\psi}^\Phi   , E^*(\cdot, f_{\Delta(\tau\otimes \chi_T, 2), s}))$ must have a simple pole at $s=1$. This means that 
there exists a cusp form $\varphi\in V_\pi$, a Schwartz function $\Phi\in \mathcal{S}(\mathrm{Mat}_2({\mathbb{A}}))$, and a residue  $R\in \mathcal{R}(1,\Delta(\tau\otimes \chi, 2))$, such that the period integral
\begin{equation*}
\int\limits_{\mathrm{Sp}_4(F)\backslash \mathrm{Sp}_4(\mathbb{A})} \int\limits_{N_{2, 8}(F)\backslash N_{2, 8}(\mathbb{A})}   \varphi(h)     \theta_{\psi}^\Phi (\alpha_T(v)(1, h)) R(v(1,h)) dvdh \not=0.
\end{equation*}

(ii) The proof is similar to the proof of (i). We use \eqref{eq-Applications-integral-representation-formula} with $s=\frac{1}{2}$, and choose data so that $\mathcal{Z}_S(\varphi, \Phi, f_{\mathcal{W}(\tau\otimes\chi_T, 2, \psi_{2T}, s})$ does not vanish at $s=\frac{1}{2}$. Since 
$
    \mathrm{Res}_{s=1} L^S(s, \pi\times \tau)\not=0,
$
we deduce that $\mathcal{Z}( \varphi,   \theta_{\psi}^\Phi   , E^*(\cdot, f_{\Delta(\tau\otimes \chi_T, 2), s}))$ must have a simple pole at $s=\frac{1}{2}$. This means that 
there exists a cusp form $\varphi\in V_\pi$, a Schwartz function $\Phi\in \mathcal{S}(\mathrm{Mat}_2({\mathbb{A}}))$, and a residue  $R\in \mathcal{R}(\frac{1}{2},\Delta(\tau\otimes \chi, 2))$, such that the period integral
\begin{equation*}
\int\limits_{\mathrm{Sp}_4(F)\backslash \mathrm{Sp}_4(\mathbb{A})} \int\limits_{N_{2, 8}(F)\backslash N_{2, 8}(\mathbb{A})}   \varphi(h)     \theta_{\psi}^\Phi (\alpha_T(v)(1, h)) R(v(1,h)) dvdh \not=0.
\end{equation*}
\end{proof}

\begin{proof}[Proof of Theorem~\ref{thm-holomorphy-of-L(s,pixtau)}]
To prove Theorem~\ref{thm-holomorphy-of-L(s,pixtau)}, we use \eqref{eq-Applications-integral-representation-formula} again. Since $\tau$ is self-dual, and by assumption $L(\frac{1}{2}, \tau\otimes\chi_T)=0$ and the central character $\chi_\tau$ of $\tau$ is trivial, we conclude that $E^*(\cdot, f_{\Delta(\tau\otimes \chi_T, 2), s})$ is holomorphic by Theorem~\ref{theorem-JiangLiuZhang2013-Poles}. Thus $\mathcal{Z}( \varphi,   \theta_{\psi}^\Phi   , E^*(\cdot, f_{\Delta(\tau\otimes \chi_T, 2), s}))$ is holomorphic. Since $\mathcal{Z}_S(\varphi, \Phi, f_{\mathcal{W}(\tau\otimes \chi_{T}, 2, \psi_{2T}), s}$ can be made non-vanishing for any $s$, and $\prod_{\nu\in S} d_{\tau_\nu\otimes\chi_T}^{\mathrm{Sp}_8}(s)$ is non-vanishing for any $s$, it follows from \eqref{eq-Applications-integral-representation-formula} that $L^S(s+\frac{1}{2}, \pi\times \tau)$ has no poles. Thus $L^S(s, \pi\times \tau)$ is holomorphic.
\end{proof}

\section*{Acknowledgements}
This work is my doctoral dissertation at the Ohio State University. I am very grateful to my advisor Jim Cogdell for providing a tremendous amount of advice, support and encouragement. I would like to thank Yuanqing Cai for helpful discussions, and to Yubo Jin for pointing out an inaccuracy in an earlier version. Finally I would like to thank the anonymous referee for a thorough review of the paper, and for many helpful comments and suggestions that improved the paper.

\bibliographystyle{alpha}
\bibliography{References}

\end{document}